\documentclass[12pt]{amsart}

\usepackage[text={425pt,650pt},centering]{geometry}
\usepackage{color}
\usepackage{esint,amssymb}
\usepackage{graphicx}
\usepackage{tikz}
\usetikzlibrary{calc} 
\usepackage{mathtools}

\newtheorem{theorem}{Theorem}
\newtheorem{proposition}[theorem]{Proposition}
\newtheorem{lemma}[theorem]{Lemma}
\newtheorem{corollary}[theorem]{Corollary}

\theoremstyle{definition}
\newtheorem{remark}[theorem]{Remark}

\newcommand{\cref}[1]{Corollary~\ref{c.#1}}

\numberwithin{equation}{section}
\numberwithin{theorem}{section}

\newcommand{\N}{\mathbb{N}}
\newcommand{\R}{\mathbb{R}}

\renewcommand{\l}{\Omega}
\newcommand{\E}{\mathbb{E}}
\renewcommand{\P}{\mathbb{P}}
\newcommand{\F}{\mathcal{F}}

\newcommand{\Zd}{\mathbb{Z}^d}
\newcommand{\Rd}{\mathbb{R}^d}

\newcommand{\ep}{\varepsilon}

\DeclareMathOperator{\dist}{dist}

\DeclareMathOperator*{\osc}{osc}
\DeclareMathOperator{\var}{var}

\DeclareMathOperator{\supp}{supp}
\DeclareMathOperator{\divg}{div}

\newcommand{\X}{\mathcal{X}}  
\newcommand{\Y}{\mathcal{Y}}

\renewcommand{\tilde}{\widetilde}
\renewcommand{\div}{\divg}

\begin{document}

\title[Stochastic homogenization of convex integral functionals]{Quantitative stochastic homogenization of convex integral functionals}

\begin{abstract}
We present quantitative results for the homogenization of uniformly convex integral functionals with random coefficients under independence assumptions. The main result is an error estimate for the Dirichlet problem which is algebraic (but sub-optimal) in the size of the error, but optimal in stochastic integrability. As an application, we obtain quenched $C^{0,1}$ estimates for local minimizers of such energy functionals.  
\end{abstract}

\author[S. N. Armstrong]{Scott N. Armstrong}
\address{Ceremade (UMR CNRS 7534), Universit\'e Paris-Dauphine, Paris, France}
\email{armstrong@ceremade.dauphine.fr}

\author[C. K. Smart]{Charles K. Smart}
\address{Department of Mathematics, Massachusetts Institute of Technology, Cambridge, MA 02139}
\email{smart@math.mit.edu}

\date{\today}

\maketitle

\section{Introduction}


\subsection{Informal summary of results}

We consider stochastic homogenization of the variational problem
\begin{equation} \label{e.min}
\mbox{minimize} \quad \int_{U} L\left(Du(x),\frac x\ep \right) \,dx \qquad \mbox{subject to} \qquad u \in g + H^1_0(U).
\end{equation}
Here $0<\ep \ll 1$ is a small parameter, $U\subseteq\Rd$ is a smooth bounded domain and $g\in H^1(U)$ is given. The precise hypotheses on the Lagrangian $L$ are given below; here we mention that $L(p,x)$ is uniformly convex in $p$ and that $L$ is a random field sampled by a given probability measure $\P$. The crucial hypothesis on the statistics of $L$ is a \emph{finite range of dependence} condition: roughly, for all Borel sets $U,V \subseteq \Rd$, the families $\{ L(p,x)\,:\, p\in\Rd,\, x\in U \}$ and $\{ L(p,x)\,:\, p\in\Rd,\, x\in V \}$ of random variables are assumed to be $\P$--independent provided that $\dist(U,V) \geq 1$.

\smallskip

An important special case of the model occurs if the Lagrangian is the quadratic form $L(p,x)=p\cdot A(x)p$. The corresponding Euler-Lagrange equation is then linear and the problem is equivalent to the stochastic homogenization of the equation
\begin{equation} \label{e.lin}
-\div\left( A\left( \frac x\ep \right) Du^\ep \right) = 0.
\end{equation}
This is also a continuum version of what is known in the probability literature as the \emph{random conductance model}. 

\smallskip

Dal Maso and Modica~\cite{DM1,DM2} proved, in a somewhat more general setting, the basic \emph{qualitative homogenization} result for~\eqref{e.min}: there exists a (deterministic) function $\overline L:\Rd\to \R$ called the \emph{effective Lagrangian} such that, with probability one, the unique minimizer $u^\ep$ of~\eqref{e.min} converges, as $\ep \to 0$, to the unique minimizer of the variational problem 
\begin{equation} \label{e.min2}
\mbox{minimize} \quad \int_{U} \overline L\left(Du(x) \right) \,dx \qquad \mbox{subject to} \qquad u \in g + H^1_0(U).
\end{equation}
This result was a generalization to the nonlinear setting of earlier qualitative results for linear elliptic partial differential equations in divergence form due to~Kozlov~\cite{K1}, Papanicolaou and Varadhan~\cite{PV1} and~Yurinkskii~\cite{Y}, using new variational ideas based on subadditivity that were not present in earlier works.

\smallskip

An intense focus has recently emerged on building a quantitative theory of stochastic homogenization in the case of the linear equation~\eqref{e.lin}. This escalated significantly with the work of Gloria and Otto~\cite{GO1,GO2}, who proved optimal quantitative bounds for the energy density of modified correctors and then that of Gloria Neukamm and Otto~\cite{GNO,GNO2}, who proved optimal bounds for the error in homogenization. These results were proved for discrete elliptic equations, but have been extended to the continuum setting in~\cite{GO3}. See also Mourrat~\cite{M,M2}, Marahrens and Otto~\cite{MO}, Conlon and Spencer~\cite{CS} as well as earlier works of Yurinskii~\cite{Y}, Naddaf and Spencer~\cite{NS}, Bourgeat and Piatnitski~\cite{BP} and Boivin~\cite{B}. For some recent work on limit theorems for the stochastic fluctuations, see~\cite{MoO,MN,N,R,BSW}. The analysis in the present paper was informed by some ideas from our previous work~\cite{AS}, which contained similar results for equations in nondivergence form.

\smallskip

In this paper, we present the first \emph{quantitative} results for the homogenization of~\eqref{e.min} which are also the first such results for divergence-form elliptic equations outside of the linear setting. We prove two main results: estimates for the~$L^2$ and~$L^\infty$ error in homogenization of the Dirichlet problem, which is algebraic (yet sub-optimal) in its estimate of the typical size of the error, and essentially optimal in stochastic integrability; and a ``stochastic higher regularity" result which states that local minimizers of~\eqref{e.min}, for a typical realization of the coefficients, satisfy the same \emph{a~priori} $C^{0,1}$ and $C^{1,\beta}$ regularity estimates as local minimizers of constant-coefficient energy functionals, down to microscopic and mesoscopic scales, respectively.

\smallskip

The first main result (Theorem~\ref{t.mainthm}) gives a sub-optimal algebraic error estimate in homogenization with strong stochastic integrability: it asserts roughly that, for any $s<d$, there exists an exponent $\alpha>0$, depending on~$s$, the dimension~$d$ and the constants controlling the uniform convexity of $L$ and a constant $C\geq 1$, depending additionally on the given data, such that, for every $\delta \in (0,1]$,
\begin{equation} \label{e.firstresult}
\P \left[ \exists \ep \in (0,\delta],\ \fint_U \left| u^\ep(x) - u_{\mathrm{hom}}(x) \right|^2\, dx \geq C \ep^{\alpha}  \right] 
\leq  C \exp\left( -\delta^{-s} \right),
\end{equation}
where $u^\ep$ and $u_{\mathrm{hom}}$ denote the unique minimizers in $g+H^1_0(U)$ of~\eqref{e.min} and~\eqref{e.min2}, respectively. Depending on the smoothness of the given Dirichlet boundary data~$g$, this~$L^2$ estimate may be upgraded to~$L^\infty$ by interpolating the latter between $L^2$ and $C^{0,\gamma}$ and using the nonlinear De Giorgi-Nash-Moser estimate. There is no loss in stochastic integrability in this interpolation and essentially no loss in the size of the error, since the exponent~$\alpha$ is already sub-optimal. (See Corollary~\ref{cor.DP}.) We remark that, at this stage in the development of the theory, we are less concerned with the sub-optimality of the size of the error than with the strength of the stochastic integrability; the former will be improved later. In~\eqref{e.firstresult} we have obtained the best possible stochastic integrability in the sense that such an estimate is false for $s>d$. 

\smallskip

The second main result (Theorem~\ref{t.c01}) asserts that local minimizers of the energy functional in~\eqref{e.min} are much smoother than minimizers for general functionals with measurable coefficients: it states roughly that any local minimizer~$u^\ep$ of the energy functional satisfies the estimate
\begin{equation} \label{e.co1int}
\sup_{x\in B_{1/2} \setminus B_\ep} \frac{\left| u^\ep(x) - u^\ep(0) \right|}{|x|}  \leq \Y \left( 1 + \| u^\ep\|_{L^2(B_1)} \right),
\end{equation}
where $\Y$ is a random variable (i.e, it depends on the coefficients but not on $u^\ep$) which, for any $s<d$, can be chosen to satisfy
\begin{equation*} \label{}
\E \left[ \exp(\Y^s) \right] < \infty. 
\end{equation*}
This is a quenched Lipschitz estimate ``down to the microscopic scale" since the left side of~\eqref{e.co1int} is a finite difference approximation of $|Du^\ep(0)|$. 

\smallskip

The estimate~\eqref{e.co1int} can be written in other forms, such as 
\begin{equation} \label{e.du2st}
\fint_{B_r} \left| Du^\ep(x) \right|^2\, dx \leq C\left( 1 + \| u^\ep\|_{L^2(B_1)}^2 \right) \quad \mbox{for every} \ r \in \left[ \ep \mathcal Y, \frac12 \right].
\end{equation}
The latter gives very good control of the spatial averages of the energy density of $u^\ep$. As was shown by Gloria and Otto~\cite{GO1} in the linear setting, if the probability measure~$\P$ satisfies a spectral gap hypothesis, then an estimate like~\eqref{e.du2st} implies optimal bounds on the variance of the energy of, e.g.,~minimizers with periodic boundary conditions. In a future work, we will prove this and other optimal quantitative  estimates from higher regularity estimates.

\smallskip

Theorem~\ref{t.c01} also asserts that local minimizers behave even more smoothly on \emph{meso}scopic scales (those of order $\ep^\beta$ for some $\beta\in(0,1)$) by giving an improvement of flatness estimate: see~\eqref{e.c1best}.

\smallskip

The proof of the error estimates, like the arguments of~\cite{DM1,DM2}, is variational and centers on the analysis of certain subadditive and superadditive energy quantities. However, the methods here differ substantially from those of~\cite{DM1,DM2}, as quantitative results present difficulties which do not appear in the qualitative theory and which require not just a harder analysis but also a new approach to the problem. The qualitative theory is based on the observation that the energy of a minimizer with respect to affine Dirichlet conditions is subadditive with respect to the domain. This monotonicity allows one to obtain a deterministic limit for this energy, as the domain becomes large, via a relatively soft argument based on the ergodic theorem.

\smallskip

To obtain a convergence rate for this limit (see Theorem~\ref{t.mulimit}), we introduce a new \emph{super}additive energy quantity by removing the boundary condition and adding a linear term to the energy functional. This is a kind of convex dual of the subadditive quantity, as we explain in more detail in Section~\ref{s.energies}. The main part of the analysis is to show that minimizers of the dual quantity are close to affine functions in a suitable sense, which implies that the subadditive and superadditive quantities are close to each other, up to a small error. Thus the quantities are in fact additive, up to a suitably small error, which gives the desired rate for the limits. 
This is the focus of Sections~\ref{s.energies} and~\ref{sec.qcc}, and the proof of Theorem~\ref{t.mainthm} is then completed in Section~\ref{s.DP} with the help of an oscillating test function argument. 

\smallskip

The proof of the quenched Lipschitz estimate is inspired by Avellaneda and Lin~\cite{AL1,AL2} who showed, using a perturbation argument in the context of periodic media and linear equations, that solutions of a heterogeneous equation inherit higher regularity from the homogenized equation. While we cannot make use of compactness arguments in the stochastic setting, the error estimates in Theorem~\ref{t.mainthm} are strong enough to implement a quantitative version of this technique.

\smallskip

The closest previous result to the quenched Lipschitz estimate is a quenched H\"older estimate due to Marahrens and Otto~\cite{MO} (recently extended to the continuum case by Gloria and Marahrens~\cite{GM}). They proved a $C^{0,1-\delta}$ estimate for linear equations, for every $\delta>0$, with somewhat weaker stochastic integrability (in our notation, they obtained that all finite moments of $\Y$ are bounded). The methods of proof in all of these works are completely different from the one here and based on logarithmic Sobolev or spectral gap inequalities. Here we also apply concentration of measure, but we use it in a more elementary form and in a modular way. We therefore believe our results will extend in a straightforward way to coefficients satisfying only much weaker mixing conditions. Moreover, H\"older estimates, as noted by Avellaneda and Lin~\cite{AL1,AL2} in the periodic case, are significantly easier to obtain than a Lipschitz estimate, which is the critical estimate for this problem.

\smallskip

While the results in this paper are completely new in the nonlinear setting, we emphasize that even for linear equations we prove new results: compared to  previous works, our error estimates exhibit stronger stochastic integrability and the quenched Lipschitz estimate is new. Moreover, while we prove our results under an independence assumption, we believe that the arguments can be modified in a simple way to handle quantitative ergodicity assumptions in other forms (such as weaker mixing conditions or spectral gap-type assumptions) and to generalize easily to systems of equations under strong ellipticity assumptions. 

\smallskip

In the next three subsections, we present the precise hypotheses and the statements of the main results.  

\subsection{Modeling assumptions}

We take $d\in\N$, $d\geq 2$ and $\Lambda \geq 1$ to be parameters which are fixed throughout the paper. We require the integrands~$L$ of our energy functionals to satisfy the following conditions:
\begin{enumerate}
\item[(L1)] $L:\Rd \times \Rd \to \R$ \ is a \emph{Carath\'eodory function}, that is, $L(p,x)$ is measurable in $x$ and continuous in $p$.

\smallskip

\item[(L2)] $L$ is uniformly convex in $p$: for every $p_1,p_2,x\in \Rd$,
\begin{equation*} \label{}
\frac14|p_1-p_2|^2 \leq  \frac12 L(p_1,x) + \frac 12 L(p_2,x) -L\left( \frac12 p_1+\frac12p_2,x \right)  \leq \frac\Lambda4|p_1-p_2|^2 .
\end{equation*}
\end{enumerate} 

Note that (L2) implies $L(\cdot,x)$ is $C^1$, for each $x\in \Rd$, and $\left[ D_pL(\cdot,x) \right]_{C^{0,1}(\Rd)} \leq \Lambda$, where $D_pL$ denotes the gradient of $L$ with respect to the first variable. 

\smallskip

We define $\Omega$ to be the set of all such functions:
\begin{equation*} \label{}
\Omega := \left\{ L \,:\,  \mbox{$L$ satisfies~(L1) and (L2)} \right\}.
\end{equation*}
Note that $\Omega$ depends on the fixed parameter $\Lambda>1$. We endow $\Omega$ with the following family of~$\sigma$--algebras: for each Borel $U \subseteq \Rd$, define
\begin{multline*} \label{}
\F(U):= \mbox{the $\sigma$--algebra generated by the family of random variables}\\
L \mapsto \int_{U} L(p,x) \phi(x)\, dx, \quad p\in \Rd, \ \phi \in C^\infty_c(\Rd). 
\end{multline*}
The largest of these is denoted by $\F:= \F(\Rd)$. It is also convenient to define a subset of $\Omega$ consisting of Lagrangians $L$ such that $L$ and $D_pL$ are locally bounded in $p$, uniformly in $x$. For each $K\geq 0$, We set
\begin{equation*} \label{}
\Omega(K):= \left\{ L \in\Omega\,:\, \forall p,x\in\Rd, \  |p|^2 - K(1+|p|) \leq L(p,x) \leq \Lambda|p|^2 +K(1+|p|) \right\}.
\end{equation*}

\smallskip

The \emph{random environment} is modeled by a given probability measure $\P$ on $(\l,\F)$. The expectation with respect to~$\P$ is denoted by~$\E$. We require~$\P$ to satisfy the following three assumptions:

\smallskip

\begin{enumerate}

\item[(P1)] $\P$ has a unit range of dependence: for all Borel subsets $U,V\subseteq \Rd$ such that $\dist(U,V) \geq 1$,
\begin{equation*} \label{}
\mbox{$\F(U)$ and $\F(V)$ are $\P$--independent.}
\end{equation*}

\item[(P2)] $\P$ is stationary with respect to $\Zd$--translations: for every $z\in \Zd$ and $E\in \F$,
\begin{equation*} \label{}
\P \left[ E \right] = \P \left[ T_z E \right],
\end{equation*}
where the translation group $\{T_z\}_{z\in\Zd}$ acts on $\l$ by $(T_zL)(p,x) = L(p,x+z)$.

\smallskip

\item[(P3)] $L$ and $D_pL$ are bounded locally uniformly in~$p$ and uniformly on the support of~$\P$: there exists $K_0\geq 1$ such that
\begin{equation*} \label{}
\qquad \P \Big[ L \in \Omega(K_0) \Big] = 1.
\end{equation*}
\end{enumerate}

These hypotheses  are stronger than those of~\cite{DM2} in several respects. First, for the sake of simplicity, we consider only the case of quadratic growth. This assumption is probably not essential, and we speculate that  adaptations of our arguments should give results, for example, in the case of Lagrangians growing like $|p|^m$ for $m>1$. Second,  in (L2) we have strengthened the assumption of convexity to uniform convexity. While this assumption can probably be relaxed, some  form of strict convexity is essential to our method. Third, the assumption of ergodicity has been strengthened to the independence condition~(P1). Quantitative ergodicity assumptions are of course required for quantitative results, although our methods yield quantitative homogenization results under, for example, much weaker \emph{uniform mixing} conditions as well.

\subsection{A sub-optimal error estimate for the Dirichlet problem}
The first main result of the paper is an estimate for the error in homogenization of the Dirichlet problem. It gives a sub-optimal algebraic estimate for the size of the error but with essentially optimal stochastic integrability.

In the following statement and throughout the paper, we denote the Lebesgue measure of a set $E\subseteq \Rd$ by $|E|$ and set $\fint_U f(x)\,dx:= |U|^{-1} \int_U f(x)\,dx$.

\begin{theorem}
\label{t.mainthm}
Let $U \subseteq \Rd$ be a bounded Lipschitz domain, $M\geq 1$, $t>2$ and $s\in (0,d)$. There exist $\alpha(d,\Lambda,t)>0$, $C(d,\Lambda,s,t,U)\geq 1$ and a nonnegative random variable $\X$ on $(\Omega,\F)$, depending on $(d,\Lambda,M,t,s)$ and satisfying
\begin{equation} \label{e.intX}
\E \big[ \exp(  \X )  \big] \leq CM^d,
\end{equation}
such that the following holds: for every $L \in\Omega$, $\ep \in (0,1]$ and $g\in W^{1,t}(U)$ such that 
\begin{equation} \label{e.gineq}
K_0 + \left(  \fint_U |Dg(x)|^t \, dx\right)^{1/t} \leq M,
\end{equation}
the unique functions~$u^\ep,u_{\mathrm{hom}} \in g+H^1(U)$ for which \begin{equation} \label{e.localminz1}
\int_{U} L\left(Du^\ep(x),\frac x\ep \right) \,dx \leq \int_{U} L\left(Dw(x),\frac x\ep \right) \, dx \quad \mbox{for every} \ w \in g + H^1_0(U),
\end{equation}
and 
\begin{equation} \label{e.localminz2}
\int_{U} \overline L(Du_{\mathrm{hom}}(x)) \,dx \leq \int_{U} \overline L(Dw(x)) \, dx \quad \mbox{for every} \ w \in g + H^1_0(U),
\end{equation}
satisfy the estimate
\begin{equation}\label{e.mainthm}
\fint_U \left| u^\ep(x) - u_{\mathrm{hom}}(x) \right|^2\, dx 
 \leq CM^2\left(1+ \X\ep^s  \right) \ep^{\alpha(d-s)}.
\end{equation}
\end{theorem}

By Chebyshev's inequality, the conclusion of Theorem~\ref{t.mainthm} implies, in the notation of the theorem, that for some $C(d,\Lambda,s,t,U)\geq 1$ and every $s\in (0,d)$ and $\delta \in (0,1]$,
\begin{equation*}
\P \left[ \exists \ep \in (0,\delta ],\  \fint_U \left| u^\ep(x) - u_{\mathrm{hom}}(x) \right|^2\, dx \geq CM^2 \ep^{\alpha(d-s)}  \right] 
\leq  CM^d \exp\left( -\delta^{-s} \right).
\end{equation*}
This is an algebraic estimate for the size of the homogenization error with very strong bounds on the probability of deviations. However, there is more information in~\eqref{e.mainthm} than the latter, and the former is often more convenient to work with since it is in the form of an \emph{a priori} estimate (i.e., $\mathcal X$ is independent of $\ep$, $g$, etc). Note the tradeoff between our control on the error threshold and the probability of deviations: as the exponent loses power, we gain more stochastic integrability, and vice versa. The dependence of $\X$ on $M$ can of course be removed in the linear case (i.e., $L$ a quadratic form), but in the general nonlinear setting the integrability of~$\X$ necessarily exhibits some mild dependence on~$M$.  
\smallskip

While Theorem~\ref{t.mainthm} measures the error only in $L^2$ in space, we also obtain interior error estimates in $L^\infty$ by interpolating between the $L^2$ error and the De Giorgi-Nash-Moser~$C^{0,\gamma}$ estimates. See Corollary~\ref{cor.DP}.

\subsection{Stochastic higher regularity}

It is well known that minimizers of variable-coefficient energy functionals do not, in general,  have better H\"older regularity than that provided by the De Giorgi-Nash-Moser estimate ($C^{0,\gamma}$ for a small $\gamma> 0$) or Sobolev regularity than that provided by the Meyers estimate (which is $W^{1,2+\delta}$ for a small $\delta > 0$). Nevertheless, the following theorems asserts that the regularity is typically much better for energy functionals sampled by a probability measure~$\P$ with a finite range of dependence: minimizers have $C^{0,1}=W^{1,\infty}$ regularity, down to microscopic scales, and even $C^{1,\beta}$ regularity down to mesoscopic scales.  

\begin{theorem}
\label{t.c01}
Fix $M\geq 1$ and $s\in(0,d)$. Then there exists $C(d,\Lambda,s)\geq 1$ and a nonnegative random variable $\Y$ on $(\Omega,\F)$, depending on $(d,\Lambda,M,s)$, satisfying
\begin{equation} \label{e.Yinteg}
\E \big[ \exp\left(  \Y^s \right)  \big] \leq CM^{d}
\end{equation}
and the following: for every $L \in\Omega$, $R\geq 2$ and $u\in H^1(B_R)$ satisfying 
\begin{equation} \label{e.osccontrol2}
K_0 + \frac 1R\left(  \fint_{B_R} |u(x)|^2 \, dx\right)^{1/2} \leq M
\end{equation}
and
\begin{equation} \label{e.localmin2}
\int_{B_R} L\left(Du(x),x \right) \,dx \leq \int_{B_R} L\left(Dw(x),x \right) \, dx \quad \mbox{for every} \ w \in u+H^1_0(B_R),
\end{equation}
we have the estimate
\begin{equation}\label{e.c01est}
\frac1r \osc_{B_r} u  \leq C M \quad \mbox{for every} \ \mathcal Y \leq r \leq \frac12 R.
\end{equation}
Moreover, there exist $\beta(d,\Lambda)>0$ and $c(d,\Lambda,s)>0$ such that, for every $\gamma \in (0,\beta]$ and $r\in [R^{c\gamma},R/2]$,
\begin{equation}\label{e.c1best}
\inf_{p\in\Rd} \frac{1}{r} \osc_{x\in B_{r}} \left( u(x) + p\cdot x \right) \leq C\Y M \left( \frac rR \right)^{\gamma}.
\end{equation}
\end{theorem}

It is appropriate to consider a coarsening of the $C^{0,1}$ seminorm by removing the effect of microscopic oscillations, as in the left side of~\eqref{e.c01est}, because the regularizing effect is due to the correlation structure of the coefficients: which of course cannot have influence on scales smaller than the correlation length scale. Here we coarsen down to a \emph{random} scale $\mathcal Y$. Of course, we may also coarsen at the unit scale if we allow the right side to be random, as~\eqref{e.c01est} implies
\begin{equation*}
\frac1r \osc_{B_r} u  \leq C\mathcal Y M \quad \mbox{for every} \ 1 \leq r \leq \frac12 R.
\end{equation*}
On the other hand, if the~$L$'s sampled by $\P$ are uniformly smooth, then~\eqref{e.c01est} implies a full $C^{0,1}$ estimate without the coarsening because in this case we may control the smaller scales by applying local Schauder estimates.

\smallskip

By the Caccioppoli inequality, the estimate~\eqref{e.c01est} also implies
\begin{equation*} \label{}
\sup_{\mathcal Y\leq r \leq R/4}  \fint_{B_r} \left| Du(x) \right|^2\, dx \leq C M^2.
\end{equation*}
This gives very strong control of the energy density of local minimizers down to the unit scale. In the linear setting, special cases of this kind of estimate (applied to modified correctors and the Green's functions) lie at the heart of the quantitative theory of Gloria and Otto~\cite{GO1,GO2} and Gloria, Neukamm and Otto~\cite{GNO,GNO2}. Theorem~\ref{t.c01} can therefore be used, together with spectral gap-type concentration inequalities, to obtain an alternative proof of the optimal quantitative estimates obtained in these papers.

\smallskip

The proof of Theorem~\ref{t.c01} is a nonlinear and quantitative version of an idea of Avellaneda and Lin~\cite{AL1,AL2}: since the heterogeneous energy functional is close to the homogenized functional on large scales, we can obtain higher regularity by treating minimizers of the former as a perturbations of those of the latter. This idea was formalized in~\cite{AL1,AL2} in the context of periodic media, via compactness arguments. In the stochastic setting here, we do not have compactness and therefore the perturbation argument must be more quantitative. The perfect tool is actually Theorem~\ref{t.mainthm}: what is needed is an algebraic rate of convergence in homogenization and strong control of the stochastic fluctuations. Note that the algebraic exponent in Theorem~\ref{t.mainthm} disappears ``into the constant" in Theorem~\ref{t.c01} (so it is irrelevant that the exponent is sub-optimal), but the strong stochastic integrability~$\Y$ is inherited from that of~$\X$.

\subsection{Outline of the paper}

Section~\ref{sec.qcc} is the heart of the paper, wherein we state and prove the core result, Theorem~\ref{t.mulimit}, on the convergence of the subadditive and superadditive energies. In Section~\ref{s.energies} we introduce the key concepts and make some preliminary observations in preparation for the proof of this result. We derive the error estimate for the Dirichlet problem in Section~\ref{s.DP} by reducing it to Theorem~\ref{t.mulimit}, and from it we obtain the stochastic $C^{0,1}$ estimate in Section~\ref{s.C01}.

\section{Subadditive and superadditive energies}
\label{s.energies}

The analysis in the first part of this paper is centered on two monotone quantities involving the energy. Up to normalizing factors, one is \emph{subadditive} and the other is \emph{superadditive}. The former was considered already in~\cite{DM1,DM2} and was the basis of the qualitative proof of homogenization there, while the latter is considered for the first time here. In this section we define these quantities and review some of their elementary properties, explain why they are convex duals of each other. We begin by reviewing some notation.

\subsection{Convention for constants}
Throughout the paper, unless otherwise indicated, the symbols $C$ and $c$ denote constants which depend on the dimension $d$ and the parameter $\Lambda$ in~(L2) and may vary in each occurrence.

\subsection{Suppressing the dependence on $L$} 
\label{ss.dropL}
Throughout, the probability space is~$(\Omega,\F,\P)$ and~$L$ denotes the canonical element of~$\Omega$ that is sampled according to~$\P$. Since it is cumbersome to display dependence on~$L$ in each of our quantities, we often suppress this dependence in our notation (for example, in the statement of Theorem~\ref{t.mainthm}, each of $u^\ep$, $u$ and $\mathcal X$ depend on $L$). However, the reader should keep in mind that any quantity implicitly defined in terms of~$L$ is random, and the symbols~$\P$ and~$\E$ should always be understood with respect to this randomness.

\subsection{Definition of the energy quantities $\mu$ and $\nu$}
\label{ss.defs}

For each $p,q\in \Rd$, bounded open subset $U\subseteq \Rd$ and $L\in\l$, we define the two quantities
\begin{equation*} \label{}
\mu(U,q,L) := \min\left\{ \fint_U \left( L(Dw(x),x) - q\cdot Dw(x) \right) \, dx \,:\, w\in H^1(U) \right\}
\end{equation*}
and
\begin{equation*} \label{}
\nu(U,p,L):= \min \left\{ \fint_U L(Dw(x),x)\, dx \,:\, w - \ell_p\in H^1_0(U)  \right\},
\end{equation*}
where $\ell_p$ denotes the plane $\ell_p(x):= p\cdot x$. The quantity $\nu$ was introduced by Dal Maso and Modica~\cite{DM1,DM2} and was central to their proof of qualitative homogenization. Note that, for every $p\in\Rd$, $U\subseteq\Rd$ and $L\in\l$,
\begin{equation} \label{e.munuord}
\sup_{q\in\Rd} \left(q\cdot p + \mu(U,q,L) \right) \leq  \nu(U,p,L).
\end{equation}

\smallskip

To build some intuition for $\mu$ and $\nu$, and to see that they are convex dual to each other in some sense, we examine the case that $L$ has no spatial dependence, i.e., $L(p,x)=L(p)$. The values of $\mu$ and $\nu$ are then easy to compute, as there is no dependence on $U$ and the integrals may be removed. We obtain
\begin{equation*} \label{}
\mu(U,q,L) = \mu(q,L) = \min_{p\in\Rd} \left( L(p) - q\cdot p \right) \quad \mbox{and} \quad \nu(U,p,L) = \nu(p,L) = L(p).
\end{equation*}
The Legendre transform $L^*$ of $L$ may therefore be written as
\begin{equation*} \label{}
L^*(q) = \sup_{p\in\Rd} \left( p\cdot q - L(p) \right) = - \mu(q,L).
\end{equation*}
We thus observe that~$\nu$ identifies $L$ while~$\mu$ naturally identifies $L^*$. We may also write~$L$ in terms of~$\mu$ by duality:
\begin{equation} \label{e.Lbymu}
L(p) = \sup_{q\in\Rd} \left( p\cdot q - L^*(q) \right) = \sup_{q\in\Rd} \left( p\cdot q + \mu(q,L) \right).
\end{equation}
In Section~\ref{s.idenL}, we generalize~\eqref{e.Lbymu} to the stochastic case.

\smallskip

As mentioned in the previous subsection, we usually suppress the dependence of $\mu$ and $\nu$ on $L$, unless required for clarity, by writing $\mu(U,q)$ and $\nu(U,p)$. However, the reader should keep in mind that these quantities are random variables. 

\smallskip

We note that $\mu(U,q)$ is well-defined and finite. Indeed, by~(P3) we have
\begin{equation} \label{e.mubounds0}
\P \left[ \mbox{for every bounded, open} \ U\subseteq\Rd, \ -2(K_0+|q|)^2  \leq \mu(U,q) \leq K_0 \right] = 1.
\end{equation}
The first inequality holds because the left side (rather crudely) bounds $L(p,x)-q\cdot p$ from below, uniformly in~$p$, on the support of $\P$, which we see from the first inequality in~(P3). We get the second inequality in~\eqref{e.mubounds} by taking zero as a test function in the definition of~$\mu$ and using the second inequality in~(P3). In particular, as $K_0\geq 1$, we obtain
\begin{equation} \label{e.mubounds}
 \left| \mu(U,q) \right| \leq 2(K_0+|q|)^2 \quad \mbox{$\P$--a.s.}
\end{equation}
A similar argument as the one for~\eqref{e.mubounds} leads to the bound
\begin{equation} \label{e.nubounds}
|p|^2 - K_0\left(1+|p| \right) \leq \nu(U,p) \leq \Lambda |p|^2 + K_0\left( 1 + |p| \right) \quad \mbox{$\P$--a.s.}
\end{equation}
Here we used Jensen's inequality in the expression for the energy of a minimizer to get the left side, and test with the zero function to get the right side. 

\subsection{The minimizers $u$ and $v$}
Up to an additive constant, the minimization problem in the definition of $\mu$ has a unique minimizer in $H^1(U)$. (The uniqueness of the minimizer follows from uniform convexity, see Lemma~\ref{l.convexL2} below.) We denote the unique minimizer which has mean zero on every connected component of $U$ by $u(\cdot,U,q)$. In other words, $u(\cdot,U,q)$ is the unique element of $H^1(U)$ satisfying
\begin{equation*} \label{}
\left\{ 
\begin{aligned}
& \fint_{U} \left( L(Du(x,U,q),x)  - q\cdot Du(x,U,q) \right)\, dx \\
& \quad \leq \fint_{U} \left( L(Dw(x),x) - q\cdot Dw(x) \right)\, dx && \forall\, w\in H^1(U),\\
& \fint_V u(x,U,q)\,dx= 0 && \forall \, V\subseteq U \ \mbox{with} \ V \ \mbox{and} \ U \setminus V \ \mbox{open.}
\end{aligned}
\right.
\end{equation*}
We denote the spatial average of~$Du(\cdot,U,q)$ by
\begin{equation} \label{e.P}
P(U,q):= \fint_U Du(x,U,q) \, dx.
\end{equation}

We let $v(\cdot,U,p)$ denote the minimizer for $\nu(U,p)$, that is, the unique function in $H^1(U)$ satisfying 
\begin{equation*} \label{}
\left\{ 
\begin{aligned}
& \fint_{U} L\left(Dv(x,U,p),x\right) \, dx \leq \fint_{U} L\left(Dv(x,U,p)+Dw(x),x\right)\,dx && \forall  w\in H^1_0(U),\\
& v(\cdot,U,p) - \ell_p \in H^1_0(U).
\end{aligned}
\right.
\end{equation*}
We stress once again that $u(\cdot,U,q)$, $v(\cdot,U,p)$ and $P(U,q)$ are random elements, as they depend on~$L$.

\subsection{Notation for cubes}
\label{ss.cubes}
For each $x\in \Rd$ and $n\in\N_*$, we define the triadic cube
\begin{equation*} \label{}
Q_n(x) := 3^n \lfloor 3^{-n} x \rfloor + \left(-\frac12 3^n,\frac12 3^n\right)^d.
\end{equation*}
Here $\lfloor r \rfloor$ denotes, for $r\in \R$, the largest integer not greater than $r$ and, for a point $x=(x_1,\ldots,x_d)\in\Rd$, we set $\lfloor x \rfloor:=\left( \lfloor x_1 \rfloor,\ldots,\lfloor x_d \rfloor \right)$. We have centered each cube $Q_n(x)$ at a point on the lattice $3^n\Zd$ in order to keep them disjoint. In particular, we note that $Q_n(x)$ is not necessarily centered at $x$ and, neglecting a subset of Lebesgue measure zero (the boundary of the open cubes), we see that $y\in Q_n(x)$ if and only if $Q_n(x) = Q_n(y)$. For $m,n\in\N_*$, $Q_{n+m}$ is the disjoint union, up to a zero measure set, of $3^{dm}$ cubes of the form $Q_n(x)$. We write $Q_n:= Q_n(0)$. 

\smallskip

We also define the \emph{trimmed} triadic cube by 
\begin{equation*} \label{}
Q_n^\circ(x):= 3^n \lfloor 3^{-n} x \rfloor + \left(-\frac12\left(3^n-1\right), \frac12\left(3^n-1\right) \right)^d.
\end{equation*} 
The trimmed cube $Q^\circ_n(x)$ is obtained from $Q_n(x)$ by removing a layer of thickness $1/2$ from each face. The reason we have trimmed this layer near the boundary is that it ensures that the trimmed cubes are separated by a unit distance from each other, which is convenient when we work with the independence assumption~(P2). We set $Q^\circ_n:=Q^\circ_n(0)$.

\smallskip

For future reference we note that the proportion of volume occupied by the trimmed region is of order $3^{-n}$ since, for any $n,m\in\N$,
\begin{align*} \label{}
|Q_{n+m}| - 3^{dm} |Q^\circ_n| & = 3^{d(n+m)} - 3^{dm}\left( 3^n-1 \right)^d   \leq C3^{d(n+m)-n}
\end{align*}
which implies
\begin{equation} \label{e.trim}
\frac{3^{dm}\left|Q^\circ_{n}\right|}{\left|Q_{n+m}\right|} \geq 1-C3^{-n}
\end{equation}

\smallskip

We introduce a third family of cubes $\tilde Q_{n+1}$ defined for $n\in\N_*$ by
\begin{equation*} \label{}
\tilde Q_{n+1}(x) := 3^n \left\lfloor 3^{-n} x \right \rfloor + Q_{n+1}.
\end{equation*}
Thus $\tilde Q_{n+1}(x)$ is the cube centered at the same point as $Q_n(x)$, but with side lengths three times larger. These cubes are not disjoint and each cube of the form $\tilde Q_{n+1}(x)$ intersects~$5^d-1$ others. Note that $\tilde Q_{n+1}(x)$ is the translation by an element of $3^n\Zd$ of the cube $Q_{n+1}$, and thus~$\P$ has the same statistical properties in these cubes by the stationarity assumption. We use this family of cubes when we compare the energy at different scales in Lemma~\ref{l.patching}, as the argument there requires some overlapping cubes in the construction.

\subsection{Monotonicity of $\mu$ and $\nu$ and the definition of $\overline L$}

The quantity $\mu$ has a monotonicity property which is immediate from its definition, obtained by simply restricting the minimizers of larger regions to smaller ones. Namely, the map $U \mapsto |U| \mu(U,L,q)$ is \emph{superadditive}, by which we mean that, for all collections of pairwise disjoint bounded open subsets $U_1,\ldots,U_k\subseteq \Rd$ and every open set $U\subseteq \Rd$ such that 
\begin{equation} \label{e.UUk}
U_1\cup \cdots \cup U_k \subseteq U \quad \mbox{and} \quad \left| U \setminus (U_1\cup \cdots \cup U_k) \right| = 0,
\end{equation}
we have
\begin{equation} \label{e.musuperadd}
\mu(U,q) \geq \sum_{j=1}^k \frac{|U_j|}{|U|} \mu(U_j,q).
\end{equation}
In other words, $\mu(U,q)$ is bounded below by a weighted average of $\{\mu(U_j,q)\}_{j=1}^k$. To obtain~\eqref{e.musuperadd}, we note that, for each $j$, 
\begin{equation*} \label{}
\int_{U_j} \left(L\left(Du(x,U,q),x\right) -q\cdot Du(x,U,q)\right)dx \geq |U_j| \mu(U_j,q)
\end{equation*}
and then sum over $j\in \{1,\ldots,k\}$.

\smallskip

The superadditivity of $\mu$ implies that $\E \left[ \mu(Q_n,q) \right]$ is a monotone nondecreasing sequence in~$n$. Indeed, recall that $Q_{n+m}$ is the disjoint union of $3^{dm}$ cubes of the form $Q_n(x)$, up to a zero measure set, and therefore~\eqref{e.musuperadd} gives
\begin{equation} \label{e.mono}
\mu(Q_{n+m},q) \geq 3^{-dm} \sum_{Q_n(x) \subseteq Q_{n+m}} \mu(Q_n(x),q).
\end{equation}
Taking expectations gives
\begin{equation} \label{e.Emono}
\E \left[ \mu(Q_{n+m},q) \right] \geq \E \left[ \mu(Q_n,q) \right].
\end{equation}
In view of~\eqref{e.mubounds}, the quantity
\begin{equation} \label{e.mubardef}
\overline \mu(q):= \sup_{ n\in \N} \E \left[ \mu(Q_n,q) \right] 
\end{equation}
is finite and hence, by~\eqref{e.Emono}, we have the limit
\begin{equation} \label{e.monolim}
\lim_{n\to \infty} \E \left[ \mu(Q_n,q) \right] = \overline \mu(q). 
\end{equation}

\smallskip

We next recall from~\cite{DM2} that, for each $p\in\Rd$, the quantity $U\mapsto |U|\nu(U,p)$ is \emph{subadditive}, i.e., 
for all collections of pairwise disjoint bounded open subsets $U_1,\ldots,U_k\subseteq \Rd$ and open $U\subseteq \Rd$ satisfying~\eqref{e.UUk}, we have
\begin{equation} \label{e.nusubadd}
\nu(U,p) \leq \sum_{j=1}^k \frac{|U_j|}{|U|} \nu(U_j,p).
\end{equation}
This holds because a candidate for a minimizer for $\nu(\cdot,p)$ in $U$ can be obtained by assembling the minimizers in each of the $U_j$'s. Similarly to~\eqref{e.mono} and~\eqref{e.Emono}, specializing to the triadic cubes we deduce that, for every $m,n\in\N$ and $p\in\Rd$,
\begin{equation} \label{e.mono.nu}
\nu(Q_{n+m},p) \leq 3^{-dm} \sum_{Q_n(x) \subseteq Q_{n+m}} \nu(Q_n(x),p)
\end{equation}
and taking expectations and applying stationarity yields 
\begin{equation} \label{e.Emono.nu}
\E \left[ \nu(Q_{n+m},p) \right] \leq \E \left[ \nu(Q_n,p) \right].
\end{equation}
We define the \emph{effective Lagrangian} $\overline L:\Rd \to \R$ by
\begin{equation} \label{e.Lbar}
\overline L(p) := \inf_{n\in\N} \E \left[ \nu(Q_n,p) \right].
\end{equation}
Note that this is the same definition for $\overline L$ given in~\cite{DM1,DM2}. By~\eqref{e.Emono.nu}, we have 
\begin{equation} \label{e.monolimnu}
\lim_{n\to \infty} \E\left[ \nu(Q_n,p) \right] = \overline L(p).
\end{equation}

\subsection{Comparisons between $\E[\mu(Q_n,q)]$ and $\E[\mu(Q_n^\circ,q)]$}

For our reference, we record here a few observations concerning the expectation of $\mu$ in the trimmed and untrimmed dyadic cubes. We first note that~$\mu$ is also monotone with respect to the trimmed cubes, up to a small error. We have:
\begin{equation} \label{e.trimsa}
\mu(Q_{n+m}^\circ,q) \geq 3^{-dm} \sum_{Q_n^\circ(x) \subseteq Q_{n+m}} \mu(Q_n^\circ(x),q) - C\left( K_0 + |q| \right)^2 3^{-n} \quad \mbox{$\P$--a.s.}
\end{equation}
To obtain~\eqref{e.trimsa}, we note that in view of the remarks in Section~\ref{ss.cubes}, we may write $Q_{m+n}^\circ$, up to a zero measure set, as the union of $3^{dm}$ cubes of the form $Q^\circ_n(x)\subseteq Q_{m+n}^\circ$ and an open set of measure at most $C3^{m}$. We then deduce~\eqref{e.trimsa} from~\eqref{e.mubounds},~\eqref{e.trim} and~\eqref{e.musuperadd}. Taking expectations and using stationarity gives
\begin{equation} \label{e.Etrimmedmono}
\E \left[ \mu( Q^\circ_{m+n},q) \right]  \geq \E \left[ \mu(Q^\circ_{n},q) \right]  - C(K_0+|q|)^2 3^{-n}.
\end{equation}
By a similar argument, we obtain
\begin{equation} \label{e.trimsanu}
\nu(Q_{n+m}^\circ,p) \leq 3^{-dm} \sum_{Q_n^\circ(x) \subseteq Q_{n+m}} \nu(Q_n^\circ(x),q) + C\left( K_0 + |p| \right)^2 3^{-n} \quad \mbox{$\P$--a.s.}
\end{equation}

\smallskip

It is also useful to have some comparison between $\mu$ and $\nu$ in the trimmed and untrimmed cubes. By $| Q_n\setminus \overline{Q}_n^\circ| \leq 3^{-n}|Q_n|$,~\eqref{e.mubounds} and~\eqref{e.musuperadd}, we have
\begin{equation} \label{e.Ecubes}
\mu(Q_n^\circ,q) \leq \mu(Q_n,q) + C(K_0+|q|)^2 3^{-n} \quad \mbox{$\P$--a.s.} 
\end{equation}
Similarly, by~\eqref{e.nubounds} and~\eqref{e.nusubadd}, we have
\begin{equation} \label{e.Ecubesnu}
\nu(Q_n^\circ,p) \geq \nu(Q_n,p) - C(K_0+|p|)^2 3^{-n} \quad \mbox{$\P$--a.s.} 
\end{equation}
We also need an inequality bounding $\E\left[ \mu(Q_n,q) \right]$ from above by $\E\left[ \mu(Q_n^\circ,q) \right]$:
\begin{multline} \label{e.infamous}
\E\left[ \mu(Q_n,q) \right] \\ \leq \E \left[ \mu(Q_n^\circ,q) \right] + 3^d \left( \E \left[ \mu(Q_{n+1}^\circ,q) \right] -\E \left[ \mu(Q_n^\circ,q) \right] + C(K_0+|q|)^23^{-n}\right). 
\end{multline}
To get this, observe that the cube $Q_{n+1}^\circ$ is the disjoint union (up to a set of measure zero) of the untrimmed cube $Q_n$, $3^d-1$ trimmed cubes of the form $Q_n^\circ(x)$ and a set of measure at most $C3^{n(d-1)}$. We obtain~\eqref{e.infamous} after applying superadditivity, stationarity and~\eqref{e.mubounds} to this partition.

\subsection{Basic energy estimates}
In this subsection we record two simple consequences of the uniform convexity assumption~(L2) which are used repeatedly in the paper. The first is the following lemma, which gives gradient estimates for functions which are close to minimizers. The lemma is classical, e.g., it is almost the same as~Giaquinta~\cite[Chapter IX, Lemma 4.1]{G}.

\begin{lemma}
\label{l.convexL2}
For every $L\in\l$, $q\in\Rd$, bounded open $U\subseteq \Rd$ and $w,\xi\in H^1(U)$,
\begin{multline*} \label{}
\fint_{U} \left| Dw(x) - D\xi(x) \right|^2 \, dx \\
\leq 2 \bigg( \fint_{U} \big(L(Dw(x),x) - q\cdot Dw(x)\big)\,dx + \fint_{U}  \big(L(D\xi(x),x) - q\cdot D\xi(x) \big) \, dx \\  - 2\mu(U,q,L) \bigg).
\end{multline*}
\end{lemma}
\begin{proof}
Set $\zeta:= \frac12 w + \frac12 \xi \in H^1(U)$. Using~(L2), we compute
\begin{align*} \label{}
\mu(U,q,L) & \leq \fint_{U} \big( L(D\zeta (x),x) - q\cdot D\zeta(x) \big) \, dx \\
& \leq \frac12 \fint_{U} \big(L(Dw(x),x) - q\cdot Dw(x) \big) \, dx  \\
& \qquad + \frac12 \fint_{U} \big( L(D\xi(x),x) - q\cdot D\xi(x) \big) \, dx  -\frac14 \fint_{U} \left| Dw(x) - D\xi(x) \right|^2 \, dx.
\end{align*}
A rearrangement of the previous inequality gives the lemma. 
\end{proof}

In some arguments in the next section, we apply~Lemma~\ref{l.convexL2} in the case that $w=u(\cdot,U,q)$ and $\xi=u(\cdot,V,q)$ for open sets $U$, $V$ satisfying $U\subseteq V$ and $|V\setminus U| = 0$ (i.e., $V$ is the interior of the closure of the disconnected set $U$). The conclusion of the lemma gives 
\begin{equation} \label{e.convexUV}
\fint_{U} \left| Du(x,U,q) - Du(x,V,q) \right|^2 \, dx \leq 2 \left( \mu(V,q) - \mu(U,q) \right).
\end{equation}
A particular case of~\eqref{e.convexUV} allows us to compare the gradients of minimizers in the trimmed and untrimmed cubes. Write $U:= Q_n^\circ \cup (Q_n \setminus Q^\circ_n)$, apply the inequality with $V=Q_n$ and then use~\eqref{e.mubounds}~\eqref{e.trim},~\eqref{e.musuperadd} and to get 
\begin{align*}
\lefteqn{ \fint_U \left| Du(x,U,q) - Du(x,Q_n,q) \right|^2\,dx } \qquad & \\
& \leq C \left( \mu(Q_n,q) - \mu(U,q)  \right)  \\
& \leq C \left( \mu(Q_n,q) - \frac{|Q_n^\circ|}{|Q_n|} \mu(Q_n^\circ,q) -  \frac{|Q_n\setminus Q_n^\circ|}{|Q_n|} \mu(Q_n\setminus Q_n^\circ,q) \right) \\
& \leq C \left( \mu(Q_n,q) - \mu(Q_n^\circ,q) \right) + C \left( K_0^2 + |q|^2 \right) 3^{-n}.
\end{align*} 
Taking expectations and using that 
\begin{equation*}
u(\cdot,U,q) \vert_{Q_n^\circ} \equiv u(\cdot,Q_n^\circ,q) \quad \mbox{and} \quad u(\cdot,U,q) \vert_{Q_n\setminus Q_n^\circ} \equiv u(\cdot,Q_n\setminus Q_n^\circ,q),
\end{equation*}
as well as~\eqref{e.mubounds} and~\eqref{e.trim} again, we get
\begin{multline}
\label{e.trimsubject}
\frac{1}{|Q_n|} \int_{Q_n^\circ}\left| Du(x,Q_n^\circ,q) - Du(x,Q_n,q) \right|^2\,dx \\ + \frac{1}{|Q_n|} \int_{Q_n\setminus Q_n^\circ} \left| Du(x,Q_n,q) \right|^2 \,dx  \\
\leq C\left(\E\left[ \mu(Q_n,q)\right] - \E\left[ \mu(Q_n^\circ,q) \right]\right) + C \left( K_0^2 + |q|^2 \right) 3^{-n}.
\end{multline}

Another consequence of Lemma~\ref{l.convexL2} is the gradient bound
\begin{equation} \label{e.vgradbnd}
\fint_{U} \left| Du(x,U,q) \right|^2 \, dx \leq 6(K_0+|q|)^2 \quad \mbox{$\P$--a.s.}
\end{equation}
This we get by comparing $u(\cdot,U,q)$ to the zero function, using~(P3) and~\eqref{e.mubounds}. 

\smallskip

We also have the following variation of Lemma~\ref{l.convexL2} from a nearly identical argument: for every $w,\xi \in H^1(U)$ such that $\frac12 w+\frac12 \xi -\ell_p\in H^1_0(U)$, we have
\begin{multline} \label{e.convexL2nu}
\fint_{U} \left| Dw(x) - D\xi(x) \right|^2 \, dx \\ \leq 2 \left( \fint_{U} \left(L(Dw(x),x) + L(D\xi(x),x) \right) \, dx - 2\nu(U,p) \right).
\end{multline}
(Recall that $\ell_p$ is the plane $\ell_p(x) = p\cdot x$.) Comparing $v(\cdot,U,p)$ to~$\ell_p$, applying~(P3) and~\eqref{e.convexL2nu} and using the triangle inequality, we get
\begin{equation} \label{e.vgradbndnu}
\fint_{U} \left| Dv(x,U,p)  \right|^2 \, dx \leq C\left( K_0+|p| \right)^2 \quad \mbox{$\P$--a.s.}
\end{equation}

\smallskip

The second consequence of uniform convexity is kind of converse of Lemma~\ref{l.convexL2} which allows us to perturb minimizers without increasing the energy too much.

\begin{lemma}
\label{l.converseL2}
For every $L\in\l$, $q\in\Rd$, bounded open $U\subseteq \Rd$ and $w,\xi\in H^1(U)$,
\begin{multline*} \label{}
 \fint_{U} \big( L(Dw(x),x)  - q\cdot Dw(x) \big) \, dx \\ 
 \leq 2 \fint_{U} \big( L(D\xi(x),x) - q\cdot D\xi(x) \big)\, dx - \mu(U,q,L) + 2\Lambda \fint_{U} \left| Dw(x) - D\xi(x) \right|^2 \, dx.
\end{multline*}
\end{lemma}
\begin{proof}
We set $\zeta := 2\xi - w \in H^1(U)$ so that $\xi = \frac12 w + \frac12 \zeta$ and then use~(L2) to get
\begin{align*} \label{}
\lefteqn{ \fint_U \big( L(D\xi(x),x)  - q\cdot D \xi(x) \big)\, dx } \qquad & \\
& \geq \frac12 \fint_U \big( L(Dw(x),x) - q\cdot Dw(x) \big) \,dx  + \frac12 \fint_U\big( L(D\zeta(x),x)  - q\cdot D\zeta(x) \big)\, dx \\
& \qquad  - \frac\Lambda4 \fint_U \left| Dw(x) - D\zeta(x) \right|^2\, dx  \\
& \geq  \frac12 \fint_U \big( L(Dw(x),x) -q\cdot Dw(x) \big) \,dx  + \frac12 \mu(U,q,L)  \\
& \qquad - \Lambda \fint_U \left| Dw(x) - D\xi(x) \right|^2\, dx.
\end{align*}
A rearrangement now yields the lemma. 
\end{proof}

Similar to Lemma~\ref{l.convexL2}, we often use Lemma~\ref{l.converseL2} in the case that $\xi= u(\cdot,U,q)$ is the minimizer for $\mu(U,q)$, in which case the conclusion gives, for every $w\in H^1(U)$,
\begin{equation} \label{e.converseL2}
 \fint_{U} \big( L(Dw(x),x) -q\cdot Dw(x) \big) \, dx \leq \mu(U,q) + 2\Lambda \fint_{U} \left| Dw(x) - Du(x,U) \right|^2 \, dx.
\end{equation}
We also obtain a version of Lemma~\ref{l.converseL2} for planar boundary conditions, which states that, for every $p\in\Rd$ and $w,\xi \in H^1(U)$ such that $2\xi - w - \ell_p\in H^1_0(U)$, we have
\begin{multline} \label{e.converseL2nu}
 \fint_{U} L(Dw(x),x) \, dx \\ \leq 2 \fint_{U} L(D\xi(x),x)\, dx - \nu(U,p,L) + 2\Lambda \fint_{U} \left| Dw(x) - D\xi(x) \right|^2 \, dx.
\end{multline}

\subsection{Further properties of $\mu$, $\nu$ and $\overline L$}

For our reference, we record here some properties of $\mu$ and $\nu$ and their minimizers, particularly regarding their dependence on $p$ and $q$. 

\smallskip

An immediate consequence of~\eqref{e.convexL2nu} and~\eqref{e.converseL2nu} is that $p \mapsto \nu(x,U,p)$ is uniformly convex in $p$. Precisely, we claim that  
\begin{equation} \label{e.nuconvexp}
\frac14 |p_1-p_2|^2 \leq \frac12 \nu(U,p_1) + \frac12 \nu(U,p_2) - \nu(U,\tfrac12p_1+\tfrac12p_2) \leq \frac\Lambda4 |p_1-p_2|^2. 
\end{equation}
To get the first inequality of~\eqref{e.nuconvexp}, apply~\eqref{e.convexL2nu} with $w = v(\cdot,U,p_1)$ and $\xi = v(\cdot,U,p_2)$; to get the second inequality, apply~\eqref{e.converseL2nu} with $p=p_1$, $w = v(\cdot,U,p_2)$ and $\xi = v(\cdot,U,\frac12p_1+\frac12p_2)$. A further consequence of~\eqref{e.nubounds} and~\eqref{e.nuconvexp} is the continuity of $\nu$ in $p$: for every $p_1,p_2\in\Rd$, we have
\begin{equation} \label{e.nucontp}
\left| \nu(U,p_1) - \nu(U,p_2) \right| \leq C \left( K_0 + |p_1|+|p_2| \right) \left|p_1-p_2\right| \quad \mbox{$\P$--a.s.}
\end{equation}
Applying~\eqref{e.convexL2nu} to $w:=v(\cdot,U,p_1)$ and $\xi:= v(\cdot,U,p_2)$ and using~\eqref{e.nuconvexp}, we get
\begin{equation} \label{e.minnucloseL2D}
\fint_U \left| Dv(x,U,p_1) - Dv(x,U,p_2) \right|^2\, dx \leq \Lambda \left|p_1-p_2\right|^2.
\end{equation}
Specializing to  $U=Q_n$ and applying the Poincar\'e inequality, we get
\begin{equation} \label{e.minnucloseL2}
3^{-2n} \fint_{Q_n} \left( v(x,Q_n,p_1) - v(x,Q_n,p_2) \right)^2\, dx \leq C  \left|p_1-p_2\right|^2.
\end{equation}

The effective Lagrangian $\overline L$ defined in~\eqref{e.Lbar} satisfies the same growth and uniform convexity conditions as $L$:  for every $p,p_1,p_2\in \Rd$, 
\begin{equation} \label{e.Lbargrowth}
|p|^2 - K_0(1+|p|) \leq \overline L(p)  \leq \Lambda |p|^2 + K_0(1+|p|),
\end{equation}
\begin{equation} \label{e.Lbarconvex}
\frac14|p_1-p_2|^2 \leq   \frac12 \overline L(p_1) + \frac 12 \overline L(p_2) - \overline L\left( \frac12 p_1+\frac12p_2 \right) \leq \frac\Lambda4|p_1-p_2|^2
\end{equation}
and
\begin{equation} \label{e.LbarLip}
\left| \overline L(p_1) - \overline L(p_2) \right| \leq C \left( K_0 + |p_1| + |p_2| \right) |p_1-p_2|.
\end{equation}
These are  immediate consequences of~\eqref{e.nubounds},~\eqref{e.nuconvexp} and~\eqref{e.nucontp}. Observe that~\eqref{e.Lbarconvex} implies that $\overline L$ is differentiable at every $p\in\Rd$ and $D\overline L$ is Lipschitz continuous. In fact, for every $p,p_1,p_2\in \Rd$, we have
\begin{equation} \label{e.DLbarLip}
\left| D\overline L(p) \right| \leq C(K_0+|p|) \quad \mbox{and} \quad \left| D \overline L(p_1) - D\overline L(p_2) \right| \leq 2\Lambda |p_1-p_2|.
\end{equation}

Finally, we record the continuity of the maps $q\mapsto \mu(U,q)$ and $q\mapsto u(\cdot,U,q)$. Using $u(\cdot,U,q_1)$ as a minimizer candidate for $\mu(U,q_2)$ and the estimate~\eqref{e.vgradbnd}, we discover that
\begin{equation*} \label{}
\mu(U,q_2) \leq \mu(U,q_1) + C|q_1-q_2| \left(K_0 + |q_1| \right) \quad \mbox{$\P$--a.s.}
\end{equation*}
We deduce that, for every $q_1,q_2\in\Rd$ and bounded open $U\subseteq \Rd$,
\begin{equation} \label{e.mucontq}
\left| \mu(U,q_1) - \mu(U,q_2) \right| \leq C \left( K_0 + |q_1| + |q_2| \right) | q_1-q_2 | \quad \mbox{$\P$--a.s.}
\end{equation}
Lemma~\ref{l.convexL2} and a similar computation yield that
\begin{equation} \label{e.minmucloseL2D}
\fint_{U} \left| Du(x,U,q_1)  -  Du(x,U,q_2) \right|^2\, dx \leq C \left( K_0 + |q_1| + |q_2| \right) | q_1-q_2 | \quad \mbox{$\P$--a.s.}
\end{equation}
Specializing to the case $U=Q_n$ and applying the Poincar\'e inequality, we get  
\begin{multline} \label{e.minmucloseL2}
3^{-2n}  \fint_{Q_n} \left( u(x,Q_n,q_1)  -  u(x,Q_n,q_2) \right)^2\, dx  \\
\leq C \left( K_0 + |q_1| + |q_2| \right) | q_1-q_2 | \quad \mbox{$\P$--a.s.}
\end{multline}

\section{Convergence of the energy and flatness of minimizers}
\label{sec.qcc}

In this section we prove our first quantitative result: a sub-optimal algebraic estimate for the rate of convergence in the limits~\eqref{e.monolim} and~\eqref{e.monolimnu} as well as for the flatness of the respective minimizers. It is the main step toward the results stated in the introduction. 

\begin{theorem}
\label{t.mulimit}
Fix $q \in \Rd$. There exist $\alpha(d,\Lambda)>0$,  $c(d,\Lambda)>0$, $C(d,\Lambda)\geq 1$ and a unique $\overline P(q)\in\Rd$ such that
\begin{equation} \label{e.mulimform}
\overline \mu(q) + \overline P \cdot q = \overline L(\overline P) 
\end{equation}
and, for every $s\in (0,d)$, $R\geq 1$, $n\in\N_*$ and $t \geq 1$,
\begin{multline} \label{e.mulimit}
\P \bigg[ \sup_{y\in B_R} \Big( \left| \mu(y+Q_n,q) - \overline \mu(q) \right| + \left| \nu(y+Q_{n},\overline P) - \overline L(\overline P) \right| \Big) \\
 \geq C(K_0+|q|)^2 3^{-n\alpha(d-s)}t  \bigg]  \leq CR^d\exp\left( -c3^{sn} t\right)
\end{multline}
and 
\begin{multline} \label{e.planes}
\P \bigg[ \sup_{y\in B_R} 3^{-2n} \fint_{y+Q_n} \left( u(x,y+Q_n,q) - \overline P \cdot ( x-y) \right)^2 + \left( v(x,y+Q_n,\overline P) - \overline P \cdot x \right)^2 \,dx \\ \geq C(K_0+|q|)^23^{-n\alpha(d-s)} t \bigg]   \leq CR^d\exp\left( -c3^{sn} t\right).
\end{multline}
\end{theorem}

We prove Theorem~\ref{t.mulimit} in Section~\ref{ss.mulimproof} using the flatness theory we construct in Section~\ref{ss.flat} plus an elementary concentration argument. The final subsection contains extensions of Theorem~\ref{t.mulimit} and a demonstration of the fact that $-\overline \mu$ and $\overline L$ are convex dual functions.

\subsection{Reduction to the case $q=0$}
\label{ss.dropq}

It suffices to prove Theorem~\ref{t.mulimit} in the case that $q=0$. To see this, suppose that the statement of the theorem holds in this special case and fix $q\in\Rd$. Consider the probability measure $\P_q$ on $(\Omega,\F)$ which is the pushforward of~$\P$ under the map $L\mapsto L_q$, where the latter is defined by
\begin{equation} \label{e.Lq}
L_q(p,x) := L(p,x) -q\cdot p.
\end{equation}
(Recall that if $\pi:\l \to \l$ is an $\F$-measurable map, then the \emph{pushforward of $\P$ under $\pi$} is the probability measure denoted by $\pi_{\#} \P$ and defined for $E \in\F$ by $\pi_{\#}\P\left[E\right]:= \P\left[ \pi^{-1}(E) \right]$). Then it is easy to check that $\P_q$ satisfies the assumptions (P1), (P2) and (P3) after we replace the constant $K_0$ in (P3) by $K_0+|q|$. Applying the special case of Theorem~\ref{t.mulimit} with~$\P_q$ in place of~$\P$ and rewriting the statement in terms of $\P$ itself, we obtain the general statement of the theorem.

\smallskip

Therefore, in Sections~\ref{ss.flat} and~\ref{ss.mulimproof} we assume $q=0$ and drop $q$ from our notation by writing, for example, $\mu(U)$, $u(\cdot,U)$ and $P(U)$ in place of $\mu(U,0)$, $u(\cdot,U,0)$ and $P(U,0)$. The variable $q$ is reintroduced in Section~\ref{s.idenL} once the proof of Theorem~\ref{t.mulimit} is complete.

\subsection{The flatness of minimizers}  
\label{ss.flat}
The primary task in the proof of Theorem~\ref{t.mulimit} is to quantify the limit~\eqref{e.monolim}. The main step, which is the focus of this subsection, is to show that, for $n\gg1$, the minimizer $u(\cdot,Q_n)$ is close to a plane. This allows us to compare $\mu$ to $\nu(\cdot,p)$ for an appropriate choice of $p\in\Rd$.

\smallskip

In the first step, we use the finite range of dependence assumption to show that, unless the expectation of~$\mu$ increases significantly when passing to a larger scale, the variance of the average slope vector $P$ must be small (recall that $P$ is defined in~\eqref{e.P}). Since the argument is based on independence, we work with the trimmed cubes. 

\begin{lemma}
\label{l.curl}
There exists $C(d,\Lambda)\geq 1$ such that, for every $n\in\N$,
\begin{equation} \label{e.varP}
\var\left[ P(Q_n^\circ) \right] \leq C \left( \E \left[ \mu(Q_{n+1}^\circ) \right] - \E \left[ \mu(Q_n^\circ) \right] +CK_0^23^{-n} \right).
\end{equation}
\end{lemma}
\begin{proof}
Fix $n\in\N$ and a unit vector $e\in \partial B_1$. Select a smooth vector field $$B:Q_{n+1}^\circ \to \Rd$$ with
\begin{equation*} \label{}
\div B = 0 \quad \mbox{in} \ Q_{n+1}^\circ
\end{equation*}
such that $B$ has compact support in $Q_{n+1}^\circ$ and satisfies $B \equiv e$ in $Q_n^\circ$ and the estimate $|B| \leq C$ in $Q_{n+1}^\circ$. To see that such a solenoidal vector field exists, consider without loss of generality the case $n=0$ the take $h\in H^1((Q_0+B_{1/2})\setminus (Q_0+B_{1/4}))$ to be the solution of the Neumann problem
\begin{equation*}
\left\{ 
\begin{aligned}
& -\Delta h = 0 &  \mbox{in} & \  (Q_0+B_{1/2})\setminus (Q_0+B_{1/4}),\\
& \partial_\nu h = e\cdot \nu & \mbox{on} & \ \partial (Q_0+B_{1/2}),\\
& \partial_\nu h = 0 & \mbox{on} & \ \partial (Q_0+B_{1/4}),
\end{aligned}
\right.
\end{equation*}
where (unlike in the rest of the paper) in the previous line $\nu$ denotes the outer unit normal vector. Then we define a vector field $\tilde B:\Rd \to \Rd$ by
\begin{equation*}
\tilde B(x) := \left\{ 
\begin{aligned}
& e\cdot x &  \mbox{if} & \  x\in Q_0+B_{1/4},\\
& e\cdot x - Dh(x) &  \mbox{if} & \  x\in (Q_0+B_{1/2})\setminus (Q_0+B_{1/4}),\\
& 0 & \mbox{if} & \ x\in \Rd \setminus (Q_0+B_{1/2}).
\end{aligned}
\right.
\end{equation*}
It is clear from the construction that $\tilde B\in L^\infty(\Rd ;\Rd)$ is solenoidal in $\Rd$ and has support in $Q_0+B_{1/2}$. We may then obtain $B$ as above by mollifying $\tilde B$. 

\smallskip

Observe that
\begin{equation} \label{e.divfree}
\fint_{Q_{n+1}^\circ} B(x) \cdot Du\!\left(x,Q^\circ_{n+1} \right)\, dx = 0.
\end{equation}

Let $U$ be the union of the trimmed subcubes of $Q_{n+1}^\circ$ of the form $Q_n^\circ(x) \subseteq Q_{n+1}^\circ$ and set $V:= Q_{n+1}^\circ \setminus \overline U$. Since $U$ and $V$ are disjoint, we have
\begin{equation*} \label{}
u(\cdot,U\cup V) \vert_U = u(\cdot,U) \quad \mbox{and} \quad u(\cdot,U\cup V) \vert_V = u(\cdot,V).
\end{equation*}
Similarly, for each $x\in Q_{n+1}^\circ$, we have
\begin{equation*} \label{}
u(\cdot,U\cup V) \vert_{Q^\circ_n(x)} = u(\cdot,Q^\circ_n(x)).
\end{equation*}
By previous two lines and stationarity, we have $\E \left[ \mu(U) \right] = \E \left[ \mu(Q^\circ_n) \right]$. Thus we may rewrite~\eqref{e.divfree} as
\begin{multline*} \label{}
\int_{U} B(x) \cdot Du(x,U)\,dx  \\
= \int_{U\cup V} B(x) \cdot \left( Du(x,U\cup V) - Du(x,Q_{n+1}^\circ) \right) \, dx - \int_{V} B(x) \cdot Du(x,V)\, dx.
\end{multline*}
Applying~\eqref{e.convexUV} and~\eqref{e.vgradbnd}, using $|B| \leq C$, $|V| \leq C3^{-n}|Q_{n+1}|$ and~\eqref{e.mubounds}, we obtain
\begin{equation} \label{e.circnote}
\left( \fint_{U} B(x) \cdot Du(x,U)\, dx \right)^2 \leq C \left( \mu(Q_{n+1}^\circ) -\mu(U\cup V)  + CK_0^23^{-n}\right).
\end{equation}
Using again that $|V| \leq C3^{-n}|Q_{n+1}|$ with~\eqref{e.mubounds}, we see from~\eqref{e.musuperadd} that 
\begin{equation*} \label{}
\E \left[ \mu(U\cup V) \right] \geq \frac{|U|}{|U\cup V|}\E \left[ \mu(U) \right] - CK_0^23^{-n} \geq  \E \left[ \mu(Q_n^\circ) \right] - CK_0^23^{-n}.
\end{equation*}
Taking the expectation of~\eqref{e.circnote} and using the previous line, we get
\begin{equation} \label{e.circnote2}
\E \left[ \left( \fint_{U} B(x) \cdot Du(x,U)\, dx \right)^2 \right] \leq C \left(\E \left[ \mu(Q_{n+1}^\circ)\right]  - \E \left[ \mu(Q_n^{\circ})\right]  + CK_0^23^{-n}\right).
\end{equation}
Finally, we note that $\dist(Q_n^\circ,U\setminus Q_n^\circ) \geq 1$ by construction and therefore, using~(P2), we see that the random variables 
\begin{multline*} \label{}
\fint_{Q_{n}^\circ} B(x) \cdot Du(x,Q_n^\circ)\, dx \qquad \mbox{and} \qquad \fint_{U \setminus Q_{n}^\circ} B(x) \cdot Du(x,U\setminus Q_n^\circ)\, dx \quad \\ \mbox{are $\P$--independent.}  
\end{multline*}
Therefore, using that $B(x) = e$ in $Q_n^\circ$ as well as 
\begin{equation*} \label{}
u(\cdot,Q_n^\circ) = u(\cdot,U)\vert_{Q_n^\circ} \quad \mbox{and} \quad u(\cdot,U\setminus Q_n^\circ) = u(\cdot,U)\vert_{U\setminus Q_n^\circ},
\end{equation*}
we obtain from independence and~\eqref{e.circnote2} that 
\begin{align*} \label{}
\var\left[ e\cdot P(Q_n^\circ) \right] & =  \var \left[ \fint_{Q_n^\circ} e \cdot Du(x,Q_n^\circ)\, dx \right] \\
& \leq \var \left[ \fint_{Q_n^\circ} e \cdot Du(x,Q_n^\circ)\, dx \right] + \var \left[ \fint_{U\setminus Q_n^\circ} B(x) \cdot Du(x,U\setminus Q_n^\circ)\, dx \right] \\
& = \var \left[  \fint_{U} B(x) \cdot Du(x,U)\, dx \right] \\
& \leq \E \left[ \left( \fint_{U} B(x) \cdot Du(x,U)\, dx \right)^2 \right] \\
& \leq C\left(\E \left[ \mu(Q_{n+1}^\circ)\right]  -\E\left[ \mu(Q_n^{\circ})\right]  + CK_0^23^{-n}\right).
\end{align*}
Summing over~$e$ in the standard basis for~$\Rd$ yields the lemma. 
\end{proof}

Motivated by the previous lemma, we define, for each $n\in \N$, the deterministic slope~$\overline P_n\in \Rd$ at the~$n$th scale by
\begin{equation*} \label{}
\overline P_n := \E \left[ P(Q_n^\circ) \right].
\end{equation*}
We note for future reference that~\eqref{e.vgradbnd} implies
\begin{equation} \label{e.Pnbnds}
\left| \overline P_n \right| \leq 3K_0.
\end{equation}
We may formulate the previous lemma in terms of the untrimmed cubes. Indeed, by~\eqref{e.Emono},~\eqref{e.trimsubject} and~\eqref{e.varP}, we have
\begin{equation} \label{e.flatuntrim}
\E \left[ \left| P(Q_n) - \overline P_n  \right|^2 \right] \leq C \left( \E \left[ \mu(Q_{n+1}) \right] - \E \left[ \mu(Q_n^\circ) \right] +CK_0^23^{-n} \right).
\end{equation}

We next present the key assertion regarding the flatness of minimizers. It states that, if $\E[ \mu(Q_{n+1})] - \E[ \mu(Q_n^\circ)]$ is small, then we can construct a candidate for the minimizer of $\mu$ on an arbitrarily large scale which is very close to a plane of slope~$\overline P_n$ and has expected energy not much more than~$\E \left[ \mu(Q_n) \right]$. By modifying the latter a little so that it has affine boundary conditions, we get a minimizer candidate for $\nu(Q_{2n}^\circ,\overline P_n)$, which gives an upper bound for $\E[\nu(Q_{2n}^\circ,\overline P_n)]$ in terms of $\E[\mu(Q_n)]$. The argument uses Lemma~\ref{l.curl} and a patching construction.

\begin{lemma}
\label{l.patching}
There exists a constant $C(d,\Lambda)>0$ and a universal $\alpha >0$ such that, for every $n\in \N$,
\begin{equation} \label{e.patching}
\E \left[ \nu(Q_{2n}^\circ,\overline P_n) \right]  \leq \E \left[ \mu(Q_n) \right] + C \left( \E \left[ \mu(Q_{n+1}) \right] - \E \left[ \mu(Q_n^\circ) \right] + CK_0^23^{-\alpha n} \right). 
\end{equation}
\end{lemma}
\begin{proof}
Throughout, we fix $n\in\N$, let $C$ denote a positive constant depending only on $(d,\Lambda)$ which may vary in each occurrence, and denote
\begin{equation*} \label{}
\tau_n := \left( \E \left[ \mu(Q_{n+1}) \right] - \E \left[ \mu(Q_n^\circ) \right] + CK_0^23^{-n/60} \right).
\end{equation*}
To prove the lemma, it suffices to construct $v\in H^1_0(Q^\circ_{2n})$ satisfying 
\begin{equation} 
\label{e.exhibition}
\E \left[ \fint_{Q^\circ_{2n}} L \left( \overline P_n+ D  v(x),x\right) \, dx \right] \\
\leq \E \left[ \mu( Q_n) \right]   + C\tau_n.
\end{equation}
In the first step we give the construction of $v$. The rest of the argument is concerned with the proof of~\eqref{e.exhibition}.

\smallskip

\emph{Step 1.} We construct the candidate minimizer $v\in H^1_0(Q^\circ_{2n})$. We build~$v$ by patching the minimizers for $\mu$ on the  family $\left\{ z+ Q_{n+1}\,:\, z\in 3^n\Zd\right\}$ of overlapping triadic cubes.  We consider a partition of unity subordinate to this family by denoting, for each $z\in 3^n\Zd$,
\begin{equation*}
\psi(y):= \int_{ Q_{n}} \psi_0(y-x)\,dx,
\end{equation*}
where $\psi_0\in C^\infty_c(\Rd)$ is a smooth function satisfying
\begin{equation*}
0\leq \psi_0\leq C3^{-dn}, \quad \int_{\Rd} \psi_0(x)\,dx=1, \quad |D  \psi_0|\leq C3^{-(d+1)n}, \quad \supp \psi_0 \subseteq  Q_n.
\end{equation*}
It follows then that $\psi$ is smooth and supported in $ Q_{n+1}$, satisfies $0\leq \psi\leq 1$, and the translates of $\psi$ form a partition of unity:
\begin{equation}\label{e.partunity}
\sum_{z\in3^n\Zd} \psi(x-z) = 1.
\end{equation}
Moreover, we have 
\begin{equation} \label{e.psiderv}
\sup_{x\in\Rd} \left| D \psi(x) \right| \leq C3^{-n}.
\end{equation}
We next two smooth cutoff functions $\xi,\zeta \in C^\infty_c(Q_{2n})$ satisfying
\begin{multline} \label{e.xicutoff}
0\leq \xi\leq 1, \quad \xi\equiv 1 \ \mbox{on} \ \left\{ x\in Q_{2n}\,:\, \dist(x,\partial  Q_{2n}) > 3^{2n/(1+\delta)} \right\}, \\
\xi \equiv 0 \ \mbox{on} \ \left\{ x \in z+Q_n \,:\, z\in 3^n\Zd,\ z+Q_{n+1} \not\subseteq Q_{2n}^\circ  \right\}, \quad |D \xi|\leq C3^{-2n/(1+\delta)},
\end{multline}
where $\delta \in (0,\beta]$ will be selected below in Step~6, and
\begin{multline} \label{e.zetacutoff}
 0\leq \zeta \leq 1, \quad \zeta \equiv 1 \ \mbox{on} \  \left\{x \in  z+Q_{n} \,:\, z\in 3^n\Zd, \ z+Q_{n+3} \subseteq Q_{2n}  \right\}, \\
 \zeta \equiv 0 \ \mbox{on} \  \left\{ x \in z+Q_{n+1} \,:\, z\in 3^n\Zd, \ z+Q_{n+1} \not\subseteq Q_{2n} \right\}, \quad |D\zeta|\leq C3^{-n}.
\end{multline} 
\smallskip

To construct $v$, we first define a vector field $\mathbf{f} \in L^2 (\Rd;\Rd)$ by 
\begin{equation} \label{e.def.f}
\mathbf{f} (x):= \zeta(x) \sum_{z\in3^n\Zd} \psi(x-z) \left( D  u(x,z+Q_{n+1}) - \overline P_n  \right).
\end{equation}
Since~$\mathbf{f}$ is not necessarily the gradient of an $H^1$ function, due to the errors made by introducing the partition of unity and the cutoff function, we need to take its Helmholtz-Hodge projection. We may write 
\begin{equation} \label{e.helmholtz}
\mathbf{f} = \overline{\mathbf{f}}+ D w - \div\mathbf{S} \quad \mbox{in} \ Q_{2n},
\end{equation}
where 
\begin{equation*}
\overline{\mathbf{f}}:= \fint_{Q_{2n}} \mathbf{f}(x)\,dx,
\end{equation*}
$w\in H^1_{\mathrm{loc}}(\Rd)$ is defined as the unique solution of 
\begin{equation*} \label{}
\left\{ 
\begin{aligned}
& -\Delta w = -\div{\mathbf{f}} \ \  \mbox{in}  \ Q_{2n}, \\
& \fint_{Q_{2n}} w(x)\,dx =0, \\
& w \ \   \mbox{is \ $Q_{2n}$--periodic},
\end{aligned}
\right.
\end{equation*}
and $\mathbf{S}$ is a skew-symmetric matrix with entries $S_{ij} \in H^1_{\mathrm{loc}}(\Rd)$ uniquely determined (up to a constant) by 
\begin{equation*} \label{}
\left\{ 
\begin{aligned}
& -\Delta S_{ij} = \partial_j f_i - \partial_i f_j \quad \mbox{in} \ Q_{2n}, \\
& S_{ij}\ \   \mbox{is \ $Q_{2n}$--periodic}.
\end{aligned}
\right.
\end{equation*}
Here $f_i$ is the $i$th entry of~$\mathbf{f}$ and $\div\mathbf{S}$ is the vector field with entries $\sum_{j=1}^d \partial_j S_{ij}$. Indeed,
one may check via a straightforward computation that each component of the vector field $\mathbf{f} - D  w + \div\mathbf{S}$ is harmonic and therefore constant by periodicity. This constant must be $\overline {\mathbf{f}}$ since $D  w$ and $\div\mathbf{S}$ have zero mean in~$Q_{2n}$. This confirms~\eqref{e.helmholtz}. 

\smallskip

We define $v\in H^1_0(Q^\circ_{2n})$ by cutting off $w$:
\begin{equation*} \label{}
v(x):= \xi(x) w(x), \quad x\in Q_{2n}.
\end{equation*}
The cutoff function~$\xi$ is supported in $Q^\circ_{2n}$ and thus~$v\in H^1_0(Q^\circ_{2n})$.

\smallskip

The rest of the proof concerns the derivation of~\eqref{e.exhibition}. This is accomplished by showing that $D  v$ is expected to be close to $\mathbf{f}$ in $L^2(Q^\circ_{2n})$ due to the fact that~$w$,~$\overline{\mathbf{f}}$ and~$\div\mathbf{S}$ each have a small expected~$L^2$ norm.

\smallskip

\emph{Step 2.} We show that, for every $z\in3^n\Zd \cap Q_{2n}$,
\begin{equation}\label{e.gradlocalize}
\E \bigg[ \fint_{z+Q_n} \Big(  \left| \mathbf{f}(x) - \zeta(x) \left( D  u(x,z+Q_{n+1}) - \overline P_n \right) \right|^2    \bigg] \leq  C \tau_n.
\end{equation}
By Lemma~\ref{l.convexL2} we have, for every $z\in3^n\Zd$,
\begin{multline*}
\sum_{y\in  \{ -3^n,0,3^n\}^d} \fint_{z+y+Q_n} \left| Du(x,z+Q_{n+1}) - Du(x,z+y+ Q_{n}) \right|^2 \, dx \\
\begin{aligned}
& \leq C \sum_{y\in  \{ -3^n,0,3^n\}^d} \left( \fint_{z+y+Q_n} L\left( Du(x,z+Q_{n+1}),x\right)\, dx - \mu(y+z+Q_n) \right) \\
& = C \left( \mu(z+Q_{n+1}) - \sum_{y\in  \{ -3^n,0,3^n\}^d} \mu(z+y+Q_n) \right).
\end{aligned}
\end{multline*}
Taking expectations and using the triangle inequality, we get, for every $z\in 3^n\Zd$ and $y\in \{ -3^n,0,3^n\}^d$,
\begin{equation}\label{e.gridtrap}
\E \left[ \fint_{z+Q_n} \left| Du(x,z+y+Q_{n+1}) - D u(x,z+Q_{n+1}) \right|^2 \, dx \right] 
\leq C \tau_n.
\end{equation}
Now observe that~\eqref{e.gradlocalize} follows from the previous inequality and the fact that, for every $z\in 3^n\Zd$ and $x\in z+Q_{n}$,
\begin{multline*}
\mathbf{f}(x) - \zeta(x) \left( D  u(x,z+Q_{n+1}) - \overline P_n \right) \\ 
= \zeta(x) \sum_{y\in \{ -3^n,0,3^n\}^d} \psi(x-y) \left( D  u(x,z+y+Q_{n+1}) - D  u(x,z+Q_{n+1}) \right).
\end{multline*}

\smallskip

\emph{Step 3.} We show that 
\begin{equation}\label{e.whipfbar}
 \E \left[  \left| \overline{\mathbf{f}}\right|^2  \right]   \leq C\tau_n.
\end{equation}
In view of~\eqref{e.zetacutoff}, it is convenient to denote
\begin{equation} \label{e.Zn}
\mathcal Z_n:= \left\{ z\in 3^n\Zd\,:\, z+Q_{n+3} \subseteq Q_{2n} \right\}. 
\end{equation}
Observe that $( 3^n\Zd \cap Q_{2n} ) \setminus \mathcal Z_n$ has $C 3^{n(d-1)}$ elements. 
By~\eqref{e.flatuntrim},~\eqref{e.gradlocalize} and~\eqref{e.vgradbnd}, we have
\begin{align*}
\E \left[ \left|\overline{\mathbf{f}} \right|^2 \right] & \leq \E \left[ 3^{-nd} \sum_{z\in 3^n\Zd \cap Q_{2n}} \left| \fint_{z+Q_n} \mathbf{f}(x)\,dx  \right|^2 \right]  \\
& \leq 2 \E \left[ 3^{-nd} \sum_{z\in 3^n\Zd \cap Q_{2n}} \left| \fint_{z+Q_n} \zeta(x)\left( D  u(x,z+Q_{n+1})\,dx-\overline P_n \right)  \right|^2 \right]  + C\tau_n  \\
& \leq 2\E \left[ 3^{-nd} \sum_{z\in \mathcal{Z}_n} \left| \fint_{z+Q_n}  D  u(x,z+Q_{n+1})\,dx-\overline P_n  \right|^2 \right] + CK_0^23^{-n} + C\tau_n \\
& \leq C \tau_n.
\end{align*}

\smallskip

\emph{Step 4.} We show that 
\begin{equation} \label{e.boundonw2}
\E \left[ 3^{-4n} \fint_{Q_{2n}} \left| w(x) \right|^2\,dx \right] \leq C K_0^23^{-n/2}.
\end{equation}
Let $\phi \in H^2_{\mathrm{loc}}(\Rd)$ denote the unique solution of
\begin{equation*} \label{}
\left\{ 
\begin{aligned}
& -\Delta \phi = w \ \  \mbox{in}  \ \Rd, \\
& \fint_{Q_{2n}} \phi(x)\,dx =0, \\
& \phi \ \  \mbox{is $Q_{2n}$--periodic}.
\end{aligned}
\right.
\end{equation*}
Integrating by parts, we have
\begin{multline}\label{e.softH2}
\int_{Q_{2n}} \left| D^2\phi(x) \right|^2\,dx = \int_{Q_{2n}} \left| w(x) \right|^2\,dx = \int_{Q_{2n}} D  \phi(x) \cdot \mathbf f(x)  \, dx \\ 
= \int_{Q_{2n}} D  \phi(x) \cdot \left( \mathbf f(x) - \E \left[ \int_{Q_n} \mathbf{f}(x)\,dx \right]  \right)  \, dx .
\end{multline}
We need a second mesoscale, given by an integer $k \in (n,2n)$ to be selected below. In what follows, we denote $ (D  \phi)_z:= \fint_{z+Q_k} D  \phi(x)\,dx$ and $\sum_z = \sum_{z \in3^k\Zd\cap Q_{2n}}$ as well as $\tilde{ \mathbf{f}}:=\mathbf{f}- \E \left[ \int_{Q_n} \mathbf{f}(x)\,dx \right] $. To estimate $\| w \|_{L^2(Q_{2n})}$, we use the following splitting:
\begin{align*}
\int_{Q_{2n}} \left| w(x) \right|^2\,dx & = \int_{Q_{2n}}  D  \phi(x) \cdot \tilde{\mathbf {f}}(x) \, dx \\
& = \sum_z \left( \int_{z+Q_k} \left( D  \phi(x)    - (D  \phi)_z  \right) \cdot  \tilde{\mathbf {f}}(x) \,dx +  (D  \phi)_z \cdot \int_{z+Q_k} \tilde{\mathbf {f}}(x)\,dx \right).
\end{align*}
To estimate the first term in the sum on the right side of the previous inequality, we use the Poincar\'e inequality, the discrete H\"older inequality,~\eqref{e.softH2} and~\eqref{e.vgradbnd}, and then Young's inequality:
\begin{align*}
\lefteqn{\sum_z  \int_{z+Q_k} \left( D  \phi(x)    - (D  \phi)_z  \right) \cdot  \tilde{\mathbf {f}}(x) \,dx  } \qquad & \\
& \leq  \sum_z 3^k \left( \int_{z+Q_n} \left| D^2  \phi(x) \right|^2\,dx \right)^{\frac12} \left( \int_{z+Q_n}  \left|  \tilde{\mathbf {f}}(x) \right|^2\,dx \right)^{\frac12} \\
& \leq 3^k \left(\sum_z \int_{z+Q_n} \left| D^2  \phi(x) \right|^2\,dx \right)^{\frac12} \left(\sum_z \int_{z+Q_n}  \left|  \tilde{\mathbf {f}}(x) \right|^2\,dx\right)^{\frac12} \\
& = 3^k \left( \int_{Q_{2n}} \left| D^2  \phi(x) \right|^2 \,dx \right)^{\frac12} \left( \int_{Q_{2n}} \left|  \tilde{\mathbf {f}}(x)\right|^2 \,dx \right)^{\frac12} \\
& \leq 3^k \left( \int_{Q_{2n}} \left| w(x) \right|^2\,dx \right)^{\frac12} \left( CK_0^2\left| Q_{2n} \right|\right)^{\frac12} \\
& \leq \frac14 \int_{Q_{2n}} \left| w(x) \right|^2\,dx  + CK_0^2\left| Q_{2n} \right| 3^{2k}.
\end{align*}
We next estimate the expectation of the second term in the sum, using two different forms of H\"older's inequality:
\begin{align}\label{e.holderapps}
\lefteqn{\E \left[ \sum_z (D  \phi)_z \cdot \int_{z+Q_k} \tilde{\mathbf{f}}(x)\,dx \right] } \qquad & \\
& = \sum_z \E \left[ (D  \phi)_z \cdot \int_{z+Q_k} \tilde{\mathbf{f}}(x)\,dx \right] \notag \\
& \leq \sum_z \E \left[ \left| \left( D  \phi \right)_z \right|^2  \right]^{\frac12} \E \left[ \left( \int_{z+Q_k} \tilde{\mathbf{f}}(x)\,dx \right)^2 \right]^{\frac12}\notag  \\
& \leq \left( \sum_z \E \left[ \left| \left( D  \phi \right)_z \right|^2  \right]  \right)^{\frac12} \left(\sum_z \E \left[ \left( \int_{z+Q_k} \tilde{\mathbf{f}}(x)\,dx \right)^2 \right]  \right)^{\frac12}.  \notag
\end{align}
For the first factor on the right side of the previous inequality, we have, by the Poincar\'e inequality and~\eqref{e.softH2}, 
\begin{align}\label{e.ingred1}
 \sum_z \E \left[ \left| \left( D  \phi \right)_z \right|^2  \right]  = \E \left[ \sum_z \left| \left( D  \phi \right)_z \right|^2 \right] & \leq  C3^{-kd}\,\E \left[ \int_{Q_{2n}} \left| D  \phi(x) \right|^2\,dx \right]  \\
& \leq C3^{-kd+4n}\,\E \left[ \int_{Q_{2n}} \left| D^2 \phi(x) \right|^2\,dx \right] \notag\\
& =  C3^{-kd+4n}\,\E \left[ \int_{Q_{2n}} \left| w(x) \right|^2\,dx \right].\notag
\end{align}
In preparation to estimate the second factor, we use independence to get
\begin{align*}
\E \left[ \left( \int_{Q_k} \tilde{\mathbf{f}}(x)\,dx \right)^2 \right] & = \E\left[ \sum_{y,y'\in 3^n\Zd\cap Q_k}  \int_{y+Q_n} \tilde{\mathbf{f}}(x)\,dx \int_{y'+Q_n} \tilde{\mathbf{f}}(x)\,dx\right] \\
& \leq  C  \left( CK_0 \left| Q_n \right| \right)^2 \left| 3^n\Zd \cap Q_k \right| = CK_0^2 3^{d(k+n)}.
\end{align*}
By stationarity, the same estimate holds with $z+Q_k$ in place of $Q_k$ provided that the cube $z+Q_k$ does not touch $\partial Q_{2n}$. For the cubes which do touch the  boundary of the macroscopic cube (and thus intersect the support of~$\zeta$), we use the following cruder, deterministic bound given by~\eqref{e.vgradbnd}:
\begin{equation*}
\left( \int_{Q_k}  \tilde{\mathbf{f}}(x) \,dx \right)^2 \leq  C3^{dk} \int_{Q_k}  \left| \tilde{\mathbf{f}}(x) \right|^2 \,dx \leq C K_0^2 3^{2dk}.
\end{equation*}
Combining these, using that there are at most $C3^{(2n-k)(d-1)}$ cubes of the form $z+Q_k$ which touch the boundary of $Q_n$, we get 
\begin{equation}\label{e.ingred2}
\sum_z \E \left[ \left( \int_{z+Q_k} \tilde{\mathbf{f}}(x)\,dx \right)^2 \right] \leq C K_0^2 3^{2dn +dk} \left( 3^{-d(k-n)} + 3^{- (2n-k)} \right).
\end{equation}
We may now estimate the right side of~\eqref{e.holderapps} using applying~\eqref{e.ingred1},~\eqref{e.ingred2} and Young's inequality. The result is:
\begin{multline}
\E \left[ \sum_z (D  \phi)_z \cdot \int_{z+Q_k} \tilde{\mathbf{f}}(x)\,dx \right] \\ \leq \frac14 \E \left[ \int_{Q_{2n}} \left| w(x) \right|^2\,dx \right] + K_0^2 \left| Q_{2n}\right| 3^{4n}\left( C' 3^{-d (k-n)} + C3^{-(2n-k)} \right) 
\end{multline}
Combining the above inequalities now yields
\begin{equation*}
\E \left[ \int_{Q_{2n}} \left| w(x) \right|^2\,dx \right]  \leq K_0^2 3^{4n} \left| Q_{2n} \right| \left( C'3^{-d (k-n)} + C3^{-(2n-k)} \right).
\end{equation*}
Taking finally $k$ to be the nearest integer to $3n/2$, we obtain~\eqref{e.boundonw2}.

\smallskip

\emph{Step 5.} We estimate the expected size of $\left| \div\mathbf{S} \right|^2$. The claim is
\begin{equation}\label{e.whipsolenoidal}
\E \left[ \fint_{Q_{2n}} \left|\div\mathbf{S}(x) \right|^2\,dx \right] \leq C\tau_n.
\end{equation}
As in the previous step, we use the abbreviations $u_z:= u(\cdot,z+Q_{n+1})$, $\psi_z:= \psi(\cdot-z)$, $\sum_z:= \sum_{z\in 3^n\Zd \cap Q_{2n} }$ and $\int_z:= \int_{z+Q_{n+1}}$
Observe that, in the sense of distributions, for every $i,j\in\{1,\ldots,d\}$, we have
\begin{multline*}
\partial_j f_i - \partial_i f_j 
= \sum_{z} \zeta \left(\partial_j\psi_z ( \partial_iu_z- \overline P_{n,i})  - \partial_i\psi_z (\partial_ju_z -\overline P_{n,j})\right) \\
+ \sum_z \psi_z \left(\partial_j\zeta (\partial_iu_z -\overline P_{n,i}))  - \partial_i\zeta (\partial_ju_z -\overline P_{n,j})\right) \quad \mbox{in} \ Q_{2n}.
\end{multline*}
The right side belongs to $L^2(Q_{2n})$, thus  $\partial_j f_i - \partial_i f_j\in L^2(Q_{2n})$ and $S_{ij}\in H^2_{\mathrm{loc}} (\Rd)$. Using the fact that, for every $x\in\Rd$,
\begin{equation} \label{}
\label{e.sillyidentity}
\sum_z \zeta(x) D\psi_z(x) = 0,
\end{equation}
we may also express the previous identity slightly differently as
\begin{multline}
\label{e.curlfL2}
\left( \partial_j f_i - \partial_i f_j \right)(x)
 = \sum_{z} \zeta(x) \left[D  \psi_z(x), D  u_z(x) - \overline P_n - \mathbf{f}(x) \right]_{ij}  \\
+ \sum_z \psi_z \left[ D  \zeta, D  u_z  -\overline P_{n} \right]_{ij} \quad \mbox{in} \ Q_{2n},
\end{multline}
where we henceforth use the notation 
\begin{equation*}
\left[ \mathbf{v},\mathbf{w} \right]_{ij} := v_j w_i - v_iw_j
\end{equation*}
for indices $i,j\in\{ 1,\ldots,d\}$ and vectors $\mathbf{v},\mathbf{w}\in\Rd$ with entries $(v_i)$ and $(w_i)$, respectively. 
Next we define, for each~$i \in \{ 1,\ldots,d\}$,
\begin{equation*}
\sigma_i := -\left( \div\mathbf{S} \right)_i = -\sum_{j=1}^d \partial_jS_{ij}.
\end{equation*}
It is evident that $\sigma_i \in H^1_{\mathrm{per}}(Q_{2n})$ and $\sigma_i$ is a solution of the equation
\begin{equation*}
-\Delta \sigma_i = -\sum_{j=1}^d \partial_j \left( \partial_j f_i - \partial_i f_j \right) \quad \mbox{in} \ Q_{2n}.
\end{equation*}
Since $\sigma_i$ has zero mean in $Q_{2n}$, there exists~$\rho_i \in H^3_{\mathrm{loc}} (\Rd)$, which is unique up to an additive constant, satisfying 
\begin{equation*} \label{}
\left\{ 
\begin{aligned}
& -\Delta \rho_i = \sigma_i \quad \mbox{in} \ \Rd, \\
& \rho_i \ \   \mbox{is \ $Q_{2n}$--periodic}.
\end{aligned}
\right.
\end{equation*}
We have the identities 
\begin{equation} \label{e.H2equal}
\int_{Q_{2n}} \left| D^2  \rho_i(x) \right|^2\,dx =  \int_{Q_{2n}} \left| \sigma_i(x)\right|^2\,dx =  \int_{Q_{2n}} D  \rho_i(x) \cdot D  \sigma_i(x)\,dx.
\end{equation}
Integrating by parts and using the equation for $\sigma_i$, we obtain
\begin{equation*}
\int_{Q_{2n}} \left| \sigma_i(x)\right|^2\,dx = \int_{Q_{2n}} D  \rho_i(x) \cdot D  \sigma_i(x)\,dx = \sum_{j = 1}^d \int_{Q_{2n}} \partial_j \rho_i(x)\left( \partial_jf_i - \partial_i f_j \right)\,dx.
\end{equation*}
To further shorten the notation, in each of the following expressions we keep the sum over~$j$ implicit (note that $i$ is not summed over even though it is repeated) and set $\left( \partial_j \rho_i \right)_z:= \fint_{z+Q_{n+1}} \partial_j\rho_i(x) \,dx$. Continuing then the computation by substituting~\eqref{e.curlfL2}, we obtain
\begin{align*}
\lefteqn{ \int_{Q_{2n}} \partial_j \rho_i(x)\left( \partial_jf_i - \partial_i f_j \right)\,dx } \qquad & \\
& =  \sum_z \int_{z} \left( \partial_j \rho_i(x)- \left( \partial_j \rho_i \right)_z \right) \zeta(x) \left[D  \psi_z(x), D  u_z(x) - \overline P_n - \mathbf{f}(x) \right]_{ij} \,dx \\
& \quad +  \sum_z \left( \partial_j \rho_i \right)_z  \int_{z} \zeta(x)\left[D  \psi_z(x), D  u_z(x) - \overline P_n - \mathbf{f}(x) \right]_{ij}\,dx \\
& \quad +  \sum_z\int_{z} \partial_j \rho_i(x)  \psi_z(x) \left[D  \zeta(x), D  u_z(x) - \overline P_n  \right]_{ij} \,dx.
\end{align*}
We put the second sum on the right side into a more convenient form via \eqref{e.def.f}, \eqref{e.partunity}, integration by parts and~\eqref{e.sillyidentity}: 
\begin{align*}
\lefteqn{ \sum_z \left( \partial_j \rho_i \right)_z  \int_{z} \zeta(x)\left[D  \psi_z(x), D  u_z(x) - \overline P_n - \mathbf{f}(x) \right]_{ij}\,dx} \quad & \\
& = \sum_{y,z}  \left( \partial_j \rho_i \right)_z \int_z \zeta(x) \psi_y(x) \left[D  \psi_z(x), D  u_z(x) - \overline P_n - \zeta(x)( D  u_y(x) - \overline P_{n} ) \right]_{ij}\,dx \\
& = \sum_{y,z}  \left( \left( \partial_j \rho_i \right)_z - \left( \partial_j \rho_i \right)_y \right) \int_z -(\zeta(x))^2 \psi_y(x) \left[D  \psi_z(x), D  u_y(x) - \overline P_n  \right]_{ij}\,dx \\
& \qquad + \sum_{z}  \left( \partial_j \rho_i \right)_z \int_z - \psi_z(x) \left[D  \zeta(x), D  u_z(x) - \overline P_n  \right]_{ij}\,dx.
\end{align*}
Combining this with the previous identity, we get 
\begin{align}
\lefteqn{\int_{Q_{2n}} \left| \sigma_i(x)\right|^2\,dx } \quad & \label{e.splitintothree} \\
& =  \sum_z \int_{z} \left( \partial_j \rho_i(x)- \left( \partial_j \rho_i \right)_z \right) \zeta(x) \left[D  \psi_z(x), D  u_z(x) - \overline P_n - \mathbf{f}(x) \right]_{ij} \,dx \notag \\
& \quad + \sum_{y,z}  \left( \left( \partial_j \rho_i \right)_y - \left( \partial_j \rho_i \right)_z \right) \int_z (\zeta(x))^2 \psi_y(x) \left[D  \psi_z(x), D  u_y(x) - \overline P_n  \right]_{ij}\,dx  \notag \\
&  \quad + \sum_{z}  \int_z \left( \partial_j \rho_i (x)- \left( \partial_j \rho_i \right)_z \right)  \psi_z(x) \left[D  \zeta(x), D  u_z(x) - \overline P_n  \right]_{ij}\,dx. \notag 
\end{align}
We now proceed to estimate the three terms on the right side of~\eqref{e.splitintothree}. For the first sum, we use~\eqref{e.psiderv}, the H\"older, discrete H\"older and Poincar\'e inequalities,~\eqref{e.H2equal} and Young's inequality to get
\begin{align*}
\lefteqn{\sum_{z} \int_{z} \left( \partial_j \rho_i(x)    - (\partial_j \rho_i)_z  \right) \zeta(x) \left[ D \psi_z(x) ,  Du_z(x) - \overline P_n- \mathbf{f}(x)  \right]_{ij}  \,dx } \qquad & \\
& \leq C \left(  \sum_z  \int_z \left| D^2 \rho_i(x) \right|^2\,dx  \right)^{\frac12}  \left(  \sum_z  \int_z \left| Du_z(x) - \overline P_n- \mathbf{f}(x) \right|^2\,dx  \right)^{\frac12} \\
& \leq \frac14 \int_{Q_{2n}} \left| \sigma_i(x) \right|^2\,dx + C\sum_z \int_z \left| Du_z(x) - \overline P_n- \mathbf{f}(x) \right|^2\,dx.
\end{align*}
Taking expectations and using~\eqref{e.gradlocalize}, we obtain
\begin{multline} \label{e.firstsumnailedk}
\E \left[ \sum_z \int_{z} \left( \partial_j \rho_i(x)- \left( \partial_j \rho_i \right)_z \right) \zeta(x) \left[D  \psi_z(x), D  u_z(x) - \overline P_n - \mathbf{f}(x) \right]_{ij} \,dx \right] \\
 \leq \frac14 \E \left[ \int_{Q_{2n}} \left| \sigma_i(x) \right|^2\,dx \right] + C\left| Q_{2n} \right| \tau_n.
\end{multline}
For the second sum, we notice that each entry vanishes unless $y \in z+Q_{n+2}$, there are at most $C$ such entries $y$ in the sum for any given entry $z$, and for such $y$ and $z$, the Poincar\'e inequality gives
\begin{equation*} \label{}
 \left|  \left( \partial_j \rho_i\right)_y - \left( \partial_j \rho_i\right)_z \right|^2 \leq C3^{2n}  |Q_n|^{-1}  \int_{z+Q_{n+3}} \left|D^2 \rho_i(x) \right|^2\,dx.
\end{equation*}
Using this,~\eqref{e.psiderv},~\eqref{e.H2equal} and~\eqref{e.mubounds}, the H\"older and Young inequalities and the fact that $3^n\Zd\cap Q_{2n}$ has $C3^{nd}$ elements, we get 
\begin{align*}
\lefteqn{ \sum_{y,z} \left(   \left( \partial_j \rho_i\right)_y - \left( \partial_j \rho_i\right)_z \right) \int_{z} (\zeta(x))^2 \psi_y(x) \left[  D \psi_z(x) ,  Du_y(x) -\overline P_n  \right]_{ij}  \,dx  } \qquad & \\
& \leq C \left(  3^{2n}|Q_{n}|^{-1} \sum_{z} \int_{z+Q_{n+3}} \left| D^2 \phi_i(x) \right|^2\,dx \right)^{\frac12}\left(  \sum_{z} 3^{-2n} CK_0^2|Q_n| \right)^{\frac12} \\
& \leq \frac14 \int_{Q_{2n}} \left| \sigma_i(x) \right|^2\,dx + C K_0^23^{nd}.
\end{align*}
For the third sum on the right side of~\eqref{e.splitintothree}, we proceed in almost the same way as for the first two, except that rather than use~\eqref{e.psiderv} we use the estimate for~$D \zeta$ in~\eqref{e.zetacutoff} and the fact that $D \zeta$ vanishes except if $z\not\in \mathcal Z_n$ (recall that $\mathcal Z_n$ is defined in~\eqref{e.Zn}) and there are at most $C3^{n(d-1)}$ such elements in the sum. We obtain:
\begin{align*}
\lefteqn{ \sum_z \int_{z} \left( \partial_j\rho_i(x)- \left( \partial_j \rho_i\right)_z \right)\psi_z(x) \left[ D \zeta(x) ,  Du_z(x) - \overline P_n\right]_{ij} \,dx } \qquad & \\
& \leq C \left(  3^{2n} \sum_{z} \int_{z} \left| D^2 \phi_i(x) \right|^2\,dx \right)^{\frac12}\left(  \sum_{z\notin\mathcal Z_n} 3^{-2n} CK_0^2|Q_n| \right)^{\frac12} \\
& \leq \frac14 \int_{Q_{2n}} \left| \sigma_i (x) \right|^2\,dx + C K_0^2 3^{n(2d - 1)}.
\end{align*}
Combining the previous two inequalities with~\eqref{e.H2equal}, ~\eqref{e.splitintothree}  and~\eqref{e.firstsumnailedk} yields
\begin{equation*}
\E \left[  \int_{Q_{2n}} \left| \sigma_i (x) \right|^2\,dx \right] \leq CK_0^2 3^{n(2d-1)} + C\left| Q_{2n} \right| \tau_n.
\end{equation*}
Dividing by~$\left| Q_{2n} \right|$ gives~\eqref{e.whipsolenoidal}.

\smallskip

\emph{Step 6.} We show that the effect of the cutoff $\xi$ in the definitions of $v$ and $\mathbf{h}$ is expected to be small: precisely,
\begin{equation}\label{e.closenvnw}
\E \left[ \fint_{Q_{2n}} \left| D  v(x) - D  w(x) \right|^2\,dx \right] 
\leq  C\tau_n. 
\end{equation}
We use the identity
\begin{equation*}
D  v(x) - D  w(x) = w(x) D  \xi(x) + (\xi(x)-1) D  w(x)
\end{equation*}
and~\eqref{e.xicutoff} to obtain
\begin{multline} \label{e.nvnwsplit}
\fint_{Q_{2n}} \left| D  v(x) - D  w(x) \right|^2\,dx \\
 \leq C3^{-4n/(1+\delta)} \fint_{Q_{2n}} \left| w(x) \right|^2\,dx + C\fint_{Q_{2n}} \left| \xi(x) - 1 \right|^2 \left| D  w(x) \right|^2\,dx.
\end{multline}
The expectation of the first integral on the right side is controlled by~\eqref{e.boundonw2}:
\begin{equation*}
\E \left[3^{-4n/(1+\delta)} \fint_{Q_{2n}} \left| w(x) \right|^2\,dx \right] \leq CK_0^2 3^{-4n/(1+\delta) + 7n/2}  \leq C 3^{-n/60} \leq C\tau_n,
\end{equation*}
where we have defined
$$\delta := \frac1{14}.$$ For the expectation of the second integral on the right side of~\eqref{e.nvnwsplit}, we recall from~\eqref{e.xicutoff} that $\xi \equiv 1$ except in $$D:= \left\{ x\in Q_{2n} \,:\, \dist(x,\partial Q_{2n}) > C3^{2n/(1+\delta)} \right\}.$$ Therefore, using that $D$ intersects at most $C3^{n(d-2\delta/(1+\delta))}$ subcubes of the form $z+Q_{n+1}$, with $z\in 3^n\Zd$, and applying~\eqref{e.vgradbnd},~\eqref{e.whipfbar} and~\eqref{e.whipsolenoidal}, we obtain
\begin{align} \label{e.cutboundarysbcs}
\lefteqn{ \E \left[\fint_{Q_{2n}} \left| \xi(x) - 1 \right|^2 \left| D  w(x) \right|^2\,dx \right]  \leq \frac1{\left| Q_{2n} \right|} \E\left[ \int_{D} \left| D  w(x) \right|^2\,dx\right] } \qquad & \\
& \leq \frac{C}{\left| Q_{2n} \right|} \E \left[ \int_{D} \left| \mathbf{f}(x) \right|^2\,dx + \int_{Q_{2n}} \left| \mathbf{f}(x) - D  w(x) \right|^2\,dx  \right] \notag \\
& \leq \frac{C}{\left| Q_{2n} \right|} \left( CK_0^2 3^{n(d-2\delta/(1+\delta))}\left| Q_{n} \right| + C| Q_{2n} |\tau_n  \right)     \leq C  \tau_n. \notag
\end{align}
Combining the previous inequality with~\eqref{e.boundonw2} and~\eqref{e.nvnwsplit}, we obtain the desired estimate for the first term on the left of~\eqref{e.closenvnw}. 

\smallskip

\emph{Step 7.}
We estimate the expected difference in $L^2(z+Q_{n})$ between $Dv$ and $Du(\cdot,z+Q_{n+1})$ for each $z \in 3^n\Zd \cap Q^\circ_{2n} $. The claim is that 
\begin{equation} \label{e.patchcomp}
\E \left[ 3^{-dn} \sum_{z \in 3^n\Zd \cap Q^\circ_{2n} }  \fint_{z+Q_n}  \left| D  v(x) - D  u(x,z+Q_{n+1}) +\overline P_n \right|^2\, dx\right] \leq C\tau_n.
\end{equation}
Indeed, for each $z \in  3^n\Zd\cap Q^\circ_{2n}$ and $x\in z+Q_n$, we have 
\begin{multline*} \label{}
D  v(x) - D  u(x,z+Q_{n+1}) +\overline P_n \\
= \left( D  v(x) - D  w(x) \right) + \left( \mathbf{f}(x) -D  u(x,z+Q_{n+1}) + \overline P_n \right) - \left( \overline{\mathbf{f}} - \div\mathbf{S} \right).
\end{multline*}
The desired estimate is a consequence of the previous inequality,~\eqref{e.gradlocalize} (note that $\zeta \equiv 1$ on $z+Q_{n+1}$ for every $z\in 3^n\Zd \cap Q^\circ_{2n}$),~\eqref{e.whipfbar},~\eqref{e.whipsolenoidal} and~\eqref{e.closenvnw}.

\smallskip

\emph{Step 8.}
We complete the argument by deriving~\eqref{e.exhibition}. By~Lemma~\ref{l.converseL2}, we have, for each $z \in 3^n\Zd\cap Q^\circ_{2n}$,
\begin{multline*} \label{}
 \fint_{z+Q_n} L(\overline P_n+D  v(x),x) \, dx 
   \leq 2 \fint_{z+Q_n} L(D  u(x,z+Q_{n+1}),x) \, dx - \mu(z+Q_{n}) \\
+ C \fint_{z+Q_n}  \left| D  v(x) - D  u(x,z+Q_{n+1})+\overline P_n \right|^2  \, dx.
\end{multline*}
In view of~\eqref{e.xicutoff}, it is convenient to denote $\mathcal Z_n':=  \left\{ z\in 3^n\Zd \,:\, \ z+Q_{n+1} \not\subseteq  Q^\circ_{2n} \right\}$ and  $U:= \cup_{z\in\mathcal Z_n'} (z+Q_n)$. Note that $\xi$ vanishes on $Q^\circ_{2n} \setminus U$ and thus $v$  does as well. Observe also that 
\begin{equation*} \label{}
\left| Q^\circ_{2n} \setminus U \right| \leq C3^{-n} \left| Q^\circ_{2n} \right| \quad \mbox{and} \quad \left| \left| \mathcal Z_n' \right| \left| Q_n \right| - \left|Q^\circ_{2n}\right| \right| \leq C3^{-n} \left| Q^\circ_{2n} \right|.
\end{equation*}
Now take the expectation of the previous inequality and sum over $z\in \mathcal Z_n'$, using \eqref{e.mubounds}, \eqref{e.vgradbnd}, \eqref{e.gridtrap}, Lemma~\ref{l.converseL2},~\eqref{e.patchcomp}, stationarity and the above to obtain
\begin{align*}
\lefteqn{ \E\left[\int_{Q^\circ_{2n}} L(\overline P_n+D  v(x), x) \, dx \right] } \qquad & \\ 
& \leq \sum_{z \in\mathcal Z_n'} \E\left[\int_{z+Q_n} L(\overline P_n+D  v(x) , x) \, dx \right]  
+ \E \left[ \int_{Q^\circ_{2n} \setminus U} L(\overline P_n ,x) \, dx \right]  \\
& \leq \left| Q^\circ_{2n} \right|\left(  \E \left[ \mu(Q_n) \right]  +C\tau_n \right).
\end{align*}
Dividing by~$|Q^\circ_{2n}|$ yields~\eqref{e.exhibition} and completes the proof of the lemma.
\end{proof}

\subsection{The proof of~Theorem~\ref{t.mulimit}}
\label{ss.mulimproof}

We use Lemma~\ref{l.patching} and a concentration argument to prove Theorem~\ref{t.mulimit}.

\begin{proof}[{Proof of Theorem~\ref{t.mulimit}}]

By the reduction explained in Section~\ref{ss.dropq}, we assume~$q=0$ and drop the variable~$q$ from our notation, as we did in the previous subsection.

\smallskip

We first argue by iterating Lemma~\ref{l.patching} that $\E \left[ \mu(Q_n) \right] \to \overline\mu$  as $n\to \infty$  at a rate which is at most a power of the length scale,~$3^n$. We then use this result and a concentration argument to improve the stochastic convergence to~\eqref{e.mulimit}, and finally obtain~\eqref{e.planes} by from this and the flatness theory. Throughout, we allow $C(d,\Lambda)\geq 1$ and $\alpha(d,\Lambda)>0$ to vary in each occurrence. 

\smallskip

\emph{Step 1.} We iterate Lemma~\ref{l.patching} to find $C(d,\Lambda)\geq1$ and $\alpha(d,\Lambda)>0$ such that, for every $n\in\N_*$,
\begin{equation} \label{e.Emurate}
\left| \overline \mu - \E \left[ \mu(Q_n) \right] \right| \leq CK_0^2 3^{-n\alpha}.
\end{equation}
We first get an analogous estimate for the trimmed cubes and then use~\eqref{e.Ecubes} and~\eqref{e.infamous} to obtain the desired inequality for the untrimmed cubes. By~\eqref{e.Etrimmedmono}, if the constant $C(d,\Lambda)\geq 1$ is taken large enough and we define
\begin{equation*} \label{}
\mu_n:=  \E \left[ \mu(Q_n^\circ) \right] - CK_0^2 3^{-n}, \quad n \in\N,
\end{equation*}
then~$\mu_n$ is an increasing sequence in~$n$. Clearly $\mu_n$ is bounded from above by~\eqref{e.mubounds}. In view of~\eqref{e.Ecubes} and~\eqref{e.infamous}, we have
\begin{equation} \label{e.trimlim}
\lim_{n\to \infty} \mu_n = \overline \mu.
\end{equation}

Fix $M\in \N$ with $M\geq 2$ to be selected below. By the pigeonhole principle, the monotonicity of $\{ \mu_n\}_{n\in\N}$ and~\eqref{e.trimlim}, we deduce, for each $n\in\N$, the existence of $m\in\{ 0,\ldots,M-1\}$ such that 
\begin{equation} \label{e.pidgeon}
\mu_{n+m+2} - \mu_{n+m} \leq \frac{2}{M} \left( \overline \mu - \mu_n \right). 
\end{equation}
We apply Lemma~\ref{l.patching}  to obtain
\begin{align*} \label{}
\overline \mu - \mu_{n+m} & = \overline\mu - \E \left[\mu(Q_{n+m}^\circ) \right] + CK_0^2 3^{-n} \\
& \leq \overline \mu - \E \left[ \mu(Q_{n+m}) \right] + C\left( \E\left[ \mu(Q_{n+m+1}^\circ) \right] -\E\left[ \mu(Q_{n+m}^\circ) \right]  +CK_0^23^{-n\alpha}\right)  \\
& \leq C\left( \E \left[ \mu(Q_{n+m+1}) \right] - \E \left[ \mu(Q_{n+m}^\circ)\right] + CK_0^2 3^{-n\alpha} \right) \\
& \leq C \left( \E \left[ \mu(Q_{n+m+2}^\circ) \right] - \E \left[ \mu(Q_{n+m}^\circ)\right] + CK_0^2 3^{-n\alpha} \right) \\
& \leq C \left( \frac1M \left( \overline \mu - \mu_n \right) + K_0^23^{-n\alpha} \right).
\end{align*}
Here we used~\eqref{e.infamous} to obtain the second line, Lemma~\ref{l.patching} and~\eqref{e.Ecubes} to get the third line,~\eqref{e.Etrimmedmono} and~\eqref{e.infamous} to get the fourth line, and finally~\eqref{e.pidgeon} in the fifth line. 

\smallskip

By monotonicity and $M\geq m\geq 0$, we obtain 
\begin{equation*} \label{}
\overline \mu - \mu_{n+M} \leq C \left( \frac1M\left( \overline \mu - \mu_n\right) + K_0^23^{-n\alpha} \right).
\end{equation*}
Taking $M := C(d,\Lambda)$ large enough, we obtain
\begin{equation*} \label{}
\overline\mu - \mu_{n+M} \leq \frac13 \left( \overline \mu - \mu_n \right) + CK_0^23^{-n\alpha}.
\end{equation*}
Since $M\geq 2$, we also have, with the same constant $C(d,\Lambda)$ on both sides,
\begin{equation*} \label{}
\overline\mu - \mu_{n+M} + CK_0^23^{-(n+M)\alpha} \leq \frac13 \left( \overline \mu - \mu_n + CK_0^23^{-n\alpha} \right).
\end{equation*}
Therefore, the sequence $\beta_k := \overline \mu - \mu_{kM} + CK_0^23^{-kM\alpha}$ satisfies
\begin{equation*} \label{}
\beta_{k+1} \leq \frac13 \beta_k.
\end{equation*}
By induction, $\beta_k \leq 3^{-k} \beta_0$. Since $\beta_0 \leq CK_0^2$ by~\eqref{e.mubounds}, we obtain in particular that 
\begin{equation*} \label{}
\overline \mu - \mu_{kM} \leq \beta_k \leq CK_0^2 3^{-k}.
\end{equation*}
By monotonicity we get, for every $m\geq kM$,
\begin{equation*} \label{}
|\overline \mu - \mu_{m}| = \overline \mu - \mu_{m} \leq CK_0^2 3^{-k}.
\end{equation*}
The previous line yields, for each $n\in\N_*$, the estimate
\begin{equation} \label{e.Emuratetrim}
\left| \overline \mu - \E \left[ \mu(Q_n^\circ) \right] \right| \leq CK_0^2 3^{-n\alpha}.
\end{equation}
By monotonicity,~\eqref{e.Ecubes} and the previous line, we get
\begin{equation*} \label{}
\left| \overline \mu - \E \left[ \mu(Q_{n}) \right]  \right| = \overline \mu - \E \left[ \mu(Q_{n}) \right] \leq \overline \mu - \E \left[ \mu(Q_{n}^\circ) \right] + CK_0^2 3^{-n}  \leq C K_0^2 3^{-n\alpha}.
\end{equation*}
This is~\eqref{e.Emurate}.

\smallskip

\emph{Step 2.} We deduce the existence of $\overline P\in\Rd$ such that, for every $n\in\N$,
\begin{equation} \label{e.Pconv}
\left| \overline P - \overline P_n \right|^2 \leq CK_0^2 3^{-n\alpha}.
\end{equation}
It is immediate from~\eqref{e.convexUV} (taking $V=Q_{n+1}$ and $U$ to be the union of the $3^d$ $n$-scale subcubes of $Q_{n+1}$) and~\eqref{e.Emurate} that 
\begin{equation*}
\left| \E \left[ P(Q_n) \right] - \E \left[ P(Q_{n+1}) \right] \right|^2 \leq CK_0^23^{-n\alpha}. 
\end{equation*}
Then from~\eqref{e.flatuntrim} we deduce that
\begin{equation*}
\left| \overline P_n - \overline P_{n+1} \right|^2 \leq CK_0 3^{-n\alpha}.
\end{equation*}
Summing this over $\{ n, n+1,\ldots\}$ yields the existence of $\overline P\in\Rd$ satisfying~\eqref{e.Pconv}. 
For future reference we note that, by~\eqref{e.Pnbnds},
\begin{equation} \label{e.Pbarbnd}
\left| \overline P \right|^2 \leq CK_0^2.
\end{equation}

\smallskip

\emph{Step 3.} After possibly redefining $\alpha(d,\Lambda)$ to be smaller, we obtain
\begin{equation} \label{e.L1conv}
\E \left[ \nu(Q_{n}^\circ,\overline P) -  \mu(Q_n^\circ) \right]  \leq CK_0^2 3^{-n\alpha}.
\end{equation}
By~\eqref{e.Etrimmedmono},~\eqref{e.infamous},~\eqref{e.nucontp},~\eqref{e.Pnbnds}, Lemma~\ref{l.patching},~\eqref{e.Emurate},~\eqref{e.Pconv} and~\eqref{e.Pbarbnd}, we have
\begin{align*} \label{}
\lefteqn{\E \left[ \nu(Q_{2n}^\circ,\overline P) -  \mu(Q_{2n}^\circ) \right] } \qquad & \\
& \leq  \E \left[ \nu(Q_{2n}^\circ,\overline P) -  \nu(Q_{2n}^\circ,\overline P_n) \right] + \E \left[ \nu(Q_{2n}^\circ,\overline P_n) -  \mu(Q_{n}^\circ) \right] + CK_0^2 3^{-n}\\
& \leq C K_0 |\overline P- \overline P_n| + C\left( \E \left[ \mu(Q^\circ_{n+2})\right] - \E \left[ \mu(Q^\circ_n) \right] + K_0^23^{-n} \right) \\
& \leq CK_0^23^{-\alpha n/2}.
\end{align*}
This yields~\eqref{e.L1conv} after we replace $\alpha$ by $\alpha/4$.

Observe also that~\eqref{e.Emurate} and~\eqref{e.L1conv} imply that $\overline \mu = \overline L(\overline P)$ and
\begin{equation} \label{e.convexpnu}
\left| \E\left[ \nu(Q_n^\circ,\overline P)\right] - \overline \mu \right| \leq CK_0^2 3^{-n\alpha}.
\end{equation}

\smallskip

\emph{Step 4.} We use independence to improve the convergence of the expectations from the previous step to convergence in $L^1(\Omega,\P)$. The claim is that, after redefining $\alpha(d,\Lambda)>0$ to be smaller, we have
\begin{equation} \label{e.L1conv2}
\E \Big[ \left| \nu(Q_{n}^\circ,\overline P) - \overline  \mu \right| + \left| \mu(Q_{n}^\circ) - \overline  \mu \right| \Big] \leq CK_0^2 3^{-n\alpha}.
\end{equation}
Using the fact that $\mu(Q_k^\circ) \leq \nu(Q_k^\circ,\overline P)$, we have 
\begin{equation*} \label{}
\left| \nu(Q_n^\circ) - \overline \mu \right| \leq \nu(Q_n^\circ) - \mu(Q_n^\circ) + \left| \mu(Q_n^\circ) - \overline \mu \right|.
\end{equation*}
Observe that 
\begin{align*} \label{}
\E \Big[ \left| \mu(Q_n^\circ) - \overline \mu \right| \Big] & \leq \E \Big[ \left| \mu(Q_n^\circ) - \E \left[ \mu(Q_n^\circ) \right]  \right| \Big] + CK_0^23^{-n\alpha} \\
& = 2 \E \Big[ \left(  \E \left[ \mu(Q_n^\circ) \right] - \mu(Q_n^\circ) \right)_+ \Big] + CK_0^23^{-n\alpha}.
\end{align*}
Moreover, 
\begin{align*} \label{}
\lefteqn{\E \left[  \left( \E \left[  \mu(Q_{n+1}^\circ) \right] - \mu(Q_{n+1}^\circ)   \right)_+^2 \right] } \qquad & \\
& \leq \E \left[  \left( \E \left[  \mu(Q_{n}^\circ) \right] - 3^{-d} \sum_{Q_n^\circ(x) \subseteq Q_{n+1}^\circ} \mu(Q_{n}^\circ(x))   \right)_+^2 \right]  + CK_0^4 3^{-n\alpha}  \\
& = 3^{-2d} \sum_{Q_n^\circ(x) \subseteq Q_{n+1}^\circ} \E \left[  \left( \E \left[  \mu(Q_{n}^\circ) \right] - \mu(Q_{n}^\circ(x))   \right)_+^2 \right]  + CK_0^4 3^{-n\alpha}  \\
& =  3^{-d} \E \left[  \left( \E \left[  \mu(Q_{n}^\circ) \right] - \mu(Q_{n}^\circ)   \right)_+^2 \right]  + CK_0^4 3^{-n\alpha}.
\end{align*}
Here we used~\eqref{e.mubounds},~\eqref{e.trimsa} and~\eqref{e.Emuratetrim} in the first line, independence in the second line and finally stationarity in the third line. Since $\E \left[  \left( \E \left[  \mu(Q_{1}^\circ) \right] - \mu(Q_{1}^\circ)   \right)_+^2 \right] \leq CK_0^4$, an iteration of the previous inequality yields
\begin{equation*} \label{}
\E \left[  \left( \E \left[  \mu(Q_{n}^\circ) \right] - \mu(Q_{n}^\circ)   \right)_+^2 \right]  \leq CK_0^4 3^{-n\alpha}.
\end{equation*}
Combining the inequalities above and using~\eqref{e.L1conv} yields~\eqref{e.L1conv2} after a redefinition of $\alpha$.

\smallskip

\emph{Step 5.} We upgrade the stochastic integrability of~\eqref{e.L1conv2}, using an elementary concentration argument. The claim is that, for every $m,n\in\N$ and $t\geq C3^{-n\alpha}$,
\begin{equation}\label{e.firstchebstep}
\P \big[ \left| \nu(Q_{n+m}^\circ,\overline P) - \overline \mu \right| + \left| \overline \mu - \mu(Q_{n+m}^\circ) \right| \geq K_0^2 t \big]  \leq \exp\left( -c3^{dm} t\right).
\end{equation} 
Fix $s>0$ and compute, using~\eqref{e.trimsa}, independence, and stationary:
\begin{align*}
\lefteqn{ \log\E \left[ \exp\left( s3^{dm} \left(\overline \mu - \mu(Q_{n+m}^\circ) \right)_+ \right)  \right]} \qquad  &  \\
& \leq \log \E \left[ \prod_{Q_n^\circ(x)\subseteq Q_{n+m}^\circ}\exp\left( s \left(\overline \mu - \mu(Q_n^\circ(x)) \right)_+ \right)  \right] + CK_0^23^{dm-n}  \\
& =  \sum_{Q_n^\circ(x)\subseteq Q_{n+m}^\circ}  \log\E \left[ \exp\left( s \left(\overline \mu - \mu(Q_n^\circ(x)) \right)_+ \right)  \right]  + CK_0^23^{dm-n}  \\
& =  3^{dm} \log \E \left[ \exp\left( s \left(\overline \mu - \mu(Q_n^\circ) \right)_+ \right) \right] + CK_0^23^{dm-n}.
\end{align*}
By~\eqref{e.mubounds}, 
\begin{equation*} \label{}
\left(\overline \mu - \mu(Q_n^\circ) \right)_+ \leq4K_0^2  \quad \mbox{$\P$--a.s.}
\end{equation*}
Therefore, using the elementary inequalities
\begin{equation*} \label{}
\left\{ 
\begin{aligned}
& \exp(t) \leq 1 + 2t && \mbox{for every} \ 0\leq t \leq 1, \\
& \log(1+t) \leq t && \mbox{for every} \ t \geq 0,
\end{aligned}
\right.
\end{equation*}
we deduce that, for each $0< s \leq (4 K_0)^{-2}$,
\begin{multline} \label{e.grinch}
\log\E \left[ \exp\left( s3^{dm} \left(\overline \mu - \mu(Q_{n+m}^\circ) \right)_+ \right)  \right] \\ 
\leq 2s 3^{dm} \E \left[  \left(\overline \mu - \mu(Q_n^\circ) \right)_+ \right] + CsK_0^23^{dm-n}. 
\end{multline}
Now an application of~\eqref{e.L1conv2} yields
\begin{equation*} \label{e.concenboom}
3^{-dm} \log\E \left[ \exp\left( s3^{dm} \left(\overline \mu - \mu(Q_{n+m}^\circ) \right)_+ \right)  \right] \leq C sK_0^2 3^{-n\alpha}.
\end{equation*}
Take $s:= (4 K_0)^{-2}$ and write the previous inequality in the form
\begin{equation*} 
3^{-dm} \log\E \left[ \exp\left( c 3^{dm} K_0^{-2} \left(\overline \mu - \mu(Q_{n+m}^\circ) \right)_+ \right)  \right] \leq C 3^{-n\alpha}.
\end{equation*}
By a similar argument, replacing $\left( \overline \mu - \mu(Q^\circ_{k}) \right)_+$ by $\left( \nu(Q_k^\circ,\overline P) -  \overline \mu  \right)_+$ and using~\eqref{e.trimsanu} rather than~\eqref{e.trimsa}, we also get
\begin{equation*} \label{}
3^{-dm} \log\E \left[ \exp\left( c 3^{dm} K_0^{-2} \left(\nu(Q_{n+m}^\circ,\overline P) -  \overline \mu  \right)_+ \right)  \right] \leq C 3^{-n\alpha}.
\end{equation*}
Define
\begin{equation*} \label{}
E(U) := \left| \nu(U,\overline P)-\overline \mu \right| + \left| \overline \mu - \mu(U) \right|
\end{equation*}
and observe by $\mu(U) \leq \nu(U,\overline P)$ that
\begin{equation} \label{e.dumbness}
E(U) \leq   2\left( \nu(U,\overline P) -  \overline \mu  \right)_+ + 2 \left(\overline \mu - \mu(U) \right)_+.
\end{equation}
Therefore we obtain
\begin{equation*} 
3^{-dm} \log\E \left[ \exp\left( c 3^{dm} K_0^{-2} E(Q_{n+m}^\circ) \right) \right] \leq C 3^{-n\alpha}.
\end{equation*}
An application of Chebyshev's inequality yields, for every $m,n\in\N$ and $t\geq C3^{-n\alpha}$:
\begin{align}\label{e.firstcheb}
\P \left[ K_0^{-2} E(Q_{n+m}^\circ) \geq t \right] &  =  \P \left[ \exp\left( c3^{dm}K_0^{-2} E(Q_{n+m}^\circ)\right) \geq \exp\left( c3^{dm}t\right) \right] \\
& \leq \exp\left( -c3^{dm} t \right) \E \left[\exp\left( c3^{dm}K_0^{-2} E(Q_{n+m}^\circ)\right) \right]  \nonumber \\
& \leq \exp\left( C3^{dm-\alpha n} - c3^{dm} t \right)  \nonumber \\
& \leq \exp\left( -c3^{dm} t\right). \nonumber
\end{align} 
This is~\eqref{e.firstchebstep}.

\smallskip

\emph{Step 6.} We complete the proof of~\eqref{e.mulimit}. The main point still to be addressed is to allow for arbitrary translations of the cubes, and this is handled by a union bound and a stationarity argument to get the desired estimate from~\eqref{e.firstchebstep}. Recall that, for every $y\in\Rd$, there exists $z\in \Zd$ with $|z - y| \leq \sqrt{d}$ and $z+Q_n^\circ \subseteq y+Q_n$ for every $n\in\N$. Thus by~\eqref{e.Ecubes} we obtain, for every $R>0$ and $n\in\N$,
\begin{equation*} \label{}
\sup_{y\in B_R} \mu(y+Q_n) \geq \max_{z\in \Zd\cap B_{R+\sqrt{d}}}\mu(z+Q_n^\circ) - CK_0^23^{-n} \quad \mbox{$\P$--a.s.}
\end{equation*}
Hence for all $R\geq 1$ and $n\in\N$,
\begin{equation*} \label{}
\sup_{y\in B_R} \left( \overline \mu - \mu(y+Q_n)\right)_+  \leq \max_{z\in \Zd\cap B_{CR}} \left( \overline \mu - \mu(z+Q_n^\circ) \right)_+ + CK_0^23^{-n} \quad \mbox{$\P$--a.s.}
\end{equation*}
By a union bound, stationarity and~\eqref{e.firstchebstep}, we obtain, for every $n,m\in\N$, $R\geq1$ and $t\geq C3^{-n\alpha}$,
\begin{align*}
\lefteqn{\P \left[ K_0^{-2} \sup_{y\in B_R} \left( \overline \mu - \mu(y+Q_{n+m}) \right)_+ \geq t \right]  } \qquad & \\
& \leq \sum_{z\in \Zd \cap B_{CR} }  \P \left[ K_0^{-2} \left( \overline \mu - \mu(z+Q_{n+m}^\circ) \right)_+ \geq t \right]  \\
& \leq CR^d\,  \P \left[ K_0^{-2} \left( \overline \mu - \mu(Q_{n+m}^\circ) \right)_+ \geq t \right] \leq CR^d \exp\left( -c3^{dm} t \right).
\end{align*}
By an analogous argument, using~\eqref{e.Ecubesnu} instead of~\eqref{e.Ecubes}, we obtain, for $t\geq C3^{-n\alpha}$,
\begin{equation*} \label{}
 \P \left[ K_0^{-2} \sup_{y\in B_R} \left( \nu(y+Q_{n+m},\overline P) -  \overline \mu  \right)_+ \geq t \right] \leq CR^d \exp\left( -c3^{dm} t \right).
\end{equation*}
Using~\eqref{e.dumbness} again and replacing $t$ by $C3^{-n\alpha}t$, we obtain, for every $n,m\in\N$, $R\geq1$ and $t\geq 1$,
\begin{equation} \label{e.chebbed2}
\P \left[ K_0^{-2} \sup_{y\in B_R} E(y+Q_{n+m}) \geq C3^{-n\alpha}t \right] \leq CR^d \exp\left( -c3^{dm-n\alpha} t \right). 
\end{equation}
To see that this implies~\eqref{e.mulimit}, fix $s\in (0,d)$.
Choose $m=m(n)$ to be the smallest positive integer such that
\begin{equation*} \label{}
s < \frac{dm- n\alpha}{n+m}.
\end{equation*}
That is, $m(n) : = \lfloor (s+\alpha)n/(d-s) \rfloor \geq c n/(d-s)$. Then~\eqref{e.chebbed2} yields, for every $t\geq 1$,
\begin{equation*} \label{}
\P \left[ K_0^{-2} \sup_{y\in B_R} E(y+Q_{n+m}) \geq C3^{-n\alpha}t \right] \leq CR^d \exp\left( -c3^{s(n+m)}t \right). 
\end{equation*}
This implies~\eqref{e.mulimit} after a redefinition of $\alpha$.

\smallskip

\emph{Step 7.} We prove the flatness estimates~\eqref{e.planes}. It is easier to work with the minimizers for $\nu$, so we handle them first and obtain the flatness of the $\mu$ minimizers as a consequence. Fix $y\in \Rd$ and denote
\begin{equation*} \label{}
v_n(x):= v(x,y+Q_n(x),\overline P), \quad x\in\Rd.
\end{equation*}
In other words, for each $n\in\N$, the function $v_n:\Rd \to \R$ is obtained by splicing together the minimizers for $\nu(\cdot,\overline P)$ in each triadic cube of the form $y+Q_n(x)\subseteq \Rd$. Observe that $v_n \in H^1_{\mathrm{loc}}(\Rd)$. 

\smallskip

Fix $m,n\in\N_*$ and estimate the $L^2$ difference between the scales $n$ and $n+m$ using the Poincar\'e inequality and~\eqref{e.convexL2nu}:
\begin{align*}
\lefteqn{ \fint_{y+Q_{n+m}} \left( v_{n+m}(x) - v_n(x) \right)^2\,dx } \qquad & \\
 & \leq C3^{2(n+m)} \fint_{y+Q_{n+m}} \left| Dv_n(x) - Dv_{n+m}(x) \right|^2\,dx  \\
 & \leq C3^{2(n+m)} \left( \fint_{y+Q_{n+m}} L(Dv_n(x),x) \,dx - \nu(y+Q_{n+m},\overline P) \right)  \\
 & = C3^{2(n+m)} \left( \fint_{y+Q_{n+m}} \nu(y+Q_n(x),\overline P) \,dx - \nu(y+Q_{n+m},\overline P) \right).
\end{align*}
Next we observe that, viewed from a length scale much larger than~$3^n$,  $v_n$ is close to the plane $\overline P\cdot x$:
\begin{align*} \label{}
\lefteqn{\fint_{y+Q_{n+m}} \left( v_n(x) - \overline P\cdot x \right)^2 \,dx} \qquad & \\
 & = \fint_{y+Q_{n+m}} \fint_{y+Q_n(\xi)}  \left( v_n(x) - \overline P\cdot x \right)^2 \,dx \, d\xi \\
& \leq  \fint_{y+Q_{n+m}} C 3^{2n} \fint_{y+Q_n(\xi)}  \left| Dv_n(x) - \overline P \right|^2 \,dx \, d\xi &&  \mbox{(by Poincar\'e ineq.)} \\
& \leq \fint_{y+Q_{n+m}} C 3^{2n} \fint_{y+Q_n(\xi)} \left( \left| Dv_n(x)\right|^2 + \left|\overline P \right|^2 \right) \,dx \, d\xi \\
& \leq C3^{2n} K_0^2 &&  \mbox{(by~\eqref{e.mubounds},~\eqref{e.Pbarbnd}).}
\end{align*}
Assembling these, we obtain
\begin{multline*} \label{}
\fint_{y+Q_{n+m}} \left| v_{n+m}(x) - \overline P\cdot x \right|^2\,dx \\
\leq C3^{2n} K_0^2 + C3^{2(n+m)} \left( \fint_{y+Q_{n+m}} \nu(y+Q_n(x),\overline P) \,dx - \nu(y+Q_{n+m},\overline P) \right).
\end{multline*}
Since $y\in\Rd$ was arbitrary, the previous inequality yields, for each $R\geq 1$,
\begin{multline*} \label{}
\sup_{y \in B_R} 3^{-2(n+m)} \fint_{y+Q_{n+m}} \left| v_{n+m}(x) - \overline P\cdot x \right|^2\,dx \\
 \leq CK_0^2\bigg( 3^{-2m} + \sup_{y\in B_R} \bigg( \left| \nu(y+Q_{n+m},\overline P) - \overline \mu\right| \\+\sup_{x\in y+Q_{n+m}} \left| \nu(y+Q_n(x),\overline P) \,dx - \overline \mu \right|  \bigg) \bigg).
\end{multline*}
Fix $s\in (3d/4,d)$ and apply~\eqref{e.mulimit} to  obtain, for every $n,m\in\N$, $R \geq 1$ and $t\geq 1$,
\begin{multline*}
\P \left[  \sup_{y \in B_R} K_0^{-2}3^{-2(n+m)} \fint_{y+Q_{n+m}} \left| v_{n+m}(x) - \overline P\cdot x \right|^2\,dx \geq C\left( 3^{-2m} +3^{-n \alpha(d-s)} \right)t \right] \\ 
\leq C\left(R^d+3^{d(n+m)}\right) \exp\left(-c3^{sn}t\right).
\end{multline*}
Take $m$ to be the smallest integer larger than $n\alpha(d-s)/2$, replace $n+m$  by $n$ and $s$ by $s-c(d-s)$ and shrink $\alpha$, if necessary, to obtain, for every $t\geq 1$,
\begin{multline*}
\P \left[  \sup_{y \in B_R} K_0^{-2}3^{-2n} \fint_{y+Q_{n}} \left| v_{n}(x) - \overline P\cdot x \right|^2\,dx \geq C3^{-n \alpha(d-s)}t \right] \\ 
\leq C\left(R^d+3^{dn}\right) \exp\left(-c3^{sn}t\right).
\end{multline*}
Replacing $s$ by $s-c(d-s)$ again, we obtain, for every $s\in (0,d)$, $n\in\N$, $R\geq 1$ and $t\geq 1$,
 \begin{multline} \label{e.flatfinnu}
\P \left[  \sup_{y \in B_R} K_0^{-2}3^{-2n} \fint_{y+Q_{n}} \left| v_{n}(x) - \overline P\cdot x \right|^2\,dx \geq C3^{-n \alpha(d-s)}t \right] \\ \leq CR^d \exp\left(-c3^{sn}t\right).
\end{multline}

We complete the proof of~\eqref{e.planes} by obtaining the flatness of minimizers for $\mu$. Fix~$y\in\Rd$. Observe that, by Lemma~\ref{l.convexL2}, 
\begin{equation*} \label{}
\fint_{y+Q_n} \left| Du(x,y+Q_n) - Dv(x,y+Q_n,\overline P) \right|^2\,dx \leq \nu(y+Q_n,\overline P) - \mu (y+Q_n).
\end{equation*}
Hence 
\begin{multline} \label{e.uslopecomp}
\left| \fint_{y+Q_n} Du(x,y+Q_n)\,dx - \overline P \right|^2  \\
= \left|\fint_{y+Q_n} \left( Du(x,y+Q_n) - Dv(x,y+Q_n,\overline P) \right) \,dx   \right|^2  \leq  \nu(y+Q_n,\overline P) - \mu (y+Q_n).
\end{multline}
and so, by the Poincar\'e inequality, 
\begin{align*} \label{}
\lefteqn{ \fint_{y+Q_n} \left( u(x,y+Q_n) - v(x,y+Q_n,\overline P) + y\cdot \overline P\right)^2\,dx} \qquad & \\
& \leq C3^{2n} \bigg(  \fint_{y+Q_n} \left| Du(x,y+Q_n) - Dv(x,y+Q_n,\overline P) \right|^2\,dx \\
& \qquad \qquad + \nu(y+Q_n,\overline P) - \mu (y+Q_n) \bigg) \\
& \leq C3^{2n} \left( \nu(y+Q_n,\overline P) - \mu (y+Q_n) \right).
\end{align*}
The previous inequality,~\eqref{e.mulimit} and~\eqref{e.flatfinnu} yield, for every $s\in (0,d)$, $n\in\N$, $R\geq 1$ and $t\geq 1$,
\begin{multline*} 
\P \left[  \sup_{y \in B_R} K_0^{-2}3^{-2n} \fint_{y+Q_{n}} \left| u(x,y+Q_n) - \overline P\cdot (x-y) \right|^2\,dx \geq C3^{-n \alpha(d-s)}t \right] \\ \leq CR^d \exp\left(-c3^{sn}t\right).
\end{multline*}
This completes the proof of~\eqref{e.planes}.
\end{proof}

\subsection{Convex duality between $\overline \mu$ and $\overline L$}
\label{s.idenL}

An immediate consequence of~\eqref{e.mulimform} is the following formula for $\overline \mu$ in terms of $\overline L$: for every $q\in\Rd$,
\begin{equation} \label{e.predual}
\overline \mu(q) = - \sup_{p\in\Rd} \left( p\cdot q - \overline L(p) \right). 
\end{equation}
Indeed, the difficult half of~\eqref{e.predual} is implied by~\eqref{e.mulimform} and the other, easier half is a consequence of~\eqref{e.munuord}. 

\smallskip

The expression~\eqref{e.predual} asserts that $- \overline\mu$ is the Legendre-Fenchel transform of $\overline L$. Since the latter is uniformly convex by~\eqref{e.Lbarconvex}, it follows by convex duality that, for every $p\in\Rd$,
\begin{equation} \label{e.Lbardual}
\overline L(p) = \sup_{q\in\Rd} \left( p\cdot q + \overline \mu(q) \right). 
\end{equation}
Since $\overline L$ is uniformly convex, its gradient $D\overline L$ is a bijective Lipschitz map on $\Rd$. The formula~\eqref{e.Lbardual} implies that $D\overline L(p)$ is the unique $q$ achieving the supremum in~\eqref{e.Lbardual}. The inverse of the this map is evidently the function $q\mapsto \overline P(q)$ given in the statement of Theorem~\ref{t.mulimit}. That is, $p = \overline P\big( D\overline L(p) \big)$ and moreover, for every $p\in\Rd$,
\begin{equation*} \label{}
\overline L(p) = p \cdot D\overline L(p) + \overline \mu\big(D\overline L(p) \big).
\end{equation*}

In particular, the map $\overline P$ can be inverted, and this allows us to reformulate the statement of Theorem~\ref{t.mulimit} so that the parameter $p$ is given rather than $q$. It is convenient to gather all of the errors we wish to measure with respect to a bounded, connected domain $U\subseteq \Rd$ and a given $p\in\Rd$ into one random variable. Set
\begin{multline} \label{e.error}
\mathcal E(U,p) := \left|  \overline L(p)  -  \mu\left(U,D\overline L(p)\right) - p\cdot D\overline L(p) \right| + \left| \overline L(p) -  \nu(U,p) \right|  \\
+ |U|^{-2/d} \fint_{U} \left( \left( v\!\left(x,U,p\right) - p \cdot x \right)^2 + \left( u\!\left(x,U,D\overline L(p)\right) - p \cdot (x-x_U) \right)^2 \right)\,dx,
\end{multline}
where $x_U:= \fint_U x\,dx$ denotes the barycenter of~$U$.

\begin{corollary}\label{cor.mulim1}
With $\alpha(d,\Lambda)>0$ as in the statement of Theorem~\ref{t.mulimit}, there exist $C(d,\Lambda)\geq 1$ and $c(d,\Lambda)>0$ such that, for every $s\in(0,d)$, $p\in\Rd$, $n\in\N$ and $t\geq 1$,
\begin{equation*} \label{}
\P \Big[ \exists y \in B_R, \ \mathcal E(y+Q_n,p) \geq C \left(K_0+|p|\right)^2 3^{-n\alpha(d-s)}t \Big]  \leq CR^d \exp\left(-c3^{sn} t\right).
\end{equation*}
\end{corollary}
\begin{proof}
Apply Theorem~\ref{t.mulimit} to $q = D\overline L(p)$. By the remarks preceding the statement of the corollary, we have~$\overline P(q)=p$. From the first inequality of~\eqref{e.DLbarLip} we have
\begin{equation*} \label{}
\big|D \overline L(p) \big| \leq C ( K_0 + |p|).
\end{equation*}
Theorem~\ref{t.mulimit} thus yields the corollary. 
\end{proof}

We conclude this section with a further refinement of~Theorem~\ref{t.mulimit} which gives some uniformity in our estimates of $\mathcal E(U,p)$ with respect to~$p$. This is needed in the next section in the argument for the error in the Dirichlet problem.

\begin{corollary}
\label{cor.mulim2}
Fix $M,R,k \geq 1$ and $s\in(0,d)$. There exist $\alpha(d,\Lambda)>0$, $c(d,\Lambda)>0$ and $C(d,\Lambda,s,k)\geq 1$ such that, for every $n\in\N$ and $t\geq 1$,
\begin{multline*} \label{}
\P \Big[ \exists p \in B_{M 3^{kn}},\, \exists y\in B_{R3^{kn}}, \ \  \mathcal E(y+Q_n,p) \geq C \left(K_0+|p|\right)^2 3^{-n\alpha(d-s)}t \Big] 
\\ \leq CM^dR^d\exp\left(-c3^{sn} t\right).
\end{multline*}
\end{corollary}
\begin{proof}
We see from~\eqref{e.nucontp},~\eqref{e.minnucloseL2},~\eqref{e.LbarLip},~\eqref{e.DLbarLip},~\eqref{e.mucontq} and~\eqref{e.minmucloseL2} that the error term is continuous in $p$, uniformly on the support of~$\P$: that is, for every $n\in\N_*$ and $p_1,p_2\in\Rd$,
\begin{equation} \label{e.Econtp}
\big| \mathcal E(Q_n,p_1) - \mathcal E(Q_n,p_2) \big| \leq C\left(K_0+|p_1|+|p_2| \right) |p_1-p_2| \quad \mbox{$\P$--a.s.}
\end{equation}
Therefore, it is enough to check the error estimate for $p$'s on a discrete mesh with spacings $h_n:= 3^{-n\alpha d}$. Denoting this mesh by $G_n:= \left( h_n\Zd \right) \cap B_{M3^{kn}}$ and letting $\Omega'\subseteq \Omega$ be the event with~$\P[\Omega']=1$ on which~\eqref{e.Econtp} holds, we have, for each $n\in\N$ and $t\geq 1$,
\begin{multline}  \label{e.contpush}
\big\{ \forall p\in G_n, \, \forall y\in B_{R3^{kn}} \ \mathcal E(y+Q_n,p) \leq C \left(K_0+|p|\right)^2 3^{-n\alpha(d-s)}t \big\} \cap \Omega' \\
 \subseteq \big\{ \forall p \in B_{M3^{kn}}, \,\forall y\in B_{R3^{kn}}, \ \mathcal E(y+Q_n,p) \leq C \left(K_0+|p|\right)^2 3^{-n\alpha(d-s)}t \big\}.
\end{multline}
Here the $C$ on the right side is larger than the one on the left to accommodate the discretization error of order $C(K_0+|p|)h_n \lesssim C(K_0+|p|)^23^{-n\alpha d}$ coming from the right side of~\eqref{e.Econtp}.

\smallskip

By~\eqref{e.contpush}, a union bound and an application of Corollary~\ref{cor.mulim1}, we find that, for each $n\in\N$ and $t\geq 1$,
\begin{align*} 
\lefteqn{ \P \Big[ \exists p\in B_{M3^{kn}}, \, y\in B_{R3^{kn}}, \ \mathcal E(y+Q_n,p) \geq C \left(K_0+|p|\right)^2 3^{-n\alpha(d-s)} t \Big]  } \qquad & \\
& \leq  \sum_{p\in G_n} \P \Big[  \exists y\in B_{R3^{kn}}, \ \mathcal E(y+Q_n,p) \geq C \left(K_0+|p|\right)^2 3^{-n\alpha(d-s)} t \Big]  \\
& \leq |G_n| \max_{p \in\Rd} \P \Big[ \,  \exists y\in B_{R3^{kn}}, \ \mathcal E(y+Q_n,p) \geq C \left(K_0+|p|\right)^2 3^{-n\alpha(d-s)}  t\Big] \\
& \leq CR^d 3^{knd} |G_n| \exp\left( -c3^{sn}t \right).
\end{align*} 
The number of elements of the set $G_n$ is easy to compute:
\begin{equation*} \label{}
\big| G_n \big| \leq C h_n^{-d} | B_{M3^{kn}}| = C M^d3^{dn(d\alpha + k)}.
\end{equation*}
We thus deduce that, for every $s\in(0,d)$, $n\in\N$ and $t\geq 1$,
\begin{multline} \label{e.unifpsp}
\P \Big[ \forall p \in B_{M3^{kn}}, \,\forall y\in B_{R3^{kn}},   \ \mathcal E(y+Q_n,p) \geq C \left(K_0+|p|\right)^2 3^{-n\alpha(d-s)}t  \Big] \\
 \leq C M^dR^d3^{dn(d\alpha + 2k)} \exp\left( -c3^{sn} t\right) \leq CM^d R^d \exp\left( Cn-c3^{sn}t \right).
\end{multline}
Set~$s_1:=(s+d)/2$, note that~$s_1\in(s,d)$ depends only on~$d$ and $s$ and apply~\eqref{e.unifpsp} with $s_1$ in place of $s$ and use the fact that, for $t\geq 1$,
\begin{equation*} \label{}
\exp\left( Cn-c3^{s_1n}t \right) \leq C\exp\left( -c3^{-sn} t\right),
\end{equation*}
to obtain
\begin{multline*} \label{}
\P \Big[ \forall p \in B_{M3^{kn}}, \,\forall y\in B_{R3^{kn}},   \ \mathcal E(y+Q_n,p) \geq C \left(K_0+|p|\right)^2 3^{-n\alpha(d-s_1)} t \Big] \\
 \leq C M^dR^d \exp\left( -c3^{sn} t\right).
\end{multline*}
Since $d-s_1= (d-s)/2$, we get the desired conclusion after replacing $\alpha$ by $\alpha/2$.  
\end{proof}

\section{The error estimate for the Dirichlet problem}
\label{s.DP}

In this section we prove Theorem~\ref{t.mainthm}, obtaining error estimates in homogenization for Dirichlet problems in bounded Lipschitz domains with fairly general boundary conditions. The arguments here are mostly technical and completely deterministic: all of the heavy lifting was done in the previous section, where in particular we proved error estimates for the Dirichlet problem in cubes with planar boundary conditions. It turns out that this is enough to give us Theorem~\ref{t.mainthm}, as we will see from fairly simple oscillating test function and energy comparison arguments. 

\subsection{The proof of Theorem~\ref{t.mainthm}}

We begin with the statement of an abstract tool which provides control of the error for general Dirichlet problems in terms of the error for the Dirichlet problem in mesoscopic cubes with planar boundary conditions. This ``black box" is oblivious to the randomness and to much of the precise structure of the problem. Although straightforward, its proof (given in the appendix) is unfortunately a rather technical and lengthy energy comparison argument relying on some classical interior regularity results. 

\begin{proposition}
\label{p.blackbox}
Let $U \subseteq \R^d$ be a bounded Lipschitz domain, $K_0,M\geq 1$ and $t>2$. Fix $\ep \in(0,1]$, $L \in\Omega(K_0)$, $g\in W^{1,t}(U)$ satisfying~\eqref{e.gineq} and $u_\ep,  u_{\mathrm{hom}} \in g + H^1_0(U)$ satisfying~\eqref{e.localminz1} and~\eqref{e.localminz2}, respectively. Select $n\in\N$ such that $3^{-n} < \ep \leq 3^{-n+1}$ and fix $m,l\in\N$ such that $m \leq l \leq n$. Then there exist constants $C(d,\Lambda,t,U)\geq1$ and $\beta(d,\Lambda,t) \in (0,1]$ such that 
\begin{multline}\label{e.blackbox}
\left| \fint_U \left( L\left(Du_\ep(x), \frac x\ep \right) - \overline L(D u_{\mathrm{hom}}(x)) \right)\,dx \right| + \fint_U (u_\ep (x)- u_{\mathrm{hom}}(x))^2 \,dx  \\
 \leq  C\mathcal E' + C M^2\left( 3^{-(l-m)} + 3^{-\beta (n-l)} \right),
\end{multline}
where 
\begin{equation*} \label{}
\mathcal E':= \sup \left\{ \mathcal E(y+Q_{m},p) + \mathcal E(y'+Q_{m+2},p') \,:\, y,y'\in B_{C3^{n}}, \ p,p'\in B_{CM3^{(n-m)d/2}} \right\}.
\end{equation*}
and $\mathcal E$ is defined in~\eqref{e.error}. 
\end{proposition}

The proof of Proposition~\ref{p.blackbox} is presented in Appendix~\ref{ss.bb}.

Assuming the proposition, numerological and bookkeeping details and the choices of the parameters $m$ and $l$ are essentially all that still stand between us and the demonstration of the first main result. 

\begin{proof}[{Proof of Theorem~\ref{t.mainthm}}]
Fix $\ep \in (0,1]$ and $s\in (0,d)$. Take $\alpha(d,\Lambda)>0$ as in the statement of Corollary~\ref{cor.mulim2} and $\beta(d,\Lambda,t)>0$ as in the statement of Proposition~\ref{p.blackbox}. Also set $s_1 := (2t+d)/3$ and $s_2:=(t+2d)/3$ so that $s< s_1<s_2<d$, with the gaps between these numbers bounded by $c(d,\Lambda,s)>0$.

\smallskip

Let $n\in \N$ be such that $3^{-n} < \ep \leq 3^{-n+1}$ and select $m=m(n)\in\N$ to be the smallest integer satisfying
\begin{equation} \label{e.pickm}
 3^{n s_1} \leq 3^{ms_2} \quad \mbox{and} \quad 2d(n-m) \leq m\alpha(d-s_2).
\end{equation}
Note that $m\leq n$. Pick $l\in\N$ to be the smallest integer such that $l \geq (m+n)/2$. It is then evident that, for an exponent $\gamma(d,\Lambda,t)>0$, 
\begin{equation*} \label{}
3^{-(l-m)} + 3^{-\beta(n-l)} \leq C\ep^\gamma.
\end{equation*}
Let $\mathcal E'_n$ be the random variable $\mathcal E'$ defined in the statement of Proposition~\ref{p.blackbox}, with respect to the choice of~$n$ and~$m(n)$, above. 

\smallskip

Proposition~\ref{p.blackbox} gives the estimate
\begin{equation*} \label{}
\fint_U (u_\ep (x)- u_{\mathrm{hom}}(x))^2 \,dx 
 \leq  C\mathcal E'_n + C M^2\ep^\gamma.
\end{equation*}
To complete the proof of the theorem, it therefore suffices to demonstrate that there exists a random variable~$\X$ satisfying~\eqref{e.intX} and $\gamma(d,\Lambda,s)>0$ such that, for every $n\in\N$,
\begin{equation} \label{e.stochintegrab}
\mathcal E'_n \leq C M^2 \left( 1+\mathcal X 3^{-sn} \right) 3^{-n\gamma}.
\end{equation}
We argue that~\eqref{e.stochintegrab} is a consequence of Corollary~\ref{cor.mulim2}. The latter implies that, for every $t\geq 1$,
\begin{equation*} \label{}
\P\left[ \mathcal E'_n \geq C \left(K_0^2+M^23^{(n-m)d}\right)3^{-m\alpha(d-s_2)}t \right] \leq C M^d\exp\left( - c3^{s_2m}t \right).
\end{equation*}
Using~\eqref{e.pickm}, this yields, for $t\geq 1$,
\begin{equation*} \label{}
\P\left[ \mathcal E'_n \geq C M^2 3^{-(n-m)d} t \right] \leq CM^d \exp\left( - c3^{s_1n}t   \right).
\end{equation*}
Replacing $t$ by $1+t$ and rearranging, we obtain, for all $t>0$,
\begin{equation*} \label{}
\P\left[\left( 3^{(n-m)d} M^{-2} \mathcal E'_n -C \right) \geq Ct \right] \leq CM^d \exp\left( - c3^{s_1n}t  \right).
\end{equation*}
Replacing $t$ by $3^{-sn}t$, we obtain, for all $t>0$,
\begin{equation} \label{e.grubcha}
\P\left[ 3^{sn}\left( 3^{(n-m)d} M^{-2} \mathcal E'_n -C \right)_+ \geq Ct \right] \leq CM^d \exp\left( - c3^{(s_1-s)n}t \right).
\end{equation}
Let $\X$ be the random variable
\begin{equation*} \label{}
\X := \sup_{n\in\N} 3^{sn}\left( 3^{(n-m)d} M^{-2} \mathcal E'_n -C \right)_+,
\end{equation*}
where $m=m(n)$ is defined as above. As $3^{-(n-m)d} \leq 3^{-n\gamma}$ for some $\gamma(d,\Lambda,s)>0$, it is evident that~\eqref{e.stochintegrab} holds. By a union bound and summing the right side of~\eqref{e.grubcha} over $n\in\N$, we obtain 
\begin{equation*} \label{}
\P \left[ \X \geq Ct \right] \leq CM^d \exp\left( -ct \right).
\end{equation*}
Replacing $\X$ by $c\X$ and integrating the previous line yields~\eqref{e.intX}. This completes the proof of Theorem~\ref{t.mainthm}.
\end{proof}

\subsection{Uniform approximation in $L^\infty$}
By interpolating $L^\infty$ between $L^2$ and $C^{0,\gamma}$ and applying Theorem~\ref{t.mainthm} and the (nonlinear version of the) De Giorgi-Nash-Moser estimates, we obtain estimates for the Dirichlet problem with the spatial error measured in $L^\infty$ rather than $L^2$. Since the estimate in Theorem~\ref{t.mainthm} is already suboptimal in the exponent, there is essentially no loss in passing from $L^2$ to $L^\infty$. We present a model result in the following corollary, which, in view of its application in the next section, is stated as a local approximation result (rather than an error estimate for the Dirichlet problem) and scaled differently (the microscopic scale is of order one). 

\begin{corollary}
\label{cor.DP}
For every $M\geq 1$, $s\in (0,d)$, there exist $\alpha(d,\Lambda)>0$, $C(d,\Lambda,s)\geq 1$ and a nonnegative random variable $\X$ on $(\Omega,\F)$ satisfying
\begin{equation*} \label{}
\E \big[ \exp(  \X )  \big] \leq CM^d 
\end{equation*}
such that the following holds: for every $L \in\Omega$ $R\geq 1$ and $u \in H^1(B_R)$ satisfying
\begin{equation*} \label{}
K_0 + \frac1R \left( \fint_{B_R} \left| u(x) \right|^2\, dx \right)^{1/2} \leq M
\end{equation*}
and
\begin{equation} \label{e.localmindp}
\int_{B_R} L\left(Du(x), x \right) \,dx \leq \int_{B_R} L\left(Dw(x),x \right) \, dx \quad \mbox{for every} \ w \in u + H^1_0(B_R),
\end{equation}
there exists $v\in H^1(B_{R/2})$ such that 
\begin{equation*}
\int_{B_{R/2}} \overline L(Dv(x)) \,dx \leq \int_{B_{R/2}} \overline L(Dw(x)) \, dx \quad \mbox{for every} \ w \in v + H^1_0(B_{R/2}),
\end{equation*}
and
\begin{equation} \label{e.DPapprox}
\sup_{x\in B_{R/2}} \left| u(x) - v(x) \right| 
 \leq CM \left(1+\X R^{-s} \right) R^{-\alpha(d-s)}.
\end{equation}
\end{corollary}

\begin{proof}
By the interior Meyers estimate (c.f.~\cite[Theorem 6.7]{Giu} and the remarks in Section~\ref{s.regest}), there exists $t(d,\Lambda)> 2$ and $C(d,\Lambda)\geq 1$ such that $u\in W^{1,t}(B_{3R/4})$ and
\begin{equation*} \label{}
\left( \fint_{B_{3R/4}} |Du(x)|^t \, dx \right)^{2/t} \leq C\left( K_0^2 +  \fint_{B_R} |Du(x)|^2\,dx \right).
\end{equation*}
We let $\X$ be the random variable in the statement of Theorem~\ref{t.mainthm} for $U=B_1$ and with $t$ as in the previous sentence. We take $v \in u+H^1_0(B_{3R/4})$ to be the unique minimizer of the Dirichlet problem 
\begin{equation*} \label{}
\int_{B_{3R/4}} \overline L(Dv(x)) \,dx \leq \int_{B_{3R/4}} \overline L(Dw(x)) \, dx \quad \mbox{for every} \ w \in u+H^1_0(B_{3R/4}). 
\end{equation*}
By Theorem~\ref{t.mainthm}, for some $\alpha(d,\Lambda)>0$, we have
\begin{align} \label{e.interset1}
\lefteqn{ \fint_{B_{R}} \left| u(x) - v (x) \right|^2\, dx } \qquad & \\
 & \leq C\left(1+R^{-s}\X \right) R^{2-\alpha(d-s)} \left( K_0^2+\left(  \fint_{B_R} |Du(x)|^t \, dx\right)^{2/t} \right) \nonumber \\
& \leq  C\left(1+R^{-s}\X \right) R^{-\alpha(d-s)} \left( K_0^2R^2+ \fint_{B_R} |u(x)|^2 \, dx \right). \nonumber
\end{align}
By the De Giorgi-Nash-Moser estimate (see~\cite{G} or~\cite{Giu}), there exists $\gamma(d,\Lambda)\in (0,1]$ such that 
\begin{align} \label{e.interset2}
\left[ u - v \right]_{C^{0,\gamma}(B_{3R/4})} & \leq \left[ u \right]_{C^{0,\gamma}(B_{3R/4})} + \left[ v \right]_{C^{0,\gamma}(B_{3R/4})}\\
& \leq C R^{-\gamma} \left( K_0R + \left( \fint_{B_{R}} \left| u(x) \right|^2 \,dx \right)^{1/2} \right). \nonumber
\end{align}

Applying the interpolation inequality
\begin{align*} \label{}
\sup_{x\in B_r} |\phi(x)| & \leq \left( \int_{B_r} |\phi(x)|^2\,dx \right)^{\gamma/(2\gamma+d)} \left( \sup_{x,y\in B_r} \frac{\phi(x)-\phi(y)}{|x-y|^\gamma} \right)^{d/(2\gamma+d)} \\
& = Cr^{d\gamma/(2\gamma+d)}  \left( \fint_{B_r} |\phi(x)|^2\,dx \right)^{\gamma/(2\gamma+d)} \left( \sup_{x,y\in B_r} \frac{\phi(x)-\phi(y)}{|x-y|^\gamma} \right)^{d/(2\gamma+d)} 
\end{align*}
to $\phi = u-v$ with $r=R/2$, and then using~\eqref{e.interset1} and~\eqref{e.interset2} to estimate the two terms on the right side, we obtain
\begin{multline*} \label{}
\sup_{x\in B_{R/2}} \left| u(x) - v(x) \right|  \\
 \leq C(1+R^{-s}\X) R^{-\alpha(d-s)\gamma/(2\gamma+d)} \left( K_0R+ \left( \fint_{B_R} |u(x)|^2 \, dx \right)^{1/2} \right). 
\end{multline*}
This yields the result, after we redefine  $\alpha$ to be $\alpha \gamma / (2\gamma+d)$.

\smallskip

The interpolation inequality is elementary and of course well-known, but for the convenience of the reader we indicate a short proof here. By homogeneity, we may assume that $|\phi(y)| = \phi(y) = \sup_{x\in B_r}|\phi(x)| = 1$. Then the $\|\phi\|_{L^2(B_r)}$ may be estimated from below by $k:= \left[ \phi \right]_{C^{0,\gamma}(B_r)}$ by observing that 
\begin{equation*} \label{}
\phi(x) \geq \left( 1 - k^{-\gamma}|x-y|^\gamma \right)_+,
\end{equation*}
and directly computing the $L^2$ norm of the function on the right. This gives the interpolation inequality.
\end{proof}

\section{Higher regularity: the quenched Lipschitz estimate}
\label{s.C01}

In this section we prove Theorem~\ref{t.c01}. The argument is based on an idea which was first applied in the context of homogenization by Avellaneda and Lin~\cite{AL1,AL2}: functions which can be approximated in $L^\infty$ by functions satisfying an improvement of flatness property must inherit this property-- at least on scales larger enough that the approximation is valid. We proceed by formalizing this idea in an elementary lemma, which makes it quite evident that Theorem~\ref{t.c01} follows from Theorem~\ref{t.mainthm}. The lemma is a deterministic statement which is oblivious even to the PDE. 

\begin{lemma}
\label{l.blackbox2}
For $r> 0$ and $\theta \in (0,1/2]$, let $\mathcal A(r,\theta) \subseteq L^\infty(B_r)$ denote the set of functions~$w\in L^\infty(B_r)$ satisfying
\begin{equation} \label{e.Artheta}
\frac{1}{\theta r} \inf_{p\in\Rd} \osc_{x\in B_{\theta r}} \left( w(x) - p\cdot x\right) \leq \frac{1}{2}\left( \frac1r \inf_{p\in\Rd} \osc_{x\in B_r} \left( w(x) - p\cdot x \right) \right).
\end{equation}
Suppose $\alpha >0$, $K\geq 0$, $1\leq r_0 \leq R/4$ and $u \in L^\infty(B_R)$ have the property that, for every $r \in [r_0,R/2]$, there exists $v\in \mathcal A(r,\theta)$ satisfying
\begin{equation} \label{e.unifapp}
\frac1r \sup_{x\in B_{r}} \left| u(x) - v(x) \right| \leq r^{-\alpha} \left( K + \frac1{2r} \osc_{B_{2r}} u \right). 
\end{equation}
Then there exists $\beta(\theta)>0$ and $C(\alpha,\theta)\geq 1$ such that, for every $s\in[r_0, R/2]$,
\begin{equation} \label{e.c01}
\frac 1s \osc_{B_s} u \leq   C \left( \frac1R \osc_{B_R} u  + \left( \frac sR\right)^{\alpha}K \right)
\end{equation}
and
\begin{multline} \label{e.c1beta}
\frac1s \inf_{p\in\Rd} \osc_{x\in B_s} \left( u(x) - p\cdot x \right) \leq C\bigg( \left( \frac sR\right)^{\beta} \frac1R \inf_{p\in\Rd} \osc_{x\in B_R} \left( u(x) - p\cdot x \right) \\
+ s^{-\alpha} \left( K+ \frac1R \osc_{B_R} u \right) \bigg) .
\end{multline}
\end{lemma}

The proof of the lemma is given below. We first apply it to show that Theorem~\ref{t.c01} is a consequence of Corollary~\ref{cor.DP}.

\begin{proof}[{Proof of Theorem~\ref{t.c01}}]
Let $M\geq 1$, $R\geq 2$ and $u\in H^1(B_R)$ satisfy~\eqref{e.osccontrol2} and~\eqref{e.localmin2}. Note that by the interior De Giorgi-Nash-Moser estimate, we have
\begin{equation*} \label{}
\frac{2}{R}\osc_{B_{R/2}} u \leq C M.
\end{equation*}
Fix $s\in (0,d)$ and let $\X$ be as in the statement of Corollary~\ref{cor.DP} with $M$ replaced by $C'M$ for $C'\geq 1$ to be selected below. Then, according to~\eqref{e.DPapprox}, there exists $C(d,\Lambda,s)\geq 1$ such that, for every $r\in [\X^{1/s} + C,R/4]$ such that 
\begin{equation*} \label{}
K_0 + \frac{1}{2r} \osc_{B_{2r}} u \leq C'M,
\end{equation*}
there exists a local minimizer $v\in H^1(B_{r})\cap L^\infty(B_r)$ of the homogenized functional satisfying
\begin{equation*} \label{}
\sup_{x\in B_r} \left| u(x) - v(x) \right| \leq r^{-\alpha(d-s)} \left( Kr+\osc_{B_{2r}}u \right).
\end{equation*}
Define the random variable $\Y:= \X^{1/s} + C$. Observe that $\Y$ satisfies~\eqref{e.Yinteg}.

\smallskip

We next establish that local minimizers of the homogenized functional satisfy the improvement of flatness property. Since $\overline L$ is uniformly convex by~\eqref{e.Lbarconvex}, there exists $\theta(d,\Lambda) \in (0,1/2]$ such that every local minimizer $w \in H^1(B_s)$ of the homogenized energy functional satisfies~\eqref{e.Artheta} for every $r\leq s/2$. This is a simple consequence of the interior $C^{1,\beta}$ estimate for uniformly convex energy functionals, which can be found in Giaquinta~\cite{G}, and a scaling argument. In the notation of Lemma~\ref{l.blackbox2}, we have that $w\in \cap_{0<r\leq s/2} \mathcal A(r,\theta)$. 

\smallskip

We claim that, for every $s\in [\Y,R/4]$,
\begin{equation*} \label{}
\frac1s \osc_{B_s} u \leq C \left( K + \frac1R \osc_{B_R} u \right). 
\end{equation*}
We argue by induction: let $n\in \N$ and suppose, for every $r \in \{ 2^{-j}R \,:\, j=2,\ldots,n\}$, 
\begin{equation} \label{e.finindy}
K_0 + \frac{1}{2r} \osc_{B_{2r}} u \leq \frac12 C'M.
\end{equation}
This implies that, for every $r\in [2^{-n}R,R/4]$,
\begin{equation*} \label{}
K_0 + \frac{1}{r} \osc_{B_{r}} u \leq C'M.
\end{equation*}
If $r\geq \Y$, then Lemma~\ref{l.blackbox2} gives
\begin{equation*} \label{}
K_0 + \frac{1}{r} \osc_{B_{r}} u \leq CM \leq \frac12 C'M,
\end{equation*}
if we select $C' = C(d,\Lambda,M,s)$ sufficiently large. Thus~\eqref{e.finindy} holds for $r=2^{-(n+1)}R$. By induction, we deduce that, for every $r\in [\Y,R/4]$,
\begin{equation*} \label{}
K_0 + \frac{1}{r} \osc_{B_{r}} u \leq C'M \leq CM.
\end{equation*}
This is~\eqref{e.c01est}. We get~\eqref{e.c1best} after applying the second conclusion of Lemma~\ref{l.blackbox2}. 
\end{proof}

We conclude this section with the proof of Lemma~\ref{l.blackbox2}.

\begin{proof}[{Proof of Lemma~\ref{l.blackbox2}}]
In this argument, $C$ and $c$ denote positive constants depending only on~$(\alpha,\theta)$ which may vary in each occurrence.

\smallskip

\emph{Step 1.} We setup the argument. Using the assumptions, we find that, for every $s\in [r_0/2,R/4]$,
\begin{equation*} \label{}
\frac1{\theta s} \inf_{p\in\Rd} \osc_{x\in B_{\theta s}} \left( u(x) - p\cdot x \right) \leq  \frac1{2} \left( \frac1s\inf_{p\in\Rd} \osc_{B_s} \left( u(x) - p\cdot x \right) \right)+C s^{-\alpha} \left( K + \frac1{2s} \osc_{B_{2s}} u \right).
\end{equation*}
Define $s_0:= R$ and, for each $j \in\N$, set $s_j:= \theta^{j-1} R/4$. Pick $m\in\N$ such that $s_{m+1} < r_0/2 \leq s_m$. The previous inequality yields, for every $j \in\{ 1,\ldots, m\}$,
\begin{equation} \label{e.bb2iter}
F_{j+1} \leq \frac12 F_j + C s_j^{-\alpha} (K+H_{j-1}),
\end{equation}
where we have set, for each $j\in\{ 0,\ldots,m\}$,
\begin{equation*} \label{}
F_j:= \frac1{s_j} \inf_{p\in\Rd} \osc_{x\in B_{s_j}} \left( u(x) - p\cdot x \right)\quad \mbox{and} \quad H_{j} := \frac{1}{s_j} \osc_{B_{s_j}} u 
\end{equation*}
Select $p_j\in \Rd$ such that 
\begin{equation*} \label{}
\frac{1}{s_j} \osc_{x\in B_{s_j}} \left( u(x) - p_j\cdot x \right) = F_j.
\end{equation*}
Observe that the triangle inequality gives the bounds, for every $j\in \{ 0,\ldots,m\}$:
\begin{equation} \label{e.FjHj}
F_j\leq H_j \leq 2|p_j| + F_j,
\end{equation}
\begin{equation} \label{e.pjtrivbnd}
|p_j| = \frac12\frac1{s_j} \osc_{x\in B_{s_j}} (p\cdot x) \leq \frac12F_j + \frac12H_j \leq H_j
\end{equation}
and, for every $j\in \{ 0,\ldots,m-1\}$,
\begin{align} \label{e.pjtriang}
\left|p_{j+1}-p_j \right|& = \frac12 \frac{1}{s_{j+1}} \osc_{x\in B_{s_{j+1}}} \left(p_{j+1}-p_j \right) \cdot x \\
& \leq \frac1{s_{j+1}} \osc_{x\in B_{s_{j+1}}} \left( u(x) - p_{j+1}\cdot x\right) + \frac1{s_{j+1}} \osc_{x\in B_{s_{j}}} \left( u(x) - p_{j}\cdot x\right) \nonumber\\
& = F_{j+1} + \frac{1}{\theta} F_j \leq C(F_{j+1}+F_j).\nonumber
\end{align}  
Note that~\eqref{e.pjtriang} gives $|p_{j+1}| \leq |p_j| + C(F_{j+1}+F_j)$ and hence, by induction and~\eqref{e.pjtrivbnd}, 
\begin{equation} \label{e.pjbndsumFi}
|p_j| \leq H_0 + C\sum_{i=0}^j  F_i.
\end{equation}
By~\eqref{e.bb2iter},~\eqref{e.FjHj} and~\eqref{e.pjbndsumFi}, we obtain, for every $j\in\{ 0,\ldots,m-1\}$,
\begin{equation} \label{e.bb2iter2}
F_{j+1} \leq \frac12 F_j + Cs^{-\alpha}_j \left(K+H_0+\sum_{i=0}^j F_i\right).
\end{equation}

\smallskip

\emph{Step 2.} By iterating~\eqref{e.bb2iter2}, we show that
\begin{equation} \label{e.bb3iterbby}
F_j \leq 2^{-j} F_0 + Cs_j^{-\alpha} \left( K+H_0 \right).
\end{equation}
Arguing by induction, we fix $A\geq 1$ be a constant to be selected below and suppose that $k\in \{ 0,\ldots,m-1\}$ is such that, for every $j\in \{0,\ldots,k\}$, 
\begin{equation} \label{e.bb3iterindhyp}
F_j \leq  2^{-j} F_0 + As_j^{-\alpha} \left( K+H_0 \right).
\end{equation}
Using~\eqref{e.bb2iter2} and this assumption, we find that 
\begin{align*}
F_{k+1} & \leq \frac12 F_k + Cs^{-\alpha}_k \left(K+H_0+\sum_{j=0}^k F_j\right) \\ 
& \leq 2^{-(k+1)} F_0 + \frac12 As_{k}^{-\alpha} \left( K + H_0 \right) \\
& \qquad + Cs_k^{-\alpha} \left( K + H_0 + \sum_{j=0}^k \left( 2^{-j}F_0 +As_{j}^{-\alpha} (K+H_0) \right) \right) \\
& \leq 2^{{-(k+1)}} F_0 + (K+H_0)s_{k+1}^{-\alpha} \left( \frac12 A + C + CAs_{k}^{-\alpha} \right).
\end{align*}
If $k \leq n$ where $n$ is such that $Cs_n^{-\alpha} \leq \frac14$, then we may select $A= C$ sufficiently large that
\begin{equation*} \label{}
 \frac12 A + C + CAs_{k}^{-\alpha} \leq \frac34A  + C \leq A.
\end{equation*}
In this case, we obtain
\begin{equation*} \label{}
F_{k+1} \leq 2^{-(k+1)}F_0 + A s_{k+1}^{-\alpha} (K+H_0).
\end{equation*}
This is~\eqref{e.bb3iterindhyp} for $j=k+1$. Since~\eqref{e.bb3iterindhyp} trivially holds for $j=0$, we obtain~\eqref{e.bb3iterbby} for all $j\in\{ 0,\ldots,n\}$ by induction. Since $1 \leq s_j / s_n \leq C$ and thus $F_j \leq CF_n$ for all $j\in \{ n+1,\ldots,m\}$, we conclude that~\eqref{e.bb3iterbby} holds for all $j\in \{ 0,\ldots, m\}$.

\smallskip

\emph{Step 3.} The conclusion. Notice that~\eqref{e.bb3iterbby} implies~\eqref{e.c1beta} for $\beta := (\log 2) / |\log \theta|$. To obtain~\eqref{e.c01}, we need to bound $H_j$, and this is obtained from~\eqref{e.pjbndsumFi} and~\eqref{e.bb3iterbby} as follows:
\begin{align*}
H_j \leq F_j + 2|p_j| & \leq 2H_0 + C\sum_{i=0}^j F_i \\
&\leq 2H_0+ \sum_{i=0}^j \left( 2^{-i} F_0 + Cs_i^{-\alpha} \left( K+H_0 \right) \right) \\
& \leq 2H_0 + F_0 + Cs_j^{-\alpha} (K+H_0)  \\
& \leq CH_0 + Cs_j^{-\alpha} K.
\end{align*}
This implies~\eqref{e.c01}.
\end{proof}

\appendix
\section{The proof of Proposition~\ref{p.blackbox} }
\label{ss.bb}
The argument for Proposition~\ref{p.blackbox} requires some ingredients from the classical regularity theory in the calculus of variations: we need (i) an interior $H^2$ estimate for minimizers of the homogenized energy functional and (ii) a global $W^{1,t}$ estimate (for Lipschitz domains), for some $t(d,\Lambda)>2$, for minimizers of both the heterogeneous and homogenized energy functionals. The latter estimate is a generalization to the nonlinear setting of the Meyers estimate for linear equations and can be found in~\cite{Giu}. The former can be found in either~\cite{G} or~\cite{Giu}.

\smallskip

We now fix some notation used in the rest of this subsection. It is convenient to rescale the functions in the hypotheses of Proposition~\ref{p.blackbox} so that the microscopic scale is of unit size and the ratio $\ep >0$ of the length scales is reflected in the size of the domain. We thus fix a (large) bounded Lipschitz domain $U\subseteq \Rd$ and take $n,m\in\N_*$ such that 
\begin{equation} \label{e.Unorm}
3^{d(n+m)} < |U| \leq 3^{d(n+m+1)}. 
\end{equation}
Essentially, this means that $3^{-(n+m)} \approx \ep$, i.e., the macroscopic scale is of order~$3^{n+m}$. We will take the mesoscopic scale to be of order~$3^n$ rather than $3^{n-m}$, as in the statement of the proposition. We fix one more parameter $l \in \N$ such that
\begin{equation*}
n\leq l \quad \mbox{and}  \quad 2l \leq m+n,
\end{equation*}
which describes the (mesoscopic) thickness of a boundary strip we need to remove in the approximation argument. (In practice, we typically choose $m = \lceil cn\rceil$ for a small $0<c\ll 1$, and $l$ roughly equidistant between $n$ and $n+m$.) We also denote the normalized domain by $\widehat U:=U / |U|$.

\smallskip

Define domains $V^\circ \subseteq V \subseteq U$ by 
\begin{equation*} \label{}
V := \left\{ x \in \Rd\,:\, x+Q_{l+2} \subseteq U \right\} \quad \mbox{and} \quad V^\circ := \left\{ x \in \Rd\,:\, x + Q_{l+4} \subseteq U \right\}.
\end{equation*}
Since $U$ is a Lipschitz domain, we have
\begin{equation} \label{e.bdrystrip}
\left| U \setminus V^\circ \right| \leq C 3^{l-m-n}  |U|,
\end{equation}
where the constant $C$ depends only on $d$ and $\widehat U$. We note also that $\dist(V^\circ,\partial V) \geq 3^l$. Denote by $\eta \in C^{\infty}_0(\Rd)$ a cutoff function satisfying
\begin{equation} \label{e.eta}
0 \leq \eta \leq 1,  \quad \eta \equiv 1 \ \mbox{on} \ \overline {V^\circ}, \quad \eta \equiv 0 \ \mbox{in} \ \Rd \setminus V \quad \mbox{and} \quad |D\eta| \leq C3^{-l}.
\end{equation}

\smallskip

Throughout we fix $L\in \Omega(K_0)$, $t>2$, $g\in W^{1,t}(U)$ and denote by $u\in g+ H^1_0(U)$ the unique minimizer of the heterogeneous energy functional: that is, $u$ satisfies
\begin{equation*} \label{}
\int_{U} L\left(Du(x), x \right) \,dx \leq \int_{U} L\left(Dw(x),x \right) \, dx \quad \mbox{for every} \ w \in g + H^1_0(U).
\end{equation*}
Also take $u_{\mathrm{hom}} \in g+H^1_0(U)$ to be the unique minimizer of the constant-coefficient energy functional, i.e., $v$ satisfies
\begin{equation*}
\int_{U} \overline L(Du_{\mathrm{hom}}(x)) \,dx \leq \int_{U} \overline L(Dw(x)) \, dx \quad \mbox{for every} \ w \in g + H^1_0(U).
\end{equation*}
For convenience, we denote
\begin{equation*} \label{}
M:= K_0 + \left( \fint_U |Dg(x)|^t\, dx\right)^{1/t}.
\end{equation*}
as well as
\begin{equation*} \label{}
\mathcal E':= \sup \left\{ \mathcal E(y+Q_{n},p) + \mathcal E(y'+Q_{n+2},p') \,:\, y,y'\in B_{C3^{n+m}}, \ p,p'\in B_{CM3^{md/2}} \right\}, 
\end{equation*}
which is precisely the rescaled version of $\mathcal E'$ defined in the statement of Proposition~\ref{p.blackbox}. The convention for the constants $C$ and $c$ in this appendix is that they depend on $(d,\Lambda,t,\widehat U)$ and may vary in each occurrence. 

\smallskip

In view of the above notation and scaling convention, to obtain Proposition~\ref{p.blackbox} it suffices to prove the estimate
\begin{multline}\label{e.BBwts}
\left| \fint_U \left( L\left(Du(x), x \right) - \overline L(D u_{\mathrm{hom}}(x)) \right)\,dx \right| + 3^{-2(n+m)}\fint_U (u (x)- u_{\mathrm{hom}}(x))^2 \,dx  \\
 \leq   C\mathcal E' + C M^2\left( 3^{n-l} + 3^{\beta (l-n-m)} \right),
\end{multline}
The main step in the proof of~\eqref{e.BBwts} is to show that the (heterogeneous) energy of $u$ is very close to the (effective) energy of $u_{\mathrm{hom}}$. The plan is to modify each minimizer in order to get a candidate for a local minimizer of the other's functional, and thus an upper bound for the energy of the other. These steps are summarized in Lemmas~\ref{l.heteromod} and~\ref{l.homomod}, below. 

\subsection{Classical regularity estimates}
\label{s.regest}

Before proceeding with the proof of~\eqref{e.BBwts}, we record here the needed estimates from regularity theory. According to the Meyers estimate~\cite[Theorem 6.8]{Giu}, there exists $r(d,\Lambda,t)\in (2,t]$ such that $u,\,u_{\mathrm{hom}}\in W^{1,r}(U)$ and 
\begin{equation} \label{e.meyapp}
\left( \fint_{U} \left| Du(x) \right|^r  \, dx \right)^{1/r} + \left( \fint_U  \left| Du_{\mathrm{hom}}(x) \right|^r\, dx \right)^{1/r} \leq CM.
\end{equation}
We also need the interior $H^2$ estimate~\cite[Theorem 8.1]{Giu} for solutions of constant coefficient functionals which, together with an easy covering argument (using that $\dist(V,\partial U) \geq c3^l$) implies~$u_{\mathrm{hom}} \in H^2(V)$ and the gives the estimate
\begin{equation} \label{e.H2app}
 3^{l} \left( \fint_{V}  \left| D^2u_{\mathrm{hom}}(x) \right|^2\, dx \right)^{1/2} \leq C M.
\end{equation}
We remark that while~\cite[Theorem 6.8]{Giu} does not include the~\emph{a priori} estimate we need (rather just the statement that the functions belong to $W^{1,r}(U)$), but it can be extracted from the proof there. Also, the hypotheses in~\cite{Giu} are slightly stronger, namely they demand that the integrand~$L(p,x)$ be $C^2$ in the~$p$ variable rather than just $C^{1,1}$. However, by inspecting the arguments, one finds that the estimates do not depend on the modulus of continuity of $D^2_pL(\cdot,x)$, rather only on an upper bound for $|D^2L(\cdot,x)| =  \left[ D_pL (\cdot,x)\right]_{C^{0,1}}$, which in our case is bounded above by~$\Lambda$. Therefore we obtain the results we need from~\cite{Giu} after a routine approximating argument (by smoothing the coefficients).


\subsection{Estimate for the (homogenized) energy of $u_{\mathrm{hom}}$}

This is the easier of the two directions. The idea is to remove the microscopic oscillations from $u$, and for this it is natural to consider a spatial average on a mesoscopic scale. We thus define
\begin{equation*} \label{}
\xi (y):= \fint_{y+Q_n} u(x)\, dx,\quad y\in V.
\end{equation*}
Notice that $\xi \in H^1(V)$. We next modify $\xi$ in order to get an element of $g+H^1_0(U)$ by setting 
\begin{equation*} 
\tilde u(x) := \eta(x) \xi(x) + \left(1-\eta(x)\right) u(x),
\end{equation*}
It is clear that~$\tilde u \in g+ H^1_0(U)$. We estimate the (homogenized) energy of~$\tilde u$ in terms of the heterogeneous energy of $u$ and the error term $\mathcal E'$ defined above.

\begin{lemma}
\label{l.heteromod}
There exists $C(d,\Lambda,t,\widehat U) \geq 1$ and $r(d,\Lambda,t)\in (2,t]$ such that
\begin{equation} \label{e.estim1}
\fint_{U} \overline L(D\tilde u(x)) \,dx \\ 
\leq \fint_U L(Du(x),x)\,dx  + C\mathcal E' + CM^2 3^{(l-n-m)(1-2/r)}.
\end{equation}
\end{lemma}
\begin{proof}
We divide the argument into several steps. 

\smallskip

\emph{Step 1.} We show that $\xi \in W^{1,\infty}(V)$. Denote, for each $y\in V$,
\begin{equation*} \label{}
p(y):= D\xi(y) = \fint_{y+Q_n} Du(x) \, dx \quad \mbox{and} \quad q(y):= D\overline L(p(y)).
\end{equation*}
Observe that, by H\"older's inequality and~\eqref{e.Unorm}, for each $y\in V$,
\begin{multline} \label{e.pybnd}
\left| p(y) \right|^2 \leq \left(\fint_{y+Q_n} \left| Du(x) \right|\, dx\right)^2 \leq  \frac{|U|}{|Q_n|} \left( \fint_{U} |Du(x)|^2\,dx \right)\\
 \leq C3^{md}\left( \fint_{U} |Du(x)|^2\,dx \right) \leq C 3^{md} M^2.
\end{multline}

\smallskip

\emph{Step 2.} 
We use~\eqref{e.meyapp} to show that 
\begin{equation} \label{e.utildeu}
\fint_{U} \left| D\tilde u(x) \right|^r  \, dx  \leq C \fint_{U} \left| D u(x) \right|^r  \, dx.
\end{equation}
Differentiating the formula above for $\tilde u$, we find
\begin{equation*} \label{}
D\tilde u(x) = Du(x) + \eta(x) \left( D\xi(x) - Du(x) \right) + D\eta(x) \left( \xi(x) - u(x) \right)
\end{equation*}
and thus, by~\eqref{e.eta},
\begin{multline*} \label{}
\left( \fint_{U} \left| D\tilde u(x) \right|^r  \, dx \right)^{1/r}  \leq \left( \fint_{U} \left| D u(x) \right|^r  \, dx \right)^{1/r} + \left( \fint_V \left| D\xi(x) - Du(x)  \right|^r \, dx\right)^{1/r} \\ + C3^{-l} \left( \fint_V \left| \xi(x) - u(x)  \right|^r \, dx \right)^{1/r}.
\end{multline*}
We obtain~\eqref{e.utildeu} from the previous inequality and the following (recall $l\geq n$):
 \begin{equation} \label{e.upone}
\fint_V \left| \xi(x) - u(x)  \right|^r \, dx \leq C3^{nr} \fint_U \left| Du(x) \right|^r \, dx
\end{equation}
and
 \begin{equation} \label{e.uptwo}
\fint_V \left| D\xi(x) - Du(x)  \right|^r \, dx \leq C \fint_U \left| Du(x) \right|^r \, dx.
\end{equation}
The second estimate~\eqref{e.uptwo} follows from the triangle inequality and
\begin{multline*} \label{}
\fint_{V} \left| D\xi(x) \right|^r\,dx  = \fint_{V} \left| \fint_{x+Q_n} Du(y)\, dy \right|^r\, dx \\
 \leq \frac{|U|}{|V|} \fint_U \left| Du(x) \right|^r\,dx \leq C \fint_U \left| Du(x) \right|^r\,dx.
\end{multline*}
To get~\eqref{e.upone}, we use that, for every $z\in Q_{n+1}$,
\begin{multline*} \label{}
\left( \fint_{Q_{n+1}(z)} \left| u(x) - \xi(x) \right|^r\, dx \right)^{1/r}  \leq \left( \fint_{Q_{n+1}(z)} \left| u(x) - \fint_{Q_n(x)} u(y) \, dy \right|^r \,dx \right)^{1/r} \\ 
+ \left( \fint_{Q_{n+1}(z)} \left| \fint_{Q_n(x)} u(y) \, dy -\xi(x) \right|^r \,dx \right)^{1/r}. 
\end{multline*}
The first term on the right is estimated by the Poincar\'e inequality:
\begin{align*} \label{}
\fint_{Q_{n+1}(z)} \left| u(x) - \fint_{Q_n(x)} u(y) \, dy \right|^r \,dx & = \fint_{Q_{n+1}(z)} \fint_{Q_n(\xi)} \left| u(x) - \fint_{Q_n(\xi)} u(y) \, dy \right|^r \,dx\, d\xi  \\
& \leq C3^{nr} \fint_{Q_{n+1}(z) } \fint_{Q_n(\xi) } \left| Du(y) \right|^r\, dy\, d\xi \\
& = C3^{nr} \fint_{Q_{n+1}(z)}  \left| Du(x) \right|^r\, dx,
\end{align*}
while, to estimate the second, we observe that, for every $x+Q_n, x'+Q_n \subseteq Q_{n+1}(z)$,
\begin{align*} \label{}
\left| \fint_{x+Q_n} u(y)\, dy - \fint_{x'+Q_n} u(y)\, dy \right| & = \left| \fint_{Q_n} \int_0^1 (x-x')\cdot Du(tx+ (1-t)x' + y)\, dt\, dy \right|  \\
& \leq 3^d \left|x-x'\right| \fint_{Q_{n+1}(z)} \left|Du(y)\right|\, dy \\
& \leq C 3^n \fint_{Q_{n+1}(z)} \left|Du(y)\right|\, dy.
\end{align*}
This yields 
\begin{align*}
\fint_{Q_{n+1}(z)} \left| \fint_{Q_n(x)} u(y) \, dy -\xi(x) \right|^r \,dx \leq C3^n \fint_{Q_{n+1}(z)} \left| Du(y) \right|^r\, dy
\end{align*}
and completes the proof of~\eqref{e.utildeu}.

\smallskip

\emph{Step 3.} We prove that
\begin{multline} \label{e.LslashDu1}
\fint_{U} L(Du(x),x)\, dx 
\geq \fint_{V^\circ} \left( \mu(y+Q_n,q(y)) +p(y)\cdot q(y)\right)\,dy \\  - CM^2 3^{(l-n-m)(1-2/r)} .
\end{multline}
Taking $u$ as a minimizer candidate in the definition of $\mu(y+Q_n,q(y))$ and using that $p(y)$ is dual to $q(y)$, we have
\begin{align*}
\lefteqn{\fint_{V^\circ} L(Du(x),x)\, dx} \qquad & \\
& \geq \fint_{V^\circ} \fint_{y+Q_n} L(Du(x),x)\, dx\,dy - \frac{1}{|V^\circ|} \int_{U \setminus V^\circ} \left| L(Du(x),x) \right| \, dx \\
& \geq  \fint_{V^\circ} \left( \mu(y+Q_n,q(y)) +p(y)\cdot q(y)\right)\,dy - \frac{C}{|U|} \int_{U \setminus V^\circ}\left(  K_0+\left| Du \right|\right)^2 \, dx.
\end{align*}
We next handle the error arising from the boundary strip. Using~\eqref{e.bdrystrip}, we have 
\begin{align*}
\lefteqn{ \fint_U L(Du(x),x)\,dx -  \fint_{V^\circ} L(Du(x),x)\,dx}  \qquad & \\
& \geq  - C3^{l-n-m} \left| \fint_U L(Du(x),x)\,dx \right|  - \frac{1}{|U|} \left| \int_{U\setminus V^\circ} L(Du(x),x)\,dx \right| \\
& \geq -C3^{l-n-m} \fint_U \left( K_0 + |Du(x)| \right)^2\,dx - \frac{C}{|U|} \int_{U\setminus V^\circ} \left( K_0 + |Du(x)| \right)^2\,dx. 
\end{align*}
Assembling the previous two strings of inequalities gives~\eqref{e.LslashDu1}, after we estimate the error terms in the following way: by~\eqref{e.meyapp}, we have
\begin{equation*} \label{}
\fint_U \left( K_0 + |Du(x)| \right)^2\,dx \leq CM^2
\end{equation*}
and then H\"older's inequality and~\eqref{e.meyapp} give 
\begin{align*} \label{}
\frac{1}{|U|} \int_{U\setminus V^\circ} \left( K_0 + |Du(x)| \right)^2\,dx & \leq \left(\frac{\left|U\setminus V^\circ\right|}{|U|} \right)^{1-2/r} \left( \fint_U  \left( K_0 + |Du(x)| \right)^r\,dx \right)^{2/r}  \\
& \leq C3^{(l-n-m)(1-2/r)} M^2.
\end{align*}
This completes the proof of~\eqref{e.LslashDu1}. 

\smallskip

\emph{Step 4.} The conclusion. According to~\eqref{e.LslashDu1}, 
\begin{align*}
\lefteqn{ \fint_U L(Du(x),x)\,dx - \fint_{V^\circ} \overline L(p(y)) \,dy }  \qquad & \\
& \geq - \fint_{V^\circ} \mathcal E(y+Q_n,p(y)) \, dy - CM^2 3^{(l-n-m)(1-2/r)} \\
& \geq - \mathcal E' - CM^2 3^{(l-n-m)(1-2/r)}.
\end{align*}
It remains to use~\eqref{e.meyapp},~\eqref{e.utildeu} and H\"older's inequality to estimate the energy of $\tilde u$ in the boundary strip, much as we did above for $u$ in Step~3. We have
\begin{multline*} \label{}
\fint_U  \overline L(D\tilde u(y)) \,dy - \fint_{V^\circ}  \overline L(p(y)) \,dy \\
 = - \frac{|U\setminus V^\circ|}{|U|} \fint_{V^\circ}  \overline L(p(y)) \,dy + \frac{1}{|U|} \int_{U\setminus V^\circ} \overline L(D\tilde u(y)) \,dy
\end{multline*}
and we estimate the error terms as follows: by~\eqref{e.utildeu},
\begin{equation*} \label{}
\frac{|U\setminus V^\circ|}{|U|} \left| \fint_{V^\circ}  \overline L(p(y)) \,dy\right| = \frac{|U\setminus V^\circ|}{|U|}\fint_{V^\circ}  \left|   \overline L(D\tilde u(y))\right|  \,dy  \leq CM^23^{l-n-m}
\end{equation*}
and, using~\eqref{e.utildeu} and H\"older's inequality again,
\begin{equation*} \label{}
 \frac{1}{|U|} \int_{U\setminus V^\circ} \overline L(D\tilde u(y)) \,dy \leq CM^23^{(l-n-m)(1-2/r)}. 
\end{equation*}
Combining these gives the lemma.
\end{proof}

\subsection{Estimate of the (heterogeneous) energy of $u$}

We next modify $u_{\mathrm{hom}}$ to obtain a minimizer candidate $\tilde u_{\mathrm{hom}}$ for the heterogeneous energy functional by a stitching together mesoscopic minimizers on an overlapping grid, not unlike the patching construction in the proof of Lemma~\ref{l.patching}. 

\smallskip

We begin the construction by defining an affine approximation to $u_{\mathrm{hom}}$ in the mesoscopic cube $z+Q_{n+2}$ by setting, for each $z \in V\cap 3^n\Zd$,
\begin{equation*} \label{}
\ell_z(x):=p(z) \cdot (x-z)  + \zeta(z),
\end{equation*}
where 
\begin{equation*} \label{}
\zeta(z) := \fint_{z+Q_{n+2}} u_{\mathrm{hom}} (x)\,dx \quad \mbox{and} \quad p(z):= D\zeta(z) = \fint_{z+Q_{n+2}} Du_{\mathrm{hom}} (x) \, dx.
\end{equation*}
For each $z\in V \cap 3^n\Zd$, we introduce mesoscopic minimizers in $z+Q_{n+2}$ of the heterogeneous energy functional with Dirichlet boundary conditions given by $l_z$:
\begin{equation*} \label{}
v_{z}(x):= v\left( x, z+Q_{n+2},D\zeta(z) \right) - D\zeta(z)\cdot z+ \zeta(z), \quad x\in z+ Q_{n+2}(z).
\end{equation*}
Observe that, for each $z\in V \cap 3^n\Zd$, we have $v_z \in H^1(z+Q_{n+2})$, 
\begin{equation*} \label{}
\fint_{z+Q_{n+2}} \left( v_z(x) - u_{\mathrm{hom}}(x) \right)\,dx = 0  \quad \mbox{and} \quad \fint_{z+Q_{n+2}} \left( Dv_z(x) - Du_{\mathrm{hom}}(x) \right)\,dx=0.
\end{equation*}
Next we patch these functions together to get a function defined on $V$. Define a smooth periodic partition of unity by setting
\begin{equation*} \label{}
\phi(x):= \int_{Q_n} \eta(x-y)\,dy.
\end{equation*}
where $\eta \in C^{\infty}(\Rd)$ satisfies
\begin{equation*} \label{}
0 \leq \eta \leq 1, \quad \eta \equiv 0 \ \mbox{in} \ \Rd \setminus Q_{n-1}, \ \quad \int_{\Rd} \eta(y)\,dy =1, \quad \mbox{and} \quad |D\eta| \leq C3^{-(d+1)n}.
\end{equation*}
Here we have essentially mollified the characteristic function of the cube $Q_n$ to obtain a function $\phi$, which is supported in $Q_{n+1}$ and satisfies
\begin{equation} \label{e.phibounds}
\sup_{x\in Q_{n+1}} \left|D\phi(x)\right| \leq C3^{-(d+1)n} |Q_n| \leq C3^{-n}
\end{equation}
and, for every $x\in\Rd$, 
\begin{equation} \label{e.phiunity}
\sum_{z\in 3^n\Zd} \phi(x-z) = 1.
\end{equation}
The latter holds since the cubes $\{ z+Q_n\,:\, z\in 3^n\Zd\}$ form a disjoint partition of~$\Rd$, up to a set of Lebesgue measure zero. Now set
\begin{equation*} \label{}
\tilde v(x):= \sum_{z\in V \cap 3^n\Zd} \phi(x-z) v_z(x), \quad x\in V.
\end{equation*}
Then $\tilde v\in H^1(V)$. Finally, we modify $\tilde v$ to match the boundary condition. Take $\eta \in C^\infty_0(\Rd)$ to be the cutoff function satisfying~\eqref{e.eta}, as above, and set
\begin{equation*} \label{}
\tilde u_{\mathrm{hom}} (x) := \eta(x) \tilde v(x) + \left(1 - \eta(x) \right) u_{\mathrm{hom}} (x).
\end{equation*}
Then $\tilde u_{\mathrm{hom}} \in g+H^1(U)$ is the minimizer candidate for the heterogeneous energy functional. We estimate its energy from above using an argument similar to the one in Lemma~\ref{l.patching} combined with some aspects of the proof of Lemma~\ref{l.heteromod} and relying on Theorem~\ref{t.mulimit}. This result is summarized in the following lemma. 

\begin{lemma}
\label{l.homomod}
There exists $C(d,\Lambda,t,\widehat U)\geq 1$ and $r(d,\Lambda,t)\in(2,t]$ such that 
\begin{multline*} 
 \fint_{U} L\left( D\tilde u_{\mathrm{hom}}(x),x \right)\, dx 
  \leq \fint_{U} \overline L( Du_{\mathrm{hom}}(x)) \, dx \\ + C\mathcal E'
 + CM^2 \left( 3^{n-l} +3^{(l-n-m)(1-2/r)} \right).
\end{multline*}
\end{lemma}

\begin{proof}
We divide the proof into several steps.

\smallskip

\emph{Step 1.}
We record some estimates on the mesoscopic affine approximations which are needed below. By H\"older's inequality and~\eqref{e.Unorm}, we have, for each $z\in V$,
\begin{multline} \label{e.pybnd2}
\left| p(z) \right|^2 \leq \left(\fint_{z+Q_{n+2}} \left| Du_\mathrm{hom}(x) \right|\, dx\right)^2 \leq  \frac{|U|}{|Q_{n+2}|} \left( \fint_{U} |Du_{\mathrm{hom}}(x)|^2\,dx \right)\\
 \leq C3^{md}\left( \fint_{U} \left|Du_{\mathrm{hom}} (x) \right|^2\,dx \right) \leq C 3^{md} M^2.
\end{multline}
In view of the definition of $v_z$ and the previous inequality, we obtain, for $z\in V\cap \Zd$,
\begin{equation} \label{e.firlinapp}
\left| \fint_{z+Q_{n+2}} L(Dv_z(x),x) \,dx - \overline L(p(z)) \right| 
\leq  \mathcal E'.
\end{equation}
We next show that overlapping mesoscopic affine approximations to $u_{\mathrm{hom}}$ agree, up to a small error. Compute, for every $z,z'\in V\cap 3^n\Zd$ such that $z'\in z+Q_{n+1}$,
\begin{align*} \label{}
p(z) - p(z') & = \fint_{z+Q_{n+2}} \left( Du_{\mathrm{hom}}(x) - Du_{\mathrm{hom}}(x+z'-z) \right)\, dx \\
& = \fint_{z+Q_{n+2}} (z-z')\cdot \int_0^1 D^2 u_{\mathrm{hom}} (x+(1-t)(z'-z))\, dt\, dx
\end{align*}
and, after changing the order of integration, applying Jensen's inequality and using $|z-z'|\leq C3^{n}$, obtain
\begin{equation} \label{e.pypz}
\left|p(z) - p(z')\right|^2 \leq C3^{2n} \fint_{z+Q_{n+3}} \left|  D^2u_{\mathrm{hom}}(x) \right|^2 \, dx.
\end{equation}
A similar computation yields
\begin{equation} \label{e.zz}
\left| \zeta(z) - \zeta(z') \right|^2 \leq C3^{4n} \fint_{z+Q_{n+3}} \left|  D^2u_{\mathrm{hom}}(x) \right|^2 \, dx.
\end{equation}
Combining these, we get control over the differences of the affine approximations:
\begin{equation} \label{e.affineagree}
\sup_{x\in z+Q_{n+3}} \left| \ell_z(x) - \ell_{z'}(x) \right|^2 \leq C 3^{4n} \fint_{z+Q_{n+3}} \left|  D^2u_{\mathrm{hom}}(x) \right|^2 \, dx.
\end{equation}
For our reference, we also return to~\eqref{e.pybnd2} and finish the estimate differently, using Jensen's inequality to obtain, for all such $z,z'\in V$,
\begin{equation} \label{e.pz}
\left|p(z') \right|^2 \leq C \fint_{z'+Q_{n+2}} \left|  Du_{\mathrm{hom}}(x) \right|^2 \, dx \leq C  \fint_{z+Q_{n+3}} \left|  Du_{\mathrm{hom}}(x) \right|^2 \, dx.
\end{equation}

\smallskip

\emph{Step 2.} We show that two mesoscopic minimizers $v_z$ and $v_{z'}$ with overlapping domains agree, up to a small error, on the overlap. The claim is that, for every $z,z'\in V \cap 3^n\Zd$ and $y\in 3^n\Zd$ such that $y+Q_n \subseteq (z+Q_{n+2})\cap (z'+Q_{n+2})$, we have 
\begin{multline} \label{e.mesogradwts}
3^{-2n} \fint_{y+Q_n} \left( v_{z'}(x) - v_z(x)  \right)^2\, dx  + \fint_{y+Q_n} \left| Dv_z(x) - Dv_{z'}(x) \right|^2\, dx \\
 \leq C\mathcal E' + C3^{n-l} \left( K_0^2 + \fint_{z+Q_{n+3}} \left| Du_{\mathrm{hom}}(x) \right|^2\, dx \right) 
+  C3^{n+l}\fint_{z+Q_{n+3}} \left| D^2u_{\mathrm{hom}}(x) \right|^2\,dx.
\end{multline}
The estimate of the first term on the left side of~\eqref{e.mesogradwts} follows from~\eqref{e.affineagree}, the definition of $\mathcal E'$, the triangle inequality and the fact that~$n\leq l$. Therefore we focus on estimating the difference of gradient overlaps. 

We first show that the energy of each mesoscopic minimizer $v_z$ spreads evenly in the $3^{2d}$ subcubes of $z+Q_{n+2}$ which are proportional to $Q_n$. 
We may these enumerate these subcubes by $y+Q_n$ for $y\in z+J_n$ where $J_n:= 3^n\Zd \cap Q_{n+2}$. Now compute, for $y\in z+J_n$, 
\begin{align*}
\lefteqn{ \fint_{y+Q_{n}} \left(  L(Dv_z(x),x) - q(y)\cdot Dv_z(x) \right)\,dx } \qquad &\\
& = 3^{2d} \fint_{z+Q_{n+2}} L(Dv_z(x),x)\,dx - 3^{2d} q(y) \cdot p(z) \\
& \qquad - \sum_{y'\in z+ J_n\setminus \{ y \}} \fint_{y'+Q_n}\left( L(Dv_z(x),x) - q(y) \cdot Dv_z(x)  \right)\,dx \\
& \leq 3^{2d} \nu(z+Q_{n+2},p(z)) - 3^{2d} q(y) \cdot p(z) - \sum_{y'\in z+ J_n\setminus \{ y \}}\mu\!\left(y'+Q_{n},q(y)\right).
\end{align*}
Using~\eqref{e.nucontp} and the triangle inequality, we obtain
\begin{align*}
\lefteqn{ \fint_{y+Q_{n}} \left(  L(Dv_z(x),x) - q(y)\cdot Dv_z(x) \right)\,dx } \qquad &\\
& \leq \mu(y+Q_n,q(y)) + C \mathcal E\left(z+Q_{n+2},p(y) \right) + \sum_{y'\in z+J_n} \mathcal E \left( y'+Q_n,p(y) \right) \\
& \qquad + C\left(K_0+|p(z)|+|p(y)| \right)\left|p(y) - p(z) \right|  + C\left|q(y) \right| \left| p(y) - p(z)\right| . 
\end{align*}
From the previous inequality,~\eqref{e.pypz} and~\eqref{e.pz}, we get
\begin{multline}\label{e.energvz}
 \fint_{y+Q_{n}} \left(  L(Dv_z(x),x) - q(y)\cdot Dv_z(x) \right)\,dx - \mu(y+Q_n,q(y)) \\
 \leq  C\mathcal E' + C3^{n} \left( K_0^2 + \fint_{z+Q_{n+3}} \left| Du_{\mathrm{hom}}(x) \right|^2\, dx \right)^{1/2} \left( \fint_{z+Q_{n+3}} \left| D^2u_{\mathrm{hom}}(x) \right|^2\,dx \right)^{1/2}.
\end{multline}
We now apply Lemma~\ref{l.convexL2} to obtain, for any $z,z'\in 3^n\Zd \cap V$ and $y\in 3^n\Zd$ such that $y+Q_{n} \subseteq (z+Q_{n+2}) \cap (z'+Q_{n+2})$,
\begin{multline*} \label{}
\fint_{y+Q_n} \left| Dv_z(x) - Dv_{z'}(x) \right|^2 \, dx \\
\leq C\mathcal E' + C3^{n} \left( K_0^2 + \fint_{z+Q_{n+3}} \left| Du_{\mathrm{hom}}(x) \right|^2\, dx \right)^{1/2} \left( \fint_{z+Q_{n+3}} \left| D^2u_{\mathrm{hom}}(x) \right|^2\,dx \right)^{1/2}.
\end{multline*}
Young's inequality gives the desired estimate for the $L^2$ difference of the gradients and completes the proof of~\eqref{e.mesogradwts}.

\smallskip

For future reference, we note that
\begin{multline}\label{e.energvz2}
\left| \fint_{z+Q_{n}}  Dv_z(x) \,dx - p(z) \right|^2  \\
 \leq  C\mathcal E' + C3^{n-l} \left( K_0^2 + \fint_{z+Q_{n+3}} \left|Du_{\mathrm{hom}}(x) \right|^2\, dx \right) 
+  C3^{n+l}\fint_{z+Q_{n+3}} \left| D^2u_{\mathrm{hom}}(x) \right|^2\,dx.
\end{multline}
To see this, we use Lemma~\ref{l.convexL2} and~\eqref{e.energvz} to compare the gradients of $v_z$ and the minimizer of $\mu(z+Q_n,q(z))$ and then apply~\eqref{e.uslopecomp}, before using Young's inequality as above.

\smallskip

\emph{Step 3.} We use Lemma~\ref{l.converseL2} and~\eqref{e.mesogradwts} to derive an upper bound for the (heterogeneous) energy of $\tilde v$ in $V$. The claim is that 
\begin{equation} \label{e.tilduhomW}
 \fint_{V} L\left( D\tilde v(x),x \right)\, dx \leq \fint_{V} \overline L( Du_{\mathrm{hom}}(x)) \, dx + C\mathcal E' + C3^{n-l}M^2.
  \end{equation}
For each $z \in W\cap 3^n\Zd$ and $x\in Q_n(z) \subseteq W$, we have
\begin{equation*} \label{}
\tilde v(x) - v_z(x) =  \sum_{y \in z+J_n'} \phi(x-y) \left(v_y(x) -   v_z(x)  \right)
\end{equation*}
where we denote $J_n':= \left\{ y\in 3^n\Zd\,:\, (y+J_n)\cap J_n \neq \emptyset \right\}$, which we observe has at most $C$ elements. Differentiating this expression yields, for all such $z$ and $x$, 
\begin{multline*} \label{}
D\tilde v(x) - Dv_z(x) \\ = \sum_{y \in z+J_n'} \left( D\phi(x-y) \left( v_y(x) - v_z(x)  \right) + \phi(x-y) \left( Dv_y(x) - Dv_z(x) \right) \right).
\end{multline*}
and then applying~\eqref{e.mesogradwts}, using the bound $|D\phi| \leq C3^{-n}$ from~\eqref{e.phibounds}, we obtain
\begin{multline} \label{e.tuhmgvz}
 \fint_{z+Q_n} \left| D\tilde v(x) - Dv_z(x) \right|^2\, dx \\
  \leq C\mathcal E' + C3^{n-l} \left( K_0^2 + \fint_{z+Q_{n+3}} \left| Du_{\mathrm{hom}}(x) \right|^2\, dx \right) 
+  C3^{n+l}\fint_{z+Q_{n+3}} \left| D^2u_{\mathrm{hom}}(x) \right|^2\,dx.
\end{multline}
The previous inequality,~\eqref{e.energvz} with $y=z$ and Lemma~\ref{l.converseL2} yield
\begin{multline*} \label{}
\fint_{z+Q_n} \left( L(D\tilde v(x),x) - q(z) \cdot D\tilde v(x) \right) \,dx - \mu(z+Q_n,q(z) ) \\
\leq C\mathcal E' + C3^{n-l} \left( K_0^2 + \fint_{z+Q_{n+3}} \left| Du_{\mathrm{hom}}(x) \right|^2\, dx \right) 
+  C3^{n+l}\fint_{z+Q_{n+3}} \left| D^2u_{\mathrm{hom}}(x) \right|^2\,dx.
\end{multline*}
By~\eqref{e.energvz2} and~\eqref{e.tuhmgvz}, we have
\begin{multline*} \label{}
\left| \fint_{z+Q_n} D\tilde v(x) \,dx - p(z) \right|^2 \\
  \leq C\mathcal E' + C3^{n-l} \left( K_0^2 + \fint_{z+Q_{n+3}} \left|Du_{\mathrm{hom}}(x) \right|^2\, dx \right) 
+  C3^{n+l}\fint_{z+Q_{n+3}} \left| D^2u_{\mathrm{hom}}(x) \right|^2\,dx.
\end{multline*}
The previous two lines yield
\begin{multline*} \label{}
\fint_{z+Q_n}  L(D\tilde v(x),x)  \,dx -  q(z)\cdot p(z)-\mu(z+Q_n,q(z) ) \\
\leq C\mathcal E' + C3^{n-l} \left( K_0^2 + \fint_{z+Q_{n+3}} \left| Du_{\mathrm{hom}}(x) \right|^2\, dx \right) 
+  C3^{n+l}\fint_{z+Q_{n+3}} \left| D^2u_{\mathrm{hom}}(x) \right|^2\,dx.
\end{multline*}
The sum of the last two terms on the left side is $-\overline L(p(z))$, up to an error of~$\mathcal E'$. Using this and~\eqref{e.firlinapp}, we get 
\begin{multline} \label{e.bleck}
\fint_{z+Q_n}  L(D\tilde v(x),x)  \,dx - \overline L(p(z)) \\
\leq C\mathcal E' + C3^{n-l} \left( K_0^2 + \fint_{z+Q_{n+3}} \left| Du_{\mathrm{hom}}(x) \right|^2\, dx \right) 
+  C3^{n+l}\fint_{z+Q_{n+3}} \left| D^2u_{\mathrm{hom}}(x) \right|^2\,dx.
\end{multline}
By the Poincar\'e inequality,
\begin{equation*} \label{}
\fint_{z+Q_{n+2}} \left| Du_{\mathrm{hom}}(x) - p(z) \right|^2\,dx \leq C 3^{2n} \fint_{z+Q_{n+2}} \left| D^2u_{\mathrm{hom}} (x) \right|^2\,dx.
\end{equation*}
Using this and 
\begin{equation*} \label{}
\overline L(p) \geq \overline L(p(z)) + q(z) \cdot (p-p(z)) - C|p-p(z)|^2,
\end{equation*}
we obtain from~\eqref{e.bleck} that
\begin{multline*} \label{}
\fint_{z+Q_n}  L(D\tilde v(x),x)  \,dx - \fint_{z+Q_{n+2}} \overline L( Du_{\mathrm{hom}}(x)) \, dx \\
\leq C\mathcal E' + C 3^{n-l}\left( K_0^2 + \fint_{z+Q_{n+3}} \left| Du_{\mathrm{hom}}(x) \right|^2\, dx \right) 
+  C3^{n+l}\fint_{z+Q_{n+3}} \left| D^2u_{\mathrm{hom}}(x) \right|^2\,dx.
\end{multline*}
Summing this over all $z\in V\cap 3^n\Zd$ such that $z+Q_{n+2} \subseteq V$ and applying~\eqref{e.meyapp} and~\eqref{e.H2app}, we at last obtain~\eqref{e.tilduhomW}.

\smallskip

\emph{Step 4.} We estimate the contribution of the energy of $\tilde u_{\mathrm{hom}}$ in the boundary strip $U \setminus V^\circ$. The claim is that 
\begin{equation} \label{e.tildeuhombndy}
 \fint_{U \setminus V^\circ} \left|D\tilde u_{\mathrm{hom}}(x)\right|^2\, dx \\
 \leq C\mathcal E' + C3^{n-l} M^2 + C3^{(l-n-m)(1-2/r)} M^2.
\end{equation}
By H\"older's inequality and~\eqref{e.meyapp},
\begin{align} \label{e.Holdbndr2}
\frac{1}{|U|} \int_{U\setminus V} |Du_{\mathrm{hom}}(x)|^2 \,dx & \leq \left(\frac{\left|U\setminus V \right|}{|U|} \right)^{1-2/r} \left( \fint_U   |Du_{\mathrm{hom}}(x)|^r\,dx \right)^{2/r}  \\
& \leq C3^{(l-n-m)(1-2/r)} M^2.\nonumber
\end{align}
By the triangle inequality,~\eqref{e.eta} and the expression
\begin{equation} \label{e.rotcorpse}
D\tilde u_{\mathrm{hom}}(x) = \eta (x) D\tilde v(x) + (1-\eta(x)) Du_{\mathrm{hom}}(x) + D\eta(x) \left(\tilde v(x) - u_{\mathrm{hom}}(x) \right),
\end{equation}
we obtain, for each $x\in V\setminus V^\circ$, 
\begin{equation*} \label{}
\left| D\tilde u_{\mathrm{hom}}(x) \right| \leq  \left|D\tilde v(x) \right|+ \left| Du_{\mathrm{hom}} (x) \right| + C3^{-l} \left| \tilde v(x) - u_{\mathrm{hom}}(x) \right|.
\end{equation*}
By the Poincar\'e inequality,~\eqref{e.affineagree} and~\eqref{e.mesogradwts}, we get, for every $z\in V \cap 3^n\Zd$, 
\begin{multline*} \label{}
3^{-2n}\fint_{z+Q_{n}}  \left| \tilde v(x) - u_{\mathrm{hom}}(x) \right|^2\, dx \\
  \leq C\mathcal E' + C \left( K_0^2 + \fint_{z+Q_{n+3}} \left| Du_{\mathrm{hom}}(x) \right|^2\, dx \right) 
+  C3^{n+l}\fint_{z+Q_{n+3}} \left| D^2u_{\mathrm{hom}}(x) \right|^2\,dx.
\end{multline*}
According to~\eqref{e.vgradbndnu},~\eqref{e.pz} and~\eqref{e.tuhmgvz}, for every $z\in V\cap 3^n\Zd$,
\begin{multline*} \label{}
\fint_{z+Q_{n}} \left| D\tilde v(x) \right|^2\, dx \\
  \leq C\mathcal E' + C \left( K_0^2 + \fint_{z+Q_{n+3}} \left| Du_{\mathrm{hom}}(x) \right|^2\, dx \right) 
+  C3^{n+l}\fint_{z+Q_{n+3}} \left| D^2u_{\mathrm{hom}}(x) \right|^2\,dx.
\end{multline*}
We now obtain~\eqref{e.tildeuhombndy} after summing the previous two inequalities over all $z\in (V\setminus V^\circ)\cap 3^n\Zd$, combining the result with~\eqref{e.rotcorpse} and using~\eqref{e.bdrystrip},~\eqref{e.H2app},~\eqref{e.Holdbndr2} and the fact  that $\tilde u_{\mathrm{hom}} \equiv u_{\mathrm{hom}}$ in $U\setminus V$.

\smallskip

The lemma now follows from~\eqref{e.Holdbndr2},~\eqref{e.tilduhomW} and~\eqref{e.tildeuhombndy}.
\end{proof}

We now complete the proof of Proposition~\ref{p.blackbox}.

\begin{proof}[{Proof of~\eqref{e.BBwts}}]
By Lemmas~\ref{l.heteromod} and~\ref{l.homomod},
\begin{align*}
 \fint_U  L\left(Du(x), x \right)\,dx & \leq  \fint_U  L\left(D\tilde u_{\mathrm{hom}}(x), x \right)\,dx \\
 & \leq  \fint_U \overline L\left(D u_{\mathrm{hom}}(x)) \right)\,dx + C\mathcal E'' \\
 & \leq \fint_U \overline L\left(D \tilde u(x)) \right)\,dx +C\mathcal E'' \\
 & \leq  \fint_U L\left(Du(x), x \right)\,dx + C\mathcal E''
\end{align*}
where  we have set
\begin{equation*} \label{}
\mathcal E'':= \mathcal E'+ C M^2\left( 3^{n-l} + 3^{\beta (l-n-m)} \right).
\end{equation*}
In particular, 
\begin{equation*} \label{}
\left| \fint_U  L\left(Du(x), x \right)\,dx -\fint_U \overline L\left(D u_{\mathrm{hom}}(x)) \right)\,dx  \right|  \leq C\mathcal E'',
\end{equation*}
which verifies part of~\eqref{e.BBwts}. The above string of inequalities also gives
\begin{equation*} \label{}
 \fint_U  L\left(D\tilde u_{\mathrm{hom}}(x), x \right)\,dx \leq \fint_U  L\left(Du(x), x \right)\,dx +  C\mathcal E''
\end{equation*}
and
\begin{equation*} \label{}
\fint_U \overline L\left(D \tilde u(x)) \right)\,dx \leq \fint_U \overline L\left(D u_{\mathrm{hom}}(x)) \right)\,dx + C\mathcal E''.
\end{equation*}
Then uniform convexity (i.e., an argument analogous to that of Lemma~\ref{l.convexL2}) yields 
\begin{equation*} \label{}
\fint_{U} \left| Du(x) - D \tilde u_{\mathrm{hom}}(x) \right|^2\, dx + \fint_{U} \left| Du_{\mathrm{hom}}(x) - D \tilde u(x) \right|^2\, dx \leq C \mathcal E''.
\end{equation*}
The Poincar\'e inequality then gives
\begin{equation*} \label{}
3^{-2(n+m)}\fint_{U} \left(  \left| u(x) -  \tilde u_{\mathrm{hom}}(x) \right|^2 + \left| u_{\mathrm{hom}}(x) -  \tilde u(x) \right|^2 \right) \, dx \leq C \mathcal E''.
\end{equation*}
Recall from the definition of $\tilde u$ that
\begin{equation*} \label{}
\tilde u(x) -u(x) = \eta(x) \left( \xi(x) - u(x) \right)
\end{equation*}
and, according to~\eqref{e.meyapp} and~\eqref{e.upone}, that
\begin{equation*} \label{}
\fint_{U} \eta^2(x)\left( \xi(x) - u(x) \right)^2\,dx \leq \fint_{V} \left( \xi(x) - u(x) \right)^2\,dx  \leq C3^{2n} M^2.
\end{equation*}
Therefore the triangle inequality gives
\begin{equation*} \label{}
3^{-2(n+m)}\fint_{U} \left|  u(x)-u_{\mathrm{hom}}(x)  \right|^2\, dx \leq C \mathcal E'' + CM3^{-2m} \leq C\mathcal E''.
\end{equation*}
This completes the proof of~\eqref{e.BBwts} and therefore of Proposition~\ref{p.blackbox}.
\end{proof}

\begin{remark}
The argument above and those of Lemma~\ref{l.heteromod} and~\ref{l.homomod} contain more information than what is stated in Proposition~\ref{p.blackbox}, namely a quantitative estimate for the difference between mesoscopic spatial averages of the (heterogeneous) energy of $u^\ep$ and $\overline L(Du_{\mathrm{hom}})$, and a quantitative statement concerning the weak convergence of $Du^\ep$ to $Du_{\mathrm{hom}}$, which can be stated in terms of an estimate on~$\| Du^\ep - Du_{\mathrm{hom}} \|_{H^{-1}}$. We leave the formulation of this result and the details of the argument to the reader.
\end{remark}

\bibliographystyle{plain}
\bibliography{energy}

\end{document}